\documentclass[11pt]{amsart}
\usepackage[margin=0.9in]{geometry} 
\usepackage{esint,amsthm,amsmath,amssymb,amsfonts,graphicx,color,comment,enumerate,caption,bbm,bm,hyperref,enumitem,mathrsfs}
\usepackage{mathtools}

\usepackage{tikz,pgf}
\usepackage{pgfplots}
\usetikzlibrary{calc}
\usetikzlibrary{patterns}

\usepackage{cleveref}
\allowdisplaybreaks

\DeclareMathOperator\supp{supp}

\newtheorem{lemma}{Lemma}[section]
\newtheorem{remark}{Remark}[section]
\numberwithin{equation}{section}
\newtheorem{theorem}{Theorem}[section]
\newtheorem{proposition}[theorem]{Proposition}

\newtheorem{corollary}[theorem]{Corollary}

\title{Bilinear Bochner-Riesz Means on the Complex Sphere}

\author[S. Bagchi, Md N. Molla, J. Singh, M. N. Vempati]
{Sayan Bagchi \and Md Nurul Molla \and Joydwip Singh \and Manasa N. Vempati} 

\address[S. Bagchi]{Department of Mathematics and Statistics, Indian Institute of Science Education and Research Kolkata, Mohanpur--741246, West Bengal, India.}
\email{sayan.bagchi@iiserkol.ac.in}

\address[Md N. Molla]{Department of General Sciences, BITS Pilani Dubai Campus, International Academic City, Dubai, 345055, UAE}
\email{nurul@dubai.bits-pilani.ac.in}

\address[J. Singh]{Department of Mathematical Sciences, Indian Institute of Science Education and Research Mohali, Mohali--140306, Punjab, India.}
\email{joydwipsingh@iisermohali.ac.in}

\address[M. N. Vempati]{Department of Mathematics, Louisiana State University, Baton Rouge, 70803, USA.}
\email{nvempati@lsu.edu}



\begin{document}

\begin{abstract}
In this paper, we establish the boundedness of the bilinear Bochner–Riesz means $\mathcal{B}^{\alpha}_R$ on the complex sphere $\mathbb{S}$. More precisely, we prove that $\mathcal{B}^{\alpha}_R$ is bounded from $L^{p_1}(\mathbb{S}) \times L^{p_2}(\mathbb{S}) \to L^p(\mathbb{S})$ where $1/p_1+1/p_2=1/p$ and $1\leq p_1, p_2 \leq \infty$, for an admissible range of exponents, with the required smoothness parameter $\alpha$ described in terms of the topological dimension of $\mathbb{S}$. To facilitate our proof, we establish several analytic estimates, including restriction-type estimates, weighted Plancherel estimates with large power of weights and bilinear weighted Plancherel estimates, which are derived from the ground up in our setting and may be considered of independent interest.
    
\end{abstract}

\maketitle

\tableofcontents

\section{Introduction}
The study of Bochner-Riesz means is a central topic in harmonic analysis, originating from the problem of improving the convergence of Fourier expansions, and is closely connected with several fundamental problems in the subject, including the restriction and Kakeya conjectures. Let us consider the bilinear analogues of the Bochner-Riesz means. For $R>0$ and $\alpha\geq 0$, the bilinear Bochner-Riesz operator is denoted by $B_R^\alpha$ and defined by
 \begin{align*}
    B_R^\alpha(f,g)(x) &= \int_{\mathbb{R}^n} \int_{\mathbb{R}^n} 
    \left( 1 - \frac{|\xi|^2 + |\eta|^2}{R^2} \right)_+^\alpha
   \widehat{f}(\xi) \, \widehat{g}(\eta) \, e^{2\pi i x \cdot (\xi + \eta)} \, d\xi \, d\eta,
 \end{align*}
for $f,g \in \mathcal{S}(\mathbb{R}^n)$, where $(t)_{+} = \max\{t, 0\}$ for $t \in \mathbb{R}$ and $\widehat{f}$ denotes the Fourier transform of $f$. 

Bilinear Bochner-Riesz means arise naturally in connection with summability and convergence questions for products of Fourier series. A central problem is to determine the optimal range of the parameter $\alpha$ for which $B_R^\alpha$ is bounded from $L^{p_1}(\mathbb{R}^n) \times L^{p_2}(\mathbb{R}^n)$ into $L^p(\mathbb{R}^n)$ under the H\"older relation $\frac{1}{p} = \frac{1}{p_1} + \frac{1}{p_2}$ and $1\leq p_1, p_2 \leq \infty$. This question has attracted considerable attention in recent years; see e.g. \cite{Grafakos_Li_Disc_Multiplier_2006, Bernicot_Germain_Bilinear_multilier_narrow_support_2013, Bernicot_Grafakos_Song_Yan_Bilinear_Bochner_Riesz_2015, Jeong_Lee_Vargas_Bilinear_Bochner_Riesz_2018, Liu_Wang_Bilinear_Bochner_Riesz_Non_Banach_2020} and references therein. 

In dimension $n=1$, the problem is essentially resolved in the Banach triangle, namely when $1 \le p_1,p_2,p \le \infty$; see \cite{Bernicot_Grafakos_Song_Yan_Bilinear_Bochner_Riesz_2015, Grafakos_Li_Disc_Multiplier_2006}. In dimensions $n \ge 2$, Bernicot et al. \cite{Bernicot_Grafakos_Song_Yan_Bilinear_Bochner_Riesz_2015} established a number of positive and negative results. Subsequently, Jeong, Lee, and Vargas
\cite{Jeong_Lee_Vargas_Bilinear_Bochner_Riesz_2018} obtained improvements in certain ranges ($p_1, p_2 \geq 2$) by connecting bilinear Bochner-Riesz means with square-function estimates. Further progress in the non-Banach range ($0<p<1$) was obtained by Liu and Wang \cite{Liu_Wang_Bilinear_Bochner_Riesz_Non_Banach_2020}, who lowered the smoothness threshold in some cases. 

We recall below some of the main results from \cite{Bernicot_Grafakos_Song_Yan_Bilinear_Bochner_Riesz_2015} and \cite{Liu_Wang_Bilinear_Bochner_Riesz_Non_Banach_2020} which will serve as the starting point of our discussion.

\begin{theorem}\cite[Proposition 4.10, 4.11]{Bernicot_Grafakos_Song_Yan_Bilinear_Bochner_Riesz_2015}, \cite[Theorem 1.1]{Liu_Wang_Bilinear_Bochner_Riesz_Non_Banach_2020}
\label{Theorem: Euclidean bilinear Bochner-Riesz for Grafakos}
Let $n \geq 2$ and $1\leq p_1, p_2 \leq \infty$ with $1/p = 1/p_1 +1/p_2$. Then the following hold
\begin{align*}
    \|B^{\alpha}_R(f,g)\|_{L^{p}(\mathbb{R}^n)} &\leq C \|f\|_{L^{p_1}(\mathbb{R}^n)} \|g\|_{L^{p_2}(\mathbb{R}^n)} ,
\end{align*}
whenever $p_1, p_2, p$ and $\alpha$ satisfy one of the following conditions:
\begin{enumerate}
    \item (Region I) $2 \leq p_1, p_2 \leq \infty$, $1\leq p \leq 2$ and $\alpha> (n-1)(1-\frac{1}{p})$. 
    \item (Region II) $2 \leq p_1, p_2, p \leq \infty$ and $\alpha> \frac{n-1}{2} + n(\frac{1}{2}-\frac{1}{p})$.
    \item (Region III) $2 \leq p_2 < \infty$, $1\leq p_1, p \leq 2$ and $\alpha> n(\frac{1}{2}-\frac{1}{p_2})-(1-\frac{1}{p})$.
    \item (Region III) $2 \leq p_1 < \infty$, $1\leq p_2, p \leq 2$ and $\alpha> n(\frac{1}{2}-\frac{1}{p_1})-(1-\frac{1}{p})$.
    \item (Region IV) $1\leq p_1 \leq 2 \leq p_2 \leq \infty$, $0<p<1$ and $\alpha> n(\frac{1}{p_1}-\frac{1}{2})$.
    \item (Region IV) $1\leq p_2 \leq 2 \leq p_1 \leq \infty$, $0<p<1$ and $\alpha> n(\frac{1}{p_2}-\frac{1}{2})$.
    \item (Region V) $1\leq p_1 \leq p_2 \leq 2$ and $\alpha> n(\frac{1}{p}-1)-(\frac{1}{p_2}-\frac{1}{2})$.
    \item (Region V) $1\leq p_2 \leq p_1 \leq 2$ and $\alpha> n(\frac{1}{p}-1)-(\frac{1}{p_1}-\frac{1}{2})$. 
\end{enumerate}
\end{theorem}

\vspace{-0.3cm}
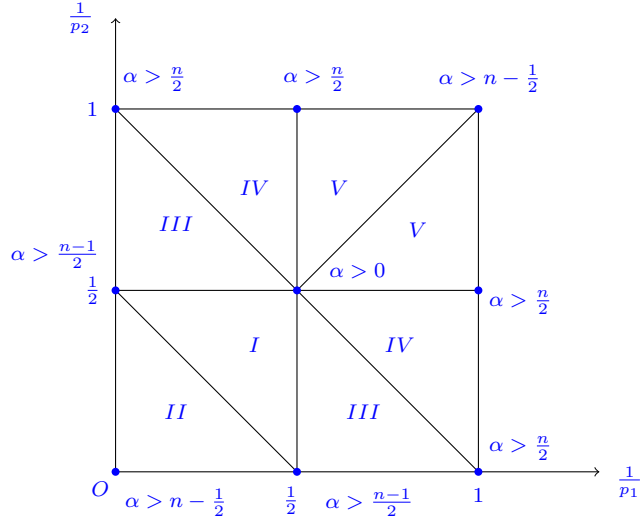
\begin{figure}[!ht]
\begin{centering}
\definecolor{qqqqff}{rgb}{0,0,1}
\begin{tikzpicture}[line cap=round,line join=round,x=0.8cm,y=0.8cm]
\draw (6,6)-- (0,6);
\draw (6,0)-- (6,6);
\draw (0,3)-- (6,3);
\draw [->] (0,0) -- (8,0);
\draw [->] (0,0) -- (0,7.5);
\draw (3,6)-- (3,3);
\draw (0,6)-- (6,0);
\draw (3,3)-- (3,0);
\draw (3,3)-- (6,6);
\draw (0,3)-- (3,0);
\begin{scriptsize}
\fill [color=qqqqff] (0,0) circle (1.5pt);
\draw[color=qqqqff] (1,-.5) node {$\alpha>n-\frac{1}{2}$};
\fill [color=qqqqff] (0,6) circle (1.5pt);
\draw[color=qqqqff] (0.65,6.5) node {$\alpha > \frac{n}{2}$};
\fill [color=qqqqff] (6,6) circle (1.5pt);
\draw[color=qqqqff] (6.19,6.5) node {$\alpha>n-\frac{1}{2}$};
\fill [color=qqqqff] (6,0) circle (1.5pt);
\draw[color=qqqqff] (6.7,0.4) node {$\alpha>\frac{n}{2}$};
\fill [color=qqqqff] (3,0) circle (1.5pt);
\draw[color=qqqqff] (4.2,-0.5) node {$\alpha>\frac{n-1}{2}$};
\fill [color=qqqqff] (3,6) circle (1.5pt);
\draw[color=qqqqff] (3.3,6.5) node {$\alpha>\frac{n}{2}$};
\fill [color=qqqqff] (0,3) circle (1.5pt);
\draw[color=qqqqff] (-1,3.6) node {$\alpha>\frac{n-1}{2}$};
\fill [color=qqqqff] (3,3) circle (1.5pt);
\draw[color=qqqqff] (4,3.33) node {$\alpha>0$};
\fill [color=qqqqff] (6,3) circle (1.5pt);
\draw[color=qqqqff] (6.7,2.8) node {$\alpha>\frac{n}{2}$};
\fill [color=qqqqff] (6,-0.5) circle (0pt);
\draw[color=qqqqff] (6.0,-0.38) node {$1$};
\fill [color=qqqqff] (3,-0.5) circle (0pt);
\draw[color=qqqqff] (2.9,-0.48) node {$\frac{1}{2}$};
\fill [color=qqqqff] (-0.25,-0.42) circle (0pt);
\draw[color=qqqqff] (-0.25,-0.28) node {$O$};
\fill [color=qqqqff] (-0.5,3) circle (0pt);
\draw[color=qqqqff] (-0.38,3.0) node {$\frac{1}{2}$};
\fill [color=qqqqff] (-0.5,6) circle (0pt);
\draw[color=qqqqff] (-0.38,6.0) node {$1$};
\fill [color=qqqqff] (8.5,-0.5) circle (0pt);
\draw[color=qqqqff] (8.5,-0.16) node {$\frac{1}{p_1}$};
\fill [color=qqqqff] (-0.5,8.5) circle (0pt);
\draw[color=qqqqff] (-0.58,7.5) node {$\frac{1}{p_2}$};
\draw[color=qqqqff] (2.3,2.1) node {$I$};
\draw[color=qqqqff] (1.0,1.0) node {$II$};
\draw[color=qqqqff] (4.1,1.0) node {$III$};
\draw[color=qqqqff] (1.0,4.1) node {$III$};
\draw[color=qqqqff] (4.7,2.1) node {$IV$};
\draw[color=qqqqff] (2.3,4.7) node {$IV$};
\draw[color=qqqqff] (3.7,4.7) node {$V$};
\draw[color=qqqqff] (5,4) node {$V$};
\end{scriptsize}
\end{tikzpicture}
        \caption{Here $O=(0,0)$, and $\alpha>\alpha(p_1, p_2)$ represents that $B^{\alpha}_R$ is bounded on $L^{p_1}(\mathbb{R}^n) \times L^{p_2}(\mathbb{R}^n) \to L^p(\mathbb{R}^n)$ for $\alpha>\alpha(p_1, p_2)$ (see Theorem \ref{Theorem: Euclidean bilinear Bochner-Riesz for Grafakos}).}
        \label{Figure: Euclidean picture}
\end{centering}        
\end{figure}

\medskip
Beyond the Euclidean setting, the boundedness of bilinear Bochner-Riesz means has also been investigated recently in several other contexts. For instance, on M\'etivier groups the bilinear Bochner-Riesz means was studied by the first three authors in \cite{Bagchi_Molla_Singh_Bilinear_Metivier_2026}, by the third author in \cite{Singh_Steins_Square_function_2026}; while analogous questions for Grushin operators were considered in \cite{Bagchi_Molla_Singh_Bilinear_Bochner_Riesz_Grushin}. A common theme in these works is to establish the boundedness results for the bilinear Bochner-Riesz means with smoothness thresholds expressed in terms of the \emph{topological dimension} of the corresponding underlying space. We would like to emphasize that the above three results, to the best of our knowledge, are the first occurrence in the literature where the role of the topological dimension appears in a systematic way in bilinear multiplier theory through Bochner-Riesz means.

\medskip
Continuing in this line of research, the present article provides a further manifestation of this phenomenon in a fundamentally different geometric setting. One of the main contributions of the present paper is to obtain an analogue of Theorem \ref{Theorem: Euclidean bilinear Bochner-Riesz for Grafakos} for the bilinear Bochner-Riesz means associated with the sub-Laplacian defined on the complex sphere, where in particular, the smoothness parameter $\alpha$ is possibly described in terms of the \emph{topological dimension} of the underlying complex sphere. On the other hand, it is also important to note that since the complex sphere is conformally equivalent to the Heisenberg group, it has traditionally served as a prototypical example in the analysis of strictly pseudoconvex CR manifolds of hypersurface type, and has been studied by various authors, see \cite{Folland_Complex_sphere_1972, Folland_Stein_Heisenberg_Group_1974, Geller_Kohn_Laplacian_1980}. While this work is motivated by phenomena observed in \cite{Bagchi_Molla_Singh_Bilinear_Metivier_2026,Bagchi_Molla_Singh_Bilinear_Bochner_Riesz_Grushin}, its proofs requires the development of new tools and substantial modification of the methods employed in those earlier works. These changes are necessitated by the geometry and spectral structure of the complex sphere. To clarify, we briefly indicate some of the main difficulties that arise in our setting. The principle obstacle we encounter is the absence of several of the fundamental tools that played a crucial role in the earlier analysis. This includes \emph{restriction-type estimates}, \emph{weighted Plancherel estimates for large power of certain weights}, \emph{bilinear weighted Plancherel estimates} and \emph{inclusion of sub-Riemannian ball into product of Euclidean balls}. In the earlier works \cite{ Bagchi_Molla_Singh_Bilinear_Metivier_2026, Bagchi_Molla_Singh_Bilinear_Bochner_Riesz_Grushin}, such estimates were already available in the literature and could be used effectively. In contrast, in the present context, the corresponding estimates do not appear to be known and must be developed from the ground up. A significant part of the present work is therefore devoted to developing suitable substitutes and establishing new estimates adapted to the complex sphere and may be of independent interest. However, the developing of these tools alone is not sufficient to carry out the previous arguments. The compact nature of the complex sphere alongside the absence of relevant group-invariant structure, the continuous spectrum and scale-invariant features, introduce additional structural barriers. A thorough analysis of these obstructions are given after Theorem \ref{Theorem: Bilinear near diagonal case}. Before stating our main results, we recall some preliminary facts concerning the complex sphere.

\medskip

For $n \geq 2$, let $\mathbb{S} := S^{2n-1} = \{ (z_1, \ldots ,z_n) \in \mathbb{C}^n : |z_1|^2 + \cdots + |z_n|^2 =1 \}$ denote the complex unit sphere. Denote $d=2n-1$ to be the topological dimension and $Q=2n$ to be the homogeneous dimension of $\mathbb{S}$. For $f \in C^{\infty}(\mathbb{S})$, define 
\begin{align*}
     Tf(z) := \frac{d}{d \theta}f(e^{i \theta} z)|_{\theta=0} .
\end{align*}
Then the sub-Laplacian $\mathcal{L}$ is defined by
\begin{align*}
    \mathcal{L} &= \Delta + T^2 ,
\end{align*}
where $\Delta$ denotes the minus of the Laplace-Beltrami operator on $\mathbb{S}$.

The sub-Laplacian $\mathcal{L}$ is non-negative, essentially self-adjoint operator on $C^{\infty}(\mathbb{S})$ with respect to the uniform measure $\sigma$ of $\mathbb{S}$ and invariant under the action of the unitary group $\mathbb{U}(n)$.

For $\ell, \ell' \in \mathbb{N}_0$, let $\mathcal{H}_{\ell, \ell'}(\mathbb{S})$ denote the restriction of homogeneous harmonic polynomials of bidegree $(\ell, \ell')$ in $\mathbb{C}^n$ to the unit sphere $\mathbb{S}$. We say a function $f$ is homogeneous of bidegree $(\ell,\ell')$, if
\begin{align*}
    f(\lambda z) = \lambda^{\ell} \Bar{\lambda}^{\ell'} f(z)  \quad \text{for all} \quad \lambda \in \mathbb{C}\setminus \{0\} \quad \text{and} \quad z \in \mathbb{C}^n .
\end{align*}
The spaces $\mathcal{H}_{\ell, \ell'}(\mathbb{S})$, often denoted by $\mathcal{H}_{\ell, \ell'}$, are called the spaces of complex spherical harmonics of bidegree $(\ell, \ell')$. These spaces are eigen spaces of the sub-Laplacian $\mathcal{L}$, in fact, for all $f \in \mathcal{H}_{\ell, \ell'}$ we have
\begin{align*}
    \mathcal{L} f &= \lambda_{\ell, \ell'} f  \quad \quad \text{where} \quad \quad \lambda_{\ell, \ell'} = 4\ell \ell'+2(n-1)(\ell+\ell') .
\end{align*}
Since $\mathcal{L}$ is essentially self-adjoint, using the the spectral theorem for $f \in C^{\infty}(\mathbb{S})$, $R>0$ and $\alpha \geq 0$, the Bochner-Riesz means associated with the sub-Laplacian $\mathcal{L}$ of order $\delta$, denoted by $S^{\delta}_R$ and is defined by
\begin{align*}
    S^{\delta}_R(\sqrt{\mathcal{L}}) f(z) &= \sum_{\ell, \ell'=0}^{\infty} \left(1- \frac{\sqrt{\lambda_{\ell, \ell'}}}{R} \right)_{+}^{\delta} \pi_{\ell, \ell'} f(z) ,
\end{align*}
where $\pi_{\ell, \ell'}$ denote the projection of $L^2(\mathbb{S})$ onto the eigen space $ \mathcal{H}_{\ell, \ell'}$.

\medskip
We would like remark that, $L^p$-boundedness of Bochner-Riesz means $S^{\delta}_R(\sqrt{\mathcal{L}})$ on $\mathbb{S}$ has been studied by Casarino and Peloso \cite{Casarino_Peloso_Riesz_Mean_2011}. They proved that for  $1\leq p<\infty$, whenever $\delta>d|1/p-1/2|$, the operator $S^{\delta}_R(\sqrt{\mathcal{L}})$ is bounded on $L^p(\mathbb{S})$. Here $d=2n-1$ is the topological dimension of $\mathbb{S}$. A notable feature of this result is the required smoothness $\delta$ depends on $d$, rather than the homogeneous dimension $Q=2n$.

\medskip
Next we consider bilinear analogue of the Bochner-Riesz means $S^{\delta}_R(\sqrt{\mathcal{L}})$ on $\mathbb{S}$. For $\alpha \geq 0 $, $R>0$ and $f, g \in C^{\infty}(\mathbb{S})$, the bilinear Bochner-Riesz means associated with the sub-Laplacian $\mathcal{L}$ is given by
\begin{align*}
    \mathcal{B}_R^{\alpha}(f,g)(z) & = \sum_{\substack{\ell_1,\ell_1' =0,\\ \ell_2,\ell_2' =0}}^{\infty} \left(1- \frac{\sqrt{\lambda_{\ell_1, \ell_1'}} + \sqrt{\lambda_{\ell_2, \ell_2'}}}{R} \right)_{+}^{\alpha} \pi_{\ell_1, \ell_1'} f(z) \,  \pi_{\ell_2, \ell_2'} g(z) .
\end{align*}

Analogous to the Euclidean setup, in this paper we are interested in the $L^{p_1}(\mathbb{S}) \times L^{p_2}(\mathbb{S})$ to $L^{p}(\mathbb{S})$ boundedness of $\mathcal{B}_R^{\alpha}$, where $1\leq p_1, p_2 \leq \infty$ and the exponents satisfy the H\"older's relation $1/p=1/p_1 +1/p_2$. In addition, our main objective is to express smoothness threshold in terms of the topological dimension $d=2n-1$ of $\mathbb{S}$. In this direction, our first main result concerning the boundedness of $\mathcal{B}_R^{\alpha}$ is the following.

\begin{theorem}
\label{Theorem: Bilinear Bochner-Riesz Main theorem}
Let $1\leq p_1, p_2 \leq \infty$ with $1/p = 1/p_1 +1/p_2$. Then the following hold
\begin{align*}
    \|\mathcal{B}_R^{\alpha}(f, g)\|_{L^{p}(\mathbb{S})} &\leq C \|f\|_{L^{p_1}(\mathbb{S})} \|g\|_{L^{p_2}(\mathbb{S})} ,
\end{align*}
with some constant $C>0$ independent of $R$, whenever $p_1, p_2, p$ and $\alpha> \alpha(p_1, p_2)$ satisfy one of the following conditions:
\begin{enumerate}
    \item (Region I) $2 \leq p_1, p_2 \leq \infty$, $1\leq p \leq 2$ and $\alpha(p_1, p_2)= (d-1)(1-\frac{1}{p})$.
    \item (Region II) $2 \leq p_1, p_2, p \leq \infty$ and $\alpha(p_1, p_2)= \frac{d-1}{2} + d(\frac{1}{2}-\frac{1}{p})$.
    \item (Region III) $2 \leq p_2 \leq \infty$, $1\leq p_1, p \leq 2$ and $\alpha(p_1, p_2)= Q(\frac{1}{p_1}-\frac{1}{2})+(d-1)(1-\frac{1}{p})$.
    \item (Region III) $2 \leq p_1 \leq \infty$, $1\leq p_2, p \leq 2$ and $\alpha(p_1, p_2)= Q(\frac{1}{p_2}-\frac{1}{2})+(d-1)(1-\frac{1}{p})$.
    \item (Region IV) $1\leq p_1 \leq 2 \leq p_2 \leq \infty$, $0<p\leq 1$ and $\alpha(p_1, p_2)= (d+1)(\frac{1}{p}-1)+Q(\frac{1}{2}-\frac{1}{p_2})$.
    \item (Region IV) $1\leq p_2 \leq 2 \leq p_1 \leq \infty$, $0<p \leq 1$ and $\alpha(p_1, p_2)= (d+1)(\frac{1}{p}-1)+Q(\frac{1}{2}-\frac{1}{p_1})$.
    \item (Region V) $1\leq p_1, p_2 \leq 2$ and $\alpha(p_1, p_2)= (d+1)(\frac{1}{p}-1)$.
\end{enumerate}
\end{theorem}

\begin{figure}[!ht]
\begin{centering}
\definecolor{qqqqff}{rgb}{0,0,1}
\begin{tikzpicture}[line cap=round,line join=round,x=0.7cm,y=0.7cm]
\draw (6,6)-- (0,6);
\draw (6,0)-- (6,6);
\draw (0,3)-- (6,3);
\draw [->] (0,0) -- (8,0);
\draw [->] (0,0) -- (0,7.5);
\draw (3,6)-- (3,3);
\draw (0,6)-- (6,0);
\draw (3,3)-- (3,0);
\draw (0,3)-- (3,0);
\begin{scriptsize}
\fill [color=qqqqff] (0,0) circle (1.5pt);
\draw[color=qqqqff] (1,-.5) node {$\alpha>d-\frac{1}{2}$};
\fill [color=qqqqff] (0,6) circle (1.5pt);
\draw[color=qqqqff] (0.8,6.5) node {$\alpha > \frac{Q}{2}$};
\fill [color=qqqqff] (6,6) circle (1.5pt);
\draw[color=qqqqff] (6.19,6.5) node {$\alpha>d+1$};
\fill [color=qqqqff] (6,0) circle (1.5pt);
\draw[color=qqqqff] (6.9,0.4) node {$\alpha>\frac{Q}{2}$};
\fill [color=qqqqff] (3,0) circle (1.5pt);
\draw[color=qqqqff] (4.2,-0.5) node {$\alpha>\frac{d-1}{2}$};
\fill [color=qqqqff] (3,6) circle (1.5pt);
\draw[color=qqqqff] (3.3,6.5) node {$\alpha>\frac{d+1}{2}$};
\fill [color=qqqqff] (0,3) circle (1.5pt);
\draw[color=qqqqff] (-1,3.7) node {$\alpha>\frac{d-1}{2}$};
\fill [color=qqqqff] (3,3) circle (1.5pt);
\draw[color=qqqqff] (3.7,3.33) node {$\alpha>0$};
\fill [color=qqqqff] (6,3) circle (1.5pt);
\draw[color=qqqqff] (7,2.8) node {$\alpha>\frac{d+1}{2}$};
\fill [color=qqqqff] (6,-0.5) circle (0pt);
\draw[color=qqqqff] (6.0,-0.38) node {$1$};
\fill [color=qqqqff] (3,-0.5) circle (0pt);
\draw[color=qqqqff] (2.9,-0.48) node {$\frac{1}{2}$};
\fill [color=qqqqff] (-0.25,-0.42) circle (0pt);
\draw[color=qqqqff] (-0.25,-0.28) node {$O$};
\fill [color=qqqqff] (-0.5,3) circle (0pt);
\draw[color=qqqqff] (-0.38,3.0) node {$\frac{1}{2}$};
\fill [color=qqqqff] (-0.5,6) circle (0pt);
\draw[color=qqqqff] (-0.38,6.0) node {$1$};
\fill [color=qqqqff] (8.5,-0.5) circle (0pt);
\draw[color=qqqqff] (8.5,-0.16) node {$\frac{1}{p_1}$};
\fill [color=qqqqff] (-0.5,8.5) circle (0pt);
\draw[color=qqqqff] (-0.58,7.5) node {$\frac{1}{p_2}$};
\draw[color=qqqqff] (2.3,2.1) node {$I$};
\draw[color=qqqqff] (1.0,1.0) node {$II$};
\draw[color=qqqqff] (4.1,1.0) node {$III$};
\draw[color=qqqqff] (1.0,4.1) node {$III$};
\draw[color=qqqqff] (4.7,2.1) node {$IV$};
\draw[color=qqqqff] (2.3,4.7) node {$IV$};
\draw[color=qqqqff] (4.7,4.7) node {$V$};
\end{scriptsize}
\end{tikzpicture}
\begin{tikzpicture}[line cap=round,line join=round,x=0.7cm,y=0.7cm]
\draw (6,6)-- (0,6);
\draw (6,0)-- (6,6);
\draw (0,3)-- (6,3);
\draw [->] (0,0) -- (8,0);
\draw [->] (0,0) -- (0,7.5);
\draw (3,6)-- (3,3);
\draw (0,6)-- (6,0);
\draw (3,3)-- (3,0);
\draw (0,3)-- (3,0);
\begin{scriptsize}
\fill [color=qqqqff] (0,0) circle (1.5pt);
\draw[color=qqqqff] (1,-.5) node {$\alpha>d-\frac{1}{2}$};
\fill [color=qqqqff] (0,6) circle (1.5pt);
\draw[color=qqqqff] (0.75,6.5) node {$\alpha > \frac{d}{2}$};
\fill [color=qqqqff] (6,6) circle (1.5pt);
\draw[color=qqqqff] (6.19,6.5) node {$\alpha>d$};
\fill [color=qqqqff] (6,0) circle (1.5pt);
\draw[color=qqqqff] (6.8,0.4) node {$\alpha>\frac{d}{2}$};
\fill [color=qqqqff] (3,0) circle (1.5pt);
\draw[color=qqqqff] (4.2,-0.5) node {$\alpha>\frac{d-1}{2}$};
\fill [color=qqqqff] (3,6) circle (1.5pt);
\draw[color=qqqqff] (3.3,6.5) node {$\alpha>\frac{d}{2}$};
\fill [color=qqqqff] (0,3) circle (1.5pt);
\draw[color=qqqqff] (-1,3.7) node {$\alpha>\frac{d-1}{2}$};
\fill [color=qqqqff] (3,3) circle (1.5pt);
\draw[color=qqqqff] (3.7,3.33) node {$\alpha>0$};
\fill [color=qqqqff] (6,3) circle (1.5pt);
\draw[color=qqqqff] (6.8,2.8) node {$\alpha>\frac{d}{2}$};
\fill [color=qqqqff] (6,-0.5) circle (0pt);
\draw[color=qqqqff] (6.0,-0.38) node {$1$};
\fill [color=qqqqff] (3,-0.5) circle (0pt);
\draw[color=qqqqff] (2.9,-0.48) node {$\frac{1}{2}$};
\fill [color=qqqqff] (-0.25,-0.42) circle (0pt);
\draw[color=qqqqff] (-0.25,-0.28) node {$O$};
\fill [color=qqqqff] (-0.5,3) circle (0pt);
\draw[color=qqqqff] (-0.38,3.0) node {$\frac{1}{2}$};
\fill [color=qqqqff] (-0.5,6) circle (0pt);
\draw[color=qqqqff] (-0.38,6.0) node {$1$};
\fill [color=qqqqff] (8.5,-0.5) circle (0pt);
\draw[color=qqqqff] (8.5,-0.16) node {$\frac{1}{p_1}$};
\fill [color=qqqqff] (-0.5,8.5) circle (0pt);
\draw[color=qqqqff] (-0.58,7.5) node {$\frac{1}{p_2}$};
\draw[color=qqqqff] (2.3,2.1) node {$I$};
\draw[color=qqqqff] (1.0,1.0) node {$II$};
\draw[color=qqqqff] (4.1,1.0) node {$III$};
\draw[color=qqqqff] (1.0,4.1) node {$III$};
\draw[color=qqqqff] (4.7,2.1) node {$IV$};
\draw[color=qqqqff] (2.3,4.7) node {$IV$};
\draw[color=qqqqff] (4.7,4.7) node {$V$};
\end{scriptsize}
\end{tikzpicture}
        \caption{Here $O=(0,0)$, and $\alpha>\alpha(p_1, p_2)$ represents that $\mathcal{B}_R^{\alpha}$ is bounded on $L^{p_1}(\mathbb{S}) \times L^{p_2}(\mathbb{S}) \to L^p(\mathbb{S})$ for $\alpha>\alpha(p_1, p_2)$. Left side picture describes Theorem \ref{Theorem: Bilinear Bochner-Riesz Main theorem}; while the right side picture describes Theorem \ref{Theorem: Bilinear Bochner-Riesz theorem with restricted f and g}.}
        \label{Figure: Metivier group picture}
\end{centering}        
\end{figure}
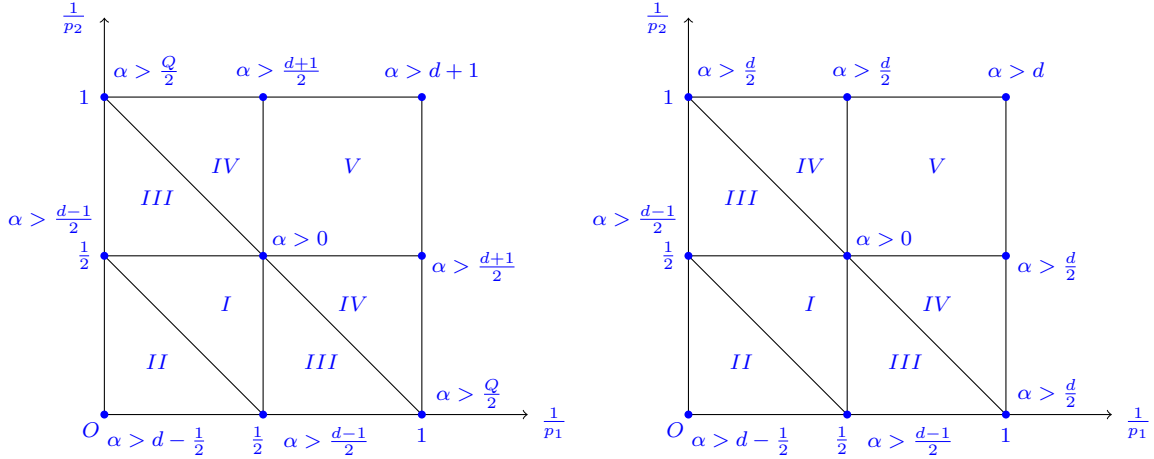

Note our result shows that the topological-dimensional phenomenon observed in the linear theory persists in the bilinear setting too in several admissible regions of both the Banach $(p\geq 1)$ and non-Banach $(0<p<1)$ ranges. We now compare Theorem \ref{Theorem: Bilinear Bochner-Riesz Main theorem} with its Euclidean counterpart, namely Theorem \ref{Theorem: Euclidean bilinear Bochner-Riesz for Grafakos}. One can also compare with \cite{Bagchi_Molla_Singh_Bilinear_Metivier_2026, Bagchi_Molla_Singh_Bilinear_Bochner_Riesz_Grushin}. As suggested by the Euclidean counterpart, the crucial step is to prove the boundedness of $\mathcal{B}^{\alpha}_R$ at certain special points with smoothness parameter expressed in terms of the topological dimension of the complex sphere. One of the main contribution of this paper is to establish the boundedness at those special points with desired smoothness threshold and then the entire range is obtained by symmetry and bilinear interpolation technique as described in \cite[Section 4.3]{Bernicot_Grafakos_Song_Yan_Bilinear_Bochner_Riesz_2015}. More precisely, we establish the boundedness of $\mathcal{B}^{\alpha}_R$ from $L^{p_1}(\mathbb{S}) \times L^{p_2}(\mathbb{S}) \to L^p(\mathbb{S})$ at the points $(p_1, p_2, p)= (2,2,1)$, $(2, \infty, 2)$, $(\infty, \infty, \infty)$, $(1,1,1/2)$, $(1,2,2/3)$ and $(1, \infty, 1)$. For the points $(p_1, p_2, p)= (2,2,1)$, $(2, \infty, 2)$, $(\infty, \infty, \infty)$ the smoothness threshold is given by $0, (d-1)/2$ and $d-1/2$ respectively. These are exact analogue of the Theorem \ref{Theorem: Euclidean bilinear Bochner-Riesz for Grafakos}, where the Euclidean dimension $n$ is replaced by the topological dimension $d=2n-1$ of $\mathbb{S}$. While for the points $(p_1, p_2, p)= (1,1,1/2)$ and $(1,2,2/3)$ our smoothness threshold is $d+1$ and $(d+1)/2$ respectively. Compared to the Euclidean estimates (Theorem \ref{Theorem: Euclidean bilinear Bochner-Riesz for Grafakos}), this entails a loss of $3/2$ and $1/2$ order in the smoothness threshold respectively. This loss is primarily due to the absence of the explicit kernel expression of bilinear Bochner-Riesz means associated with the sub-Laplacian $\mathcal{L}$ in $\mathbb{S}$ as well as our approach in handling these two points. In contrast, at the point $(p_1, p_2, p)= (1, \infty, 1)$ we are only able to prove the boundedness result with smoothness threshold $Q/2$, replacing Euclidean dimension $n$ by $Q$, the homogeneous dimension of $\mathbb{S}$. One can also compare Figure \ref{Figure: Euclidean picture} and Figure \ref{Figure: Metivier group picture}.

\medskip
Nevertheless, we can improve Theorem \ref{Theorem: Bilinear Bochner-Riesz Main theorem} at the points $(p_1, p_2, p)= (1,1,1/2)$, $(1,2,2/3)$ and $(1, \infty, 1)$ provided the input functions $f,g$ are taken from a more restrictive class rather than from the full $L^{p_1}(\mathbb{S})$ and $L^{p_2}(\mathbb{S})$ respectively. In fact, below we state our second main result of this paper, which establishes an exact analogue of Theorem \ref{Theorem: Euclidean bilinear Bochner-Riesz for Grafakos} under the additional assumption on the input functions.

\begin{theorem}
\label{Theorem: Bilinear Bochner-Riesz theorem with restricted f and g}
Fix $R>0$. Define 
\begin{align*}
     \mathfrak{A}_R := \bigoplus\limits_{\substack{\ell_1, \ell_1'=0 \\ \{|\ell_1-\ell_1'| \leq \kappa_1 R\} \cup \{|\ell_1-\ell_1'| \geq \kappa_1' R^2\}}}^{\infty} \mathcal{H}_{\ell_1, \ell_1'} \quad \text{and} \quad \mathfrak{B}_R := \bigoplus\limits_{\substack{\ell_2, \ell_2'=0 \\ \{|\ell_2-\ell_2'| \leq \kappa_2 R\} \cup \{|\ell_2-\ell_2'| \geq \kappa_2' R^2\}}}^{\infty} \mathcal{H}_{\ell_2, \ell_2'} 
\end{align*}
for some $\kappa_1, \kappa_1', \kappa_2, \kappa_2'>0$. Let $1\leq p_1, p_2 \leq \infty$ with $1/p = 1/p_1 +1/p_2$. Then the following bound holds true
\begin{align*}
    \|\mathcal{B}_R^{\alpha}(f,g)\|_{L^p(\mathbb{S})} &\leq C \|f\|_{L^{p_1}(\mathbb{S})} \|g\|_{L^{p_2}(\mathbb{S})} ,
\end{align*}
with some constant $C>0$ independent of $R$, whenever $p_1, p_2, p$ and $\alpha> \alpha(p_1, p_2)$ satisfy one of the following conditions:
\begin{enumerate}
    \item (Region III) $2 \leq p_2 \leq \infty$, $1\leq p_1, p \leq 2$, $\alpha(p_1, p_2)= d(\frac{1}{2}-\frac{1}{p_2})-(1-\frac{1}{p})$ and $g \in \mathfrak{B}_R$.
    \item (Region III) $2 \leq p_1 \leq \infty$, $1\leq p_2, p \leq 2$, $\alpha(p_1, p_2)= d(\frac{1}{2}-\frac{1}{p_1})-(1-\frac{1}{p})$ and $f \in \mathfrak{A}_R$.
    \item (Region IV) $1\leq p_1 \leq 2 \leq p_2 \leq \infty$, $0<p\leq 1$, $\alpha(p_1, p_2)= d(\frac{1}{p_1}-\frac{1}{2})$ and $g \in \mathfrak{B}_R$.
    \item (Region IV) $1\leq p_2 \leq 2 \leq p_1 \leq \infty$, $0<p\leq 1$, $\alpha(p_1, p_2)= d(\frac{1}{p_2}-\frac{1}{2})$ and $f \in \mathfrak{A}_R$.
    \item (Region V) $1\leq p_1, p_2 \leq 2$ and $\alpha(p_1, p_2)= d(\frac{1}{p}-1)$ and $f \in \mathfrak{A}_R$, $g \in \mathfrak{B}_R$ with same parity.
\end{enumerate}
Here same parity means either $f$ and $g$ both comes from $\bigoplus\limits_{\substack{\ell_i, \ell_i'=0 \\ \{|\ell_i-\ell_i'| \leq \kappa_i R\}}}^{\infty} \mathcal{H}_{\ell_i, \ell_i'}$ or from $\bigoplus\limits_{\substack{\ell_i, \ell_i'=0 \\ \{|\ell_i-\ell_i'| \geq \kappa_i' R^2\}}}^{\infty} \mathcal{H}_{\ell_i, \ell_i'}$ for $i=1,2$ respectively.
\end{theorem}

Next we show that the smoothness threshold of Theorem \ref{Theorem: Bilinear Bochner-Riesz Main theorem} can be further improved, if the input functions are considered from the following set
\begin{align*}
     \mathfrak{F}_{R^s} &:= \bigoplus\limits_{\substack{\ell, \ell'=0 \\ \{|\ell-\ell'| \leq \kappa R^s\}}}^{\infty} \mathcal{H}_{\ell, \ell'} \quad \quad \text{for} \quad s \in [0,1] \quad \text{and} \quad \kappa, R>0 .
\end{align*}
The study of the boundedness of $\mathcal{B}_R^{\alpha}(f,g)$ for functions $f, g \in \mathfrak{F}_{R^s}$, can be equivalently reformulated as the study of the following bilinear operator, which may be view as a bilinear Bochner-Riesz means localized near the diagonal. Precisely, for $s \in [0,1]$ and $\kappa, R>0$, we define
\begin{align*}
    \mathcal{B}_{R, s}^{\alpha}(f,g)(z) & = \sum_{\substack{\ell_1,\ell_1', \ell_2,\ell_2' =0\\ |\ell_1-\ell_1'|\leq \kappa R^s, |\ell_2-\ell_2'|\leq \kappa R^s }}^{\infty} \left(1- \frac{\sqrt{\lambda_{\ell_1, \ell_1'}} + \sqrt{\lambda_{\ell_2, \ell_2'}}}{R} \right)_{+}^{\alpha} \pi_{\ell_1, \ell_1'} f(z) \,  \pi_{\ell_2, \ell_2'} g(z) .
\end{align*}

Note that for $0<\kappa<1$, if we take $s=0$, then we obtain \emph{diagonal bilinear Bochner-Riesz means} denoted by $\mathcal{B}_{R, 0}^{\alpha}=: \mathcal{B}_{R, diag}^{\alpha}$ and is defined by
\begin{align*}
    \mathcal{B}_{R, diag}^{\alpha}(f,g)(z) & = \sum_{\ell_1, \ell_2 =0}^{\infty} \left(1- \frac{\sqrt{\lambda_{\ell_1, \ell_1}} + \sqrt{\lambda_{\ell_2, \ell_2}}}{R} \right)_{+}^{\alpha} \pi_{\ell_1, \ell_1} f(z) \,  \pi_{\ell_2, \ell_2} g(z) .
\end{align*}

\medskip
The following result establishes the boundedness of $\mathcal{B}_{R, s}^{\alpha}$ from $L^{p_1}(\mathbb{S})\times L^{p_2}(\mathbb{S})$ to $L^p(\mathbb{S})$. Our result can be seen as an exact analogue of the Euclidean result, Theorem \ref{Theorem: Euclidean bilinear Bochner-Riesz for Grafakos} where Euclidean dimension $n$ is replaced by $(d-1+s)$ except at the region $V$. We also note that in this short-term interaction part, our estimate improves the smoothness threshold compared to both the Euclidean result Theorem \ref{Theorem: Euclidean bilinear Bochner-Riesz for Grafakos} and also Theorem \ref{Theorem: Bilinear Bochner-Riesz Main theorem}, \ref{Theorem: Bilinear Bochner-Riesz theorem with restricted f and g}.

\begin{theorem}
\label{Theorem: Bilinear near diagonal case}
Let $1\leq p_1, p_2 \leq \infty$ with $1/p = 1/p_1 +1/p_2$ and $s\in [0,1]$. Then the following bound holds true
\begin{align*}
    \|\mathcal{B}_{R, s}^{\alpha}(f,g)\|_{L^p(\mathbb{S})} &\leq C \|f\|_{L^{p_1}(\mathbb{S})} \|g\|_{L^{p_2}(\mathbb{S})} ,
\end{align*}
with some constant $C>0$ independent of $R$, whenever $p_1, p_2, p$ and $\alpha> \alpha(p_1, p_2)$ satisfy one of the following conditions:
\begin{enumerate}
    \item (Region I) $2 \leq p_1, p_2 \leq \infty$, $1\leq p \leq 2$ and $\alpha(p_1, p_2)= (d-2+s)(1-\frac{1}{p})$.
    \item (Region II) $2 \leq p_1, p_2, p \leq \infty$ and $\alpha(p_1, p_2)= \frac{d-2+s}{2} + (d-1+s)(\frac{1}{2}-\frac{1}{p})$.
    \item (Region III) $2 \leq p_2 \leq \infty$, $1\leq p_1, p \leq 2$, $\alpha(p_1, p_2)= (d-1+s)(\frac{1}{2}-\frac{1}{p_2})-(1-\frac{1}{p})$.
    \item (Region III) $2 \leq p_1 \leq \infty$, $1\leq p_2, p \leq 2$, $\alpha(p_1, p_2)= (d-1+s)(\frac{1}{2}-\frac{1}{p_1})-(1-\frac{1}{p})$.
    \item (Region IV) $1\leq p_1 \leq 2 \leq p_2 \leq \infty$, $0<p\leq 1$, $\alpha(p_1, p_2)= (d-1+s)(\frac{1}{p_1}-\frac{1}{2})$.
    \item (Region IV) $1\leq p_2 \leq 2 \leq p_1 \leq \infty$, $0<p\leq 1$, $\alpha(p_1, p_2)= (d-1+s)(\frac{1}{p_2}-\frac{1}{2})$.
    \item (Region V) $1\leq p_1, p_2 \leq 2$ and $\alpha(p_1, p_2)= (d-1+s)(\frac{1}{p}-1)$.
\end{enumerate}
    
\end{theorem}

As noted above, the boundedness of bilinear Bochner-Riesz means associated to certain sub-elliptic operators has previously been investigated with smoothness threshold formulated in terms of the topological dimension, see \cite{Bagchi_Molla_Singh_Bilinear_Metivier_2026, Bagchi_Molla_Singh_Bilinear_Bochner_Riesz_Grushin, Singh_Steins_Square_function_2026}. The strategy in this paper is partly inspired by these aforementioned works. However, as indicated earlier, the complex geometry of the underlying space introduces several substantial new difficulties which prevents a direct adaption of the technique developed in previously studied setting. In particular, the absence of the several analytic estimates, group invariance structure on the complex sphere, together with discrete spectrum of the associated sub-Laplacian makes the analysis in this setting particularly delicate and challenging.  We now discuss how these difficulties manifest themselves in the development of several tools required in the present work.

\medskip

One of the key idea in earlier papers are the use of \emph{weighted Plancherel estimates} and \emph{restriction-type estimates} in the bilinear setting. The development of fundamental tools like \emph{weighted Plancherel estimates} (Proposition \ref{Proposition: Linear weighted Plancherel with weight}), \emph{weighted Plancherel estimates} with large power of weights (Proposition \ref{Proposition: Weighted Plancherel for large N case}), bilinear weighted Plancherel estimates (Proposition \ref{Proposition: Bilinear weighted Plancherel with weight}), \emph{restriction-type estimates} (Proposition \ref{Proposition: Truncated restriction type estimate}), as well as containment of sub-Riemannian balls into product of two Euclidean balls (Lemma \ref{Lemma: Ball contained in product of two Euclidean like ball}) in complex sphere setting constitute an important contribution of this work. We would like to emphasize that these estimates are novel in our setup. By contrast, in the earlier works of \cite{Bagchi_Molla_Singh_Bilinear_Metivier_2026, Bagchi_Molla_Singh_Bilinear_Bochner_Riesz_Grushin} analogous tools were already available in the literature. Consequently, a significant portion of this paper is devoted to building up these auxiliary tools and the establishing corresponding estimates, which are of an independent interest in the analysis on the complex sphere. For example, in case of M\'etivier group $(G \cong \mathbb{R}^{d_1+d_2})$ or Grushin operator on $\mathbb{R}^{d_1+d_2}$ setup, one can use the product structure of the underlying space to show that sub-Riemannian ball is contained inside product of two Euclidean balls (see \cite[Lemma 2.1]{Bagchi_Molla_Singh_Bilinear_Metivier_2026} \cite[Lemma 2.1]{Bagchi_Molla_Singh_Bilinear_Bochner_Riesz_Grushin}), however  for complex sphere no such obvious product structure is available, and finding analogous result (Lemma \ref{Lemma: Ball contained in product of two Euclidean like ball}) is bit more delicate and technical. We would like to also mention that one of the key estimate in proving Theorem \ref{Theorem: Bilinear Bochner-Riesz theorem with restricted f and g} is the $(L^1, L^2)$ restriction-type estimates for the operator $F_M(\sqrt{\mathcal{L}}, |T|)$ (see Proposition \ref{Proposition: Truncated restriction type estimate}). Moreover, in order to prove the Proposition \ref{Proposition: Weighted Plancherel for large N case}, we need to analyze the effect of large power of weights over the zonal spherical harmonics. This procedure produces shifts in the zonal spherical harmonic expansion, and one must also control the growth of the coefficients arising from these shifts. This makes the corresponding estimate substantially more involved in the present setup. Altogether, these estimates provide the foundation for our approach and reflect the additional analytic complexity of the complex sphere setting.

\medskip
\textbf{Notation:}
Throughout the paper we use standard notations. For two non-negative numbers $A_1$ and $A_2$, by the expression $A_1\lesssim A_2$, we mean there exists a constant $C>0$ such that $A_1\leq C A_2$. Whenever the implicit constant depends on $\epsilon$, we write $A_1\lesssim_{\epsilon} A_2$. We write $A_1\sim A_2$, when both $A_1\lesssim A_2$ and $A_2\lesssim A_1$ hold. Denote $\mathbb{N} = \{1, 2, \ldots \}$ and $\mathbb{N}_{0} = \{0, 1, 2, \ldots \}$. For a function $F: \mathbb{R}^n \to \mathbb{R}$, denote $\delta_N F(\lambda) := F(N \lambda)$. Let $\omega_{n-1}$ denote the surface measure of the unit sphere in $\mathbb{R}^n$. For any $z \in \mathbb{C}^n$, we write $\Re z$ and $\Im z$ to denote the real and imaginary part of $z$ respectively. We denote $\mathbb{U}(n)$ to be the unitary group. For any $\beta = (\beta_1, \ldots, \beta_n) \in \mathbb{N}_0^n$, denote $|\beta|=\beta_1 +\ldots + \beta_n$. Also we use the notation $e$ to denote the point $(1,0, \ldots, 0)$ of $\mathbb{S}$.

\medskip
\textbf{Organization:}
The organization of the paper is as follows. In the next section we discuss about the sub-Riemannian geometry of the underlying complex sphere, prove some properties of sub-Riemannian balls and weights. Moreover, we also record some projection estimate, properties of zonal spherical harmonics and write bilinear Bochner-Riesz operator as bilinear spectral multipliers. In section \ref{Section: Restriction type estimates} we prove $(L^1, L^2)$ restriction-type estimates for the joint functional calculus of $\sqrt{\mathcal{L}}$ and $|T|$ (Proposition \ref{Proposition: Truncated restriction type estimate}), which is one of our main contribution and also prove some other projection estimates. Section \ref{Section: Weighted Plancherel estimates} is divided into two parts: first we discuss some linear weighted Plancherel estimates and prove one of the key result, that is, weighted Plancherel estimates for the joint functional calculus of $\sqrt{\mathcal{L}}$ and $|T|$ with arbitrary large power of weights (Proposition \ref{Proposition: Weighted Plancherel for large N case}); while in the next subsection we prove bilinear analogue of the weighted Plancherel estimates (Proposition \ref{Proposition: Bilinear weighted Plancherel with weight}). Finally, the last three sections \ref{Section: Proof of first main theorem}, \ref{Section: Proof of second main theorem} and \ref{Section: Proof of third main theorem} are devoted to the proof of our three main results of this paper, Theorem \ref{Theorem: Bilinear Bochner-Riesz Main theorem}, \ref{Theorem: Bilinear Bochner-Riesz theorem with restricted f and g} and \ref{Theorem: Bilinear near diagonal case}.

\section{Analysis on Complex sphere}

In this section we recall some necessary preliminaries and  develop several key tools on the complex sphere that are essential for our subsequent analysis. More precisely, we prove that the sub-Riemannian ball centered at $e$ is contained inside a set which is described in terms of a weight function $\varpi$ and imaginary parts of $\langle z, e \rangle$ (see Lemma \ref{Lemma: Ball contained in product of two Euclidean like ball}). We also establish triangle inequality of the weight $\varpi$ (see Lemma \ref{Lemma: Triangle inequality for weight}). We further show that bilinear Bochner-Riesz means on $\mathbb{S}$ form a particular class of the more general bilinear spectral multiplier associated with the sub-Laplacian on complex sphere. 

\subsection{Sub-Riemannian geometry of the complex sphere}
For $n \geq 2$, the complex sphere $\mathbb{S} \subseteq \mathbb{C}^n$ is smooth real hypersurface of $\mathbb{C}^n$ of dimension $2n-1$. Suppose $\langle \cdot , \cdot \rangle$ denote the usual Hermitian inner product on $\mathbb{C}^n$, that is $\langle z, w \rangle = \sum_{j = 1}^n z_j\, \Bar{w}_j$. Consider $T_z \mathbb{S}$ to denote the real tangent space at each $z \in \mathbb{S}$, given by
\begin{align*}
    T_z \mathbb{S} &:= \{ w \in \mathbb{C}^n : \Re \langle z, w \rangle = 0 \} .
\end{align*}
Let $\mathbb{C} T \mathbb{S} = T \mathbb{S} \otimes_{\mathbb{R}} \mathbb{C}$ denote the complexified tangent bundle of $\mathbb{S}$.
The unit sphere $\mathbb{S}$ can be considered as a model example of CR manifold. Let $\mathbb{C} T \mathbb{C}^n$ denote the complexified tangent bundle of $\mathbb{C}^n$, which decomposes as $\mathbb{C} T \mathbb{C}^n = T_{1,0}\mathbb{C}^n \oplus T_{0,1}\mathbb{C}^n$ into holomorphic and anti-holomorphic components. Since, $\mathbb{S}$ is real hypersurface of $\mathbb{C}^n$, it has a natural CR structure. In fact, $T_{1,0}\mathbb{S} = T_{1,0}\mathbb{C}^n \cap \mathbb{C}T \mathbb{S}$ is the subbundle of $\mathbb{C} T \mathbb{S}$, that defines the CR structure. Moreover, $T_{0,1}\mathbb{S} = T_{0,1}\mathbb{C}^n \cap \mathbb{C}T \mathbb{S}$ with $T_{1,0}\mathbb{S} \cap T_{0,1}\mathbb{S} = \{0\}$. The horizontal distribution $\mathcal{H}\mathbb{S}$ is given by
\begin{align*}
    \mathcal{H}\mathbb{S} &= \Re (T_{1,0}\mathbb{S} \oplus T_{0,1}\mathbb{S}) .
\end{align*}
There is a natural group action of $S^1$ on $\mathbb{S}$, defined by
\begin{align*}
  e^{i\theta}\cdot  (z_1,\cdots , z_n) \to (e^{i \theta} z_1, \cdots , e^{i \theta} z_n).
\end{align*}
For $f \in C^{\infty}(\mathbb{S})$, define 
\begin{align}
\label{Transversal vector field}
     Tf(z) := \frac{d}{d \theta}f(e^{i \theta} z)|_{\theta=0} = i \sum_{j=1}^{n} \left( z_j \frac{\partial f}{\partial z_j } - \Bar{z}_j \frac{\partial f}{\partial \Bar{z}_j } \right).
\end{align}
Then we have
\begin{align*}
    T \mathbb{S} &= \mathcal{H} \mathbb{S} \oplus \mathbb{R} T .
\end{align*}
Let us denote
\begin{align*}
    S_{j} := \frac{\partial}{\partial z_j}- \Bar{z}_j\sum_{k=1}^{n} z_k \frac{\partial}{\partial z_k} \quad \quad \text{for} \quad \quad  j = 1, \ldots, n.
\end{align*}
Then $S_j$ generate the holomorphic tangent space $T_{1,0}\mathbb{S}$ and $\mathbb{C}T \mathbb{S}$ is generated by $S_j$, $\Bar{S}_j$ and $T$. Moreover, the sub-Laplacian $\mathcal{L}$ on $\mathbb{S}$ defined earlier can be also written as
\begin{align*}
    \mathcal{L} = -2\sum_{j=1}^{n} S_j \Bar{S}_j + \Bar{S}_j S_j .
\end{align*}

Recall that, for $\ell, \ell' \in \mathbb{N}_0$, we say a function $f$ is homogeneous of bidegree $(\ell,\ell')$ if $f(\lambda z) = \lambda^{\ell} \Bar{\lambda}^{\ell'} f(z)$ for all $\lambda \in \mathbb{C}\setminus \{0\}$ and $z \in \mathbb{C}^n$. We define $\mathcal{H}_{\ell, \ell'}(\mathbb{C}^n)$ to be the space of homogeneous harmonic polynomials of bidegree $(\ell, \ell')$ in $\mathbb{C}^n$. Since a given homogeneous harmonic function on $\mathbb{S}$ in $\mathbb{C}^n$ can be the restriction of at most one homogeneous harmonic polynomial in $\mathbb{C}^n$, every element of $\mathcal{H}_{\ell, \ell'}(\mathbb{C}^n)$ can be identified with its restriction to $\mathbb{S}$. We denote $\mathcal{H}_{\ell, \ell'}(\mathbb{S})$ to be the restriction of homogeneous harmonic polynomials of bidegree $(\ell, \ell')$ in $\mathbb{C}^n$ to the unit sphere $\mathbb{S}$. In this paper, instead of $\mathcal{H}_{\ell, \ell'}(\mathbb{S})$ we simply write it as $\mathcal{H}_{\ell, \ell'}$. The spaces $\mathcal{H}_{\ell, \ell'}$ are also called the spaces of complex spherical harmonics of bidegree $(\ell, \ell')$. It is known that (see (\cite[Theorem 2.2]{Cowling_Kilima_Sikora_sublaplacian_2011}), these spaces are invariant under the action of $\mathbb{U}(n)$, pairwise orthogonal and moreover
\begin{align*}
    L^2(\mathbb{S}) &= \bigoplus_{\ell, \ell' = 0}^{\infty} \mathcal{H}_{\ell, \ell'} .
\end{align*}
The spaces $\mathcal{H}_{\ell, \ell'}$ are finite dimensional, say $\dim(\mathcal{H}_{\ell,\ell'}) =: d(\ell, \ell')$. Let us fix an orthonormal basis $\{u_1, \ldots, u_{d(\ell, \ell')}\}$ of $ \mathcal{H}_{\ell, \ell'}$. If we denote $\pi_{\ell, \ell'}$ to be the projection of $L^2(\mathbb{S})$ onto the space $ \mathcal{H}_{\ell, \ell'}$, then we can write
\begin{align}
\label{Writing projection in terms of orthonormal basis}
    \pi_{\ell, \ell'} f(z) &= \sum_{i=1}^{d(\ell, \ell')} \langle f, u_i \rangle u_i(z) .
\end{align}
It is also known that the operators $\mathcal{L}$ and $T$ maps every $\mathcal{H}_{\ell, \ell'}$ to itself, that is, for all $f \in \mathcal{H}_{\ell, \ell'}$
\begin{align*}
    \mathcal{L} f &= \lambda_{\ell, \ell'} f  \quad \quad \text{and} \quad \quad T f = i (\ell-\ell') f \quad \text{where} \quad \lambda_{\ell, \ell'} = 4\ell \ell'+2(n-1)(\ell+\ell') .
\end{align*}

Since $\mathcal{L}$ is essentially self-adjoint, using the the spectral theorem one can define the functional calculus of $\mathcal{L}$. For any bounded Borel function $F : \mathbb{R} \to \mathbb{C}$, the spectral multiplier operator $F(\mathcal{L})$ is given by
\begin{align}
\label{Definition of spectral multiplier}
    F(\mathcal{L})f(z) &=  \sum_{\ell, \ell'=0}^{\infty} F(\lambda_{\ell, \ell'}) \, \pi_{\ell, \ell'} f(z) .
\end{align}
Moreover, $\mathcal{L}$ and $T$ commutes strongly, hence they admit a joint functional calculus (see \cite{Martini_Sharp_Multiplier_Kohn_Laplacian_2017}). We also define the operator $|T|$ by
\begin{align*}
    |T| f &:= |\ell-\ell'| f \quad \quad \text{for all} \quad f \in \mathcal{H}_{\ell, \ell'} .
\end{align*}
Consequently, $\mathcal{L}$ and $|T|$ also admit joint spectral decomposition with respect to the complex spherical harmonics of bidegree $(\ell, \ell')$.

\subsection{The sub-Riemannian distance, weight and balls}
Let $\Tilde{\varrho}$ denote the sub-Riemannian control distance on $\mathbb{S}$, that for all $z, w \in \mathbb{S}$, $\Tilde{\varrho}(z,w)$ is defined to be the infimum of length of all horizontal curves joining $z$ and $w$. Since sometimes it is difficult to work with the sub-Riemannian control distance $\Tilde{\varrho}$, it is very convenient to work with some equivalent definition. It is known that for all $z, w \in \mathbb{S}$, the sub-Riemannian control distance $\Tilde{\varrho}$ is equivalent to $\varrho$, where
\begin{align*}
    \varrho(z,w) &:= |1-\langle z, w \rangle|^{1/2} .
\end{align*}
Note that $\varrho$ is $\mathbb{U}(n)$ invariant. With respect to this distance $\varrho$ and $R>0$, we define $B(z,R)$ to be the ball centered at $z \in \mathbb{S}$ and of radius $R$, that is
\begin{align*}
    B(z,R) &= \{w \in \mathbb{S} : \varrho(z,w) < R \} .
\end{align*}
Let $\sigma(B(z, R))$ denote the measure of the ball $B(z, R)$ with respect to $\sigma$. Then (see \cite{Martini_Sharp_Multiplier_Kohn_Laplacian_2017})
\begin{align*}
    \sigma(B(z, R)) \sim \min \{1, R^{Q} \},
\end{align*}
where $Q = 2n$ is called the homogeneous dimension of the complex sphere, in the sense that there exists a constant $C>0$ such that for all $z\in \mathbb{S}$ and $\kappa, R \in (0,\infty)$,
\begin{align*}
    \sigma(B(z, \kappa R)) &\leq C (1+\kappa)^Q \,  \sigma(B(z, R)) .
\end{align*}
Then $(\mathbb{S}, \varrho, \sigma)$ become a doubling metric measure space with homogeneous dimension $Q=2n$.

Let us also define the weight function $\varpi : \mathbb{S} \times \mathbb{S} \to \mathbb{C}$ given by
\begin{align*}
    \varpi(z,w) &:= |1-|\langle z,w \rangle|^2|^{1/2} .
\end{align*}

We next record some important properties of the weight and distance functions, which will be used later in our proofs.

\begin{lemma}\cite[p. 3333]{Casarino_Cowling_Martini_Sikora_Kohn_Laplacian_Forms_2017}
\label{Lemma: Integral of weight}
Let $0\leq \gamma<1$ and $R>0$. Then for $z \in \mathbb{S}$ we have
\begin{align*}
    \int_{B(z, R)} \frac{d\sigma(w)}{\varpi(z,w)^{\gamma}} &\leq C \min\{R^{Q-\gamma}, 1\} .
\end{align*}
    
\end{lemma}

\begin{lemma}\cite[Lemma 4.1]{Martini_Sharp_Multiplier_Kohn_Laplacian_2017}
\label{Lemma: Integral of distance and weight}
Let $\beta, \gamma \geq 0$ such that $\beta+\gamma>Q$ and $\gamma<d-1$. Then for $z \in \mathbb{S}$ and $R>0$ we have
\begin{align*}
    \int_{\mathbb{S}} \frac{d\sigma(w)}{(1+R\, \varrho(z,w))^{\beta} (1+R\, \varpi(z,w))^{\gamma}} &\leq C \min\{1, R^{-Q}\} .
\end{align*}
    
\end{lemma}

Now we state two very crucial lemmas of this paper, which will be important in the proof of Theorem \ref{Theorem: Bilinear Bochner-Riesz theorem with restricted f and g} and form our central contribution for this section. Our first result stated below, discusses a special feature of the sub-Riemannian ball centered at $e$. This can be seen as an analogue of \cite[Lemma 2.1]{Bagchi_Molla_Singh_Bilinear_Metivier_2026} and \cite[Lemma 2.1]{Bagchi_Molla_Singh_Bilinear_Bochner_Riesz_Grushin}.

\begin{lemma}
\label{Lemma: Ball contained in product of two Euclidean like ball}
For any $R>0$, there exists some constant $C>0$ such that
\begin{align*}
    B(e, R) \subseteq B^{\varpi, \Im}(e, C (R, R^2)) ,
\end{align*}
where
\begin{align*}
    B^{\varpi, \Im}(e, C(R, R^2)) &:= \{z \in \mathbb{S} : \varpi(z,e) \leq C R, \  |\Im \langle z, e \rangle| \leq C R^2\} .
\end{align*}
    
\end{lemma}

\begin{proof}
Let us write $z=(z_1, z')$ where $z' = (z_2, \ldots, z_n)$. Therefore, we have
\begin{align*}
    \varrho(z,e) = |1- z_1 |^{1/2} &\sim |1-\Re z_1|^{1/2} + |\Im z_1|^{1/2} \\
    &\sim |1-(\Re z_1)^2|^{1/2} + |\Im z_1|^{1/2} \\
    &\sim |z'| + |\Im z_1|^{1/2} .
\end{align*}
Since $\varpi(z,e)=(1-|z_1|^2)^{1/2}=|z'|$, from the above observation, we obtain
\begin{align*}
    \varrho(z,e) &\sim \varpi(z,e) + |\Im \langle z, e \rangle|^{1/2} .
\end{align*}
Consequently, in view of the above estimate one easily get the lemma.  
\end{proof}

Next we establish triangle inequality for the weight function, which serves an important tool for our proofs later.
\begin{lemma}
\label{Lemma: Triangle inequality for weight}
For all $w,z,x \in \mathbb{S}$, we have
\begin{align*}
    \varpi(w,z) \leq \varpi(w,x) + \varpi(x,z) . 
\end{align*}
\end{lemma}
\begin{proof}
First note that since $\varpi(w,x)=|1-|\langle w, x \rangle|^2|^{1/2} $ is $\mathbb{U}(n)$ invariant, we can assume $x=e$ and hence it is enough to show that
\begin{align}
\label{Required to prove in triangle inequality}
    \varpi(w,z) \leq  \varpi(w,e) + \varpi(e, z) . 
\end{align}
Let us write $w=(w_1, w^{\prime})$ and $z=(z_1, z^{\prime})$. Set
\begin{align*}
    b_1:= |w'| = \sqrt{1-|w_1|^2} = \varpi(w,e) \quad \text{and} \quad b_2:= |z'| = \sqrt{1-|z_1|^2} = \varpi(e, z) .
\end{align*}
Then we write
\begin{align*}
    \langle w, z \rangle &= w_1 \Bar{z}_1 + \langle w^{\prime}, z^{\prime} \rangle .
\end{align*}
Note that Cauchy-Schwarz inequality implies $|\langle w^{\prime}, z^{\prime} \rangle| \leq |w^{\prime}| |z^{\prime}|$. Therefore application of reverse triangle inequality yields
\begin{align*}
    |\langle w, z \rangle| &\geq |w_1| |z_1| - |\langle w^{\prime}, z^{\prime} \rangle| \geq \sqrt{1-b_1^2} \sqrt{1-b_2^2} - b_1 b_2 .
\end{align*}
Consequently, we obtain
\begin{align}
\label{Firstninequality for 1 minus zw}
    1-|\langle w, z \rangle|^2 &\leq 1- \left(\sqrt{1-b_1^2} \sqrt{1-b_2^2} - b_1 b_2 \right)^2 .
\end{align}
Now using the following identity
\begin{align*}
    \left(\sqrt{1-b_1^2} \sqrt{1-b_2^2} - b_1 b_2 \right)^2 &= 1- (b_1+b_2)^2 + 2 b_1^2 b_2^2 + 2b_1 b_2 \left(1-\sqrt{1-b_1^2} \sqrt{1-b_2^2}\right) ,
\end{align*}
we can easily see that
\begin{align}
\label{Consequence of the identity}
    \left(\sqrt{1-b_1^2} \sqrt{1-b_2^2} - b_1 b_2 \right)^2 \geq 1- (b_1+b_2)^2 .
\end{align}
Therefore in view of \eqref{Firstninequality for 1 minus zw} and \eqref{Consequence of the identity} we get
\begin{align*}
    1-|\langle w, z \rangle|^2 &\leq (b_1+b_2)^2 = \big( \varpi(w,e) + \varpi(e, z) \big)^2 .
\end{align*}
Finally, taking square root in the above inequality we obtain the required estimate \eqref{Required to prove in triangle inequality}.
\end{proof}

\subsection{Projection estimates and zonal spherical harmonics} Recall that $\pi_{\ell, \ell'}$ denote the spectral projection of $L^2(\mathbb{S})$ onto the subspace $ \mathcal{H}_{\ell, \ell'}$. We have the following $(L^1, L^2)$-estimate of the spectral projection $\pi_{\ell, \ell'}$.

\begin{lemma}
\label{Lemma: L1 to l2 projection estimate}
Let $n \geq 2$. Then we have
\begin{align*}
    \|\pi_{\ell, \ell'}f\|_{L^{2}} &\leq C \lambda_{\ell, \ell'}^{(n-2)/2} (\ell+\ell')^{1/2} \, \|f\|_{L^1} .
\end{align*}
    
\end{lemma}

\begin{proof}
The proof is an immediate consequence of Theorem 3.1 of \cite{Casarino_Peloso_Riesz_Mean_2011} at $p=1$ together with the subsequent remark in \cite{Casarino_Peloso_Riesz_Mean_2011} concerning the relationship between $q_{\ell, \ell'}$ and $ Q_{\ell, \ell'}$ with the eigenvalue $\lambda_{\ell, \ell'}$ of $\mathcal{L}$ and $(\ell +\ell')$ respectively. In particular, we have
\begin{align*}
    q_{\ell, \ell'} := \min\{\ell, \ell'\} \lesssim \frac{\lambda_{\ell, \ell'}}{\ell+\ell'} \quad \text{and} \quad Q_{\ell, \ell'} := \max\{\ell, \ell'\} \lesssim (\ell+\ell') .
\end{align*}
In view of the above observation the lemma follows from \cite[Theorem 3.1]{Casarino_Peloso_Riesz_Mean_2011} at $p=1$.
\end{proof}

Note that, the space $\mathcal{H}_{\ell, \ell'}$ are finite-dimensional subspace of $L^2(\mathbb{S})$, and for fixed $w \in \mathbb{S}$, the linear functional $f \mapsto f(w)$ is bounded on $\mathcal{H}_{\ell, \ell'}$. Then by Riesz representation theorem, there is a unique function $Y_w^{\ell, \ell'}$ in $\mathcal{H}_{\ell, \ell'}$ such that 
\begin{align*}
    f(w) = \langle f, Y_w^{\ell, \ell'} \rangle \quad \quad \text{for all} \quad \quad f \in \mathcal{H}_{\ell, \ell'} .
\end{align*}
The function $Y_w^{\ell, \ell'}$ are called the zonal spherical harmonics of bidegree $(\ell, \ell')$ with pole $w$.

In the following we need some properties of the zonal spherical harmonics $Y_w^{\ell, \ell'}$. These estimates are well known in the literature. From \cite[Proposition 2.8]{Cowling_Kilima_Sikora_sublaplacian_2011} we have
\begin{align}
\label{L2 norm of zonal harmonics}
    \|Y_w^{\ell, \ell'}\|_{L^2}^2 &= \omega_{2n-1}^{-1} d(\ell, \ell') , \quad \quad \text{for all} \quad w \in \mathbb{S} ,
\end{align}
and
\begin{align*}
    Y_w^{\ell, \ell'}(z) &= \overline{Y_z^{\ell,\ell'}(w)} \quad \quad \text{for all} \quad z, w \in \mathbb{S} .
\end{align*}
Moreover, let $a=\frac{n-1}{2}$. Then in view of \cite[Corollary 2.6]{Cowling_Kilima_Sikora_sublaplacian_2011} one can see that
\begin{align}
\label{Estimate of the L2 norm}
    \|Y_w^{\ell, \ell'}\|_{L^2}^2 & \leq C~ \frac{\ell+a+\ell'+a}{n-1} (\ell+a)^{n-2} (\ell'+a)^{n-2} .
\end{align}
The following result describes the effect of multiplication by the factor $|\langle z, w \rangle|^2$ on zonal spherical harmonics. In fact, from \cite[Corollary 3.3]{Cowling_Kilima_Sikora_sublaplacian_2011}, for $\ell, \ell' \in \mathbb{N}_0$, we have
\begin{align}
\label{Shifts of zonal spherical harmonics}
    |\langle \cdot, w \rangle|^2 Y_w^{\ell,\ell'}= \alpha_{\ell,\ell'} Y_w^{\ell+1,\ell'+1}+\beta_{\ell,\ell'} Y_w^{\ell,\ell'}+\gamma_{\ell,\ell'} Y_w^{\ell-1,\ell'-1} ,
\end{align}
where
\begin{align}
\label{Definition of alpha}
    \alpha_{\ell,\ell'}& = \frac{\ell'+1}{n+\ell+\ell'}\cdot \frac{\ell+1}{n+\ell+\ell'+1} ,
\end{align}
\begin{align}
\label{Definition of beta}
    \beta_{\ell,\ell'} &= \left\{\begin{array}{ll}
        1/2, & \quad \text{if} \ n=2, \ \ell=\ell'=0 \\
        \frac{\ell^2 +\ell'^2+(n-1)(\ell+\ell')+n-2}{(n+\ell+\ell'-1)^2-1}, & \quad \text{otherwise} ,
    \end{array}\right.
\end{align}
\begin{align}
\label{Definition of gamma}
    \gamma_{\ell,\ell'} &=\left\{\begin{array}{ll}
        0 & \quad \text{if either}\quad \ell=0 \ \text{or}\ \ell'=0 \\
        \frac{n+\ell-2}{n+\ell+\ell'-2}\cdot \frac{n+\ell'-2}{n+\ell+\ell'-3} & \quad \text{otherwise} .
    \end{array} \right.
\end{align}
Recall that $\{u_1, \ldots, u_{d(\ell, \ell')}\}$ is an orthonormal basis of  $\mathcal{H}_{\ell, \ell'}$. Then from \cite[eq (15)]{Cowling_Kilima_Sikora_sublaplacian_2011} we have
\begin{align*}
    Y_{w}^{\ell, \ell'}(z) &= \sum_{i=1}^{d(\ell, \ell')} \overline{u_i(w)} u_i(z) .
\end{align*}
Consequently, in view of \eqref{Writing projection in terms of orthonormal basis} we obtain
\begin{align}
\label{Writing projection as inner product}
    \pi_{\ell, \ell'} f(z) &= \langle f, Y_{w}^{\ell, \ell'} \rangle .
\end{align}
For $F \in C_c(\mathbb{R})$, let $\mathcal{K}_{F(\mathcal{L})}(z,w)$ denote the integral kernel of the spectral multiplier $F(\mathcal{L})$, that is,
\begin{align*}
    F(\mathcal{L})f(z) &= \int_{\mathbb{S}} \mathcal{K}_{F(\mathcal{L})}(z, w) f(w) \, d\sigma(w) ,
\end{align*}
where the integral kernel is given by (see \cite[Corollary 4.2]{Cowling_Kilima_Sikora_sublaplacian_2011})
\begin{align}
\label{Linear kernel of spectral multiplier}
    \mathcal{K}_{F(\mathcal{L})}(z, w) &= \sum_{\ell, \ell' = 0}^{\infty} F(\lambda_{\ell, \ell'}) Y_{w}^{\ell, \ell'}(z) ,
\end{align}
Let $\mathcal{K}_{\exp(-t\mathcal{L})}(z,w)$ denote the heat kernel of $\mathcal{L}$. Then it is known that $\mathcal{K}_{\exp(-t\mathcal{L})}(z,w)$ satisfies the Gaussian-type heat kernel bounds (see \cite[Proposition 3.1]{Martini_Sharp_Multiplier_Kohn_Laplacian_2017}), that is, there exists $b \in (0, \infty)$ such that for all $t \in (0,\infty)$ and $z,w \in \mathbb{S}$,
\begin{align}
\label{Gaussian type heat kernel bound}
    \left| \mathcal{K}_{\exp(-t\mathcal{L})} (z,w) \right| &\leq C \sigma(B(w, t^{1/2}))^{-1} \exp(-b \varrho(z,w)^2/t) .
\end{align}

\medskip
\subsection{Bilinear spectral multipliers}
Let us set $\mathcal{L}_1 := \mathcal{L} \otimes I$ and $\mathcal{L}_2 := I \otimes \mathcal{L}$. It is known that, the operators $\mathcal{L}_1$ and $\mathcal{L}_2$ commutes strongly (see \cite[Lemma 7.24]{Konrad_Unbounded_Selfadjoint_operator_2012}). Let $m : [0, \infty)^2 \to \mathbb{C}$ be a bounded Borel function, then the bivariate spectral theorem (\cite[Theorem 5.21]{Konrad_Unbounded_Selfadjoint_operator_2012}) allows us to consider 
\begin{align}
\label{Bivariate bochner-Riesz multiplier}
    & m(\mathcal{L}_1,\mathcal{L}_2)(f \otimes g)(z, z')
    = \sum_{\substack{\ell_1,\ell_1', \ell_2,\ell_2' =0}}^{\infty} m(\lambda_{\ell_1,\ell_1'}, \lambda_{\ell_2,\ell_2'}) \pi_{\ell_1, \ell_1'} f(z) \,  \pi_{\ell_2, \ell_2'} g(z') .
\end{align}

If we take $f, g$ from some nice class of functions (e.g. finite linear combinations of $\mathcal{H}_{\ell, \ell'}$), then one can see that the above expression is well-defined everywhere in $\mathbb{S} \times \mathbb{S}$, in fact, $m(\mathcal{L}_1,\mathcal{L}_2)(f \otimes g)(z,z')$ is also continuous on $\mathbb{S} \times \mathbb{S}$ (see also \cite{Bui_Bilinear_Multiplier_2026}). Consequently, the bilinear spectral multiplier is defined by
\begin{align*}
    \mathcal{B}_m (f,g)(z) &:= m(\mathcal{L}_1,\mathcal{L}_2)(f \otimes g)(z,z) .
\end{align*}

Let $\mathcal{K}_{m(\mathcal{L}_1,\mathcal{L}_2)}^{bi}$ denote the integral kernel corresponding to the bilinear spectral multiplier $\mathcal{B}_m$, then it is given by
\begin{align}
\label{Definition of bilinear integral kernel}
    \mathcal{B}_m (f,g)(z) &= \int_{\mathbb{S}} \int_{\mathbb{S}} \mathcal{K}_{m(\mathcal{L}_1,\mathcal{L}_2)}^{bi}(z, w, u) f(w) g(u) \, d\sigma(w) \, d\sigma(u) ,
\end{align}
where
\begin{align}
\label{Definition of bilinear kernel}
    \mathcal{K}_{m(\mathcal{L}_1,\mathcal{L}_2)}^{bi}(z,w,u) & = \sum_{\substack{\ell_1,\ell_1', \ell_2,\ell_2' =0}}^{\infty} m(\lambda_{\ell_1,\ell_1'}, \lambda_{\ell_2,\ell_2'}) \overline{Y_z^{\ell_1,\ell_1'}(w)}\,  \overline{Y_z^{\ell_2,\ell_2'}(u)} .
\end{align}
In particular, for $\alpha \geq 0$ and $R>0$, let $m^{\alpha} : [0, \infty)^2 \to \mathbb{C}$ be such that
\begin{align*}
    m^{\alpha}(\eta_1, \eta_2) &= \left(1-\frac{\sqrt{\eta_1}+\sqrt{\eta_2}}{R} \right)_{+}^{\alpha} ,
\end{align*}
then the bilinear Bochner-Riesz means $\mathcal{B}_R^{\alpha}$ associated with the sub-Laplacian $\mathcal{L}$ is given by
\begin{align*}
    \mathcal{B}_R^{\alpha}(f,g)(z) &= \mathcal{B}_{m^{\alpha}}(f, g)(z) .
\end{align*}

\section{Restriction type estimates}
\label{Section: Restriction type estimates}
In this section we prove $(L^1, L^2)$ restriction-type estimates for the sub-Laplacian $\mathcal{L}$ on $\mathbb{S}$. This estimates constitutes one of the main tools in the proof of Theorem \ref{Theorem: Bilinear Bochner-Riesz theorem with restricted f and g}, where they are used to lower the required smoothness parameter from the homogeneous dimension $Q$ to the topological dimension $d$ of the underlying space.

We begin by introducing the following discrete norm: for all $N \in \mathbb{N}$ and $F: \mathbb{R} \to \mathbb{C}$ supported in $[0,1]$, the norm $\|F\|_{N,2}$ is defined by
\begin{align*}
    \|F\|_{N,2} = \left(\frac{1}{N} \sum_{i=1}^{N} \sup_{\lambda \in [\frac{i-1}{N}, \frac{i}{N}]} |F(\lambda)|^2 \right)^{1/2}.
\end{align*}
Let us recall now few estimates for this discrete norm. From \cite[eq. (3.1), (3.2)]{Molla_Singh_Commutator_Metivier_Arxiv} we have
\begin{align}
\label{Discrete norm dominated by sup norm}
    \|F\|_{N, 2} &\leq \|F\|_{L^{\infty}} ,
\end{align}
 and
\begin{align}
\label{Cowling Sikora norm relation}
    \|F\|_{L^2} \leq \|F\|_{N,2} \leq C_s \left(\|F\|_{L^2} + N^{-s} \|F\|_{L^2_s} \right) \quad \quad \text{for} \quad s>1/2 .
\end{align}
Suppose $\Theta : \mathbb{R} \to \mathbb{C}$ be a bump function supported in $[N^2/2, 2 N^2]$ for $N \in \mathbb{N}$, such that 
\begin{align}
\label{Definition of Theta}
    \sum_{M \in \mathbb{Z}} \Theta_M(\tau) =1 \quad \text{where}\quad \Theta_M(\tau) = \Theta(2^{M} \tau).
\end{align}
Let $F : \mathbb{R} \to \mathbb{C}$ be a bounded Borel function supported in $[0, N]$. Then for $M \in \mathbb{Z}$, we define $F_M : \mathbb{R} \times \mathbb{R} \to \mathbb{C}$ by 
\begin{align}
\label{Definition of truncated function}
    F_M(\eta, \tau) = \left\{\begin{array}{ll}
        F(\sqrt{\eta}) \Theta_{M}(\tau),  & \quad \eta>0  \\
        0 & \quad \text{otherwise}
    \end{array} \right.  .
\end{align}
Consequently, the joint functional calculus of $\sqrt{\mathcal{L}}$ and $|T|$ gives
\begin{align}
\label{Definition for the truncated operator}
    F_M(\sqrt{\mathcal{L}}, |T|)f(z) &= \sum_{\substack{\ell, \ell'=0 \\ \ell \neq \ell'}}^{\infty} F_M(\sqrt{\lambda_{\ell, \ell'}}, |\ell-\ell'|)\, \pi_{\ell, \ell'}f(z) .
\end{align}
Note that since $F_M(\sqrt{\lambda_{\ell, \ell'}}, |\ell-\ell'|) = F(\sqrt{\lambda_{\ell, \ell'}}) \Theta(2^M |\ell-\ell'|)$ and $\Theta(0)=0$, hence without loss of generality, we can assume $\ell \neq \ell'$, in the above expression \eqref{Definition for the truncated operator} of $F_M(\sqrt{\mathcal{L}}, |T|)$.

We now present the main result of this section, which may be viewed as a restriction-type estimate for the sub-Laplacian with truncation along the spectrum of $|T|$. This type of truncation has been also considered earlier in various setups, for instance in \cite[Lemma 11]{Martini_Muller_Multiplier_Grushin_2014}, \cite[Lemma 10]{Martini_Spectral_Multiplier_Heisenberg_Reiter_2015}, \cite[Theorem 2.1]{Niedorf_Metivier_group_2023} and \cite[Theorem 3.4]{Niedorf_Bochner_Riesz_Grushin_2022}. All of these previous estimates were established within the framework of linear multiplier theory, and in the bilinear setting this was first considered in \cite{Bagchi_Molla_Singh_Bilinear_Bochner_Riesz_Grushin}. To the best of our knowledge, however, estimate of this type  for the complex sphere $\mathbb{S}$ have not yet been studied before and therefore constitute a new contribution in our setting.
\begin{proposition}
\label{Proposition: Truncated restriction type estimate} 
Let $F_M(\sqrt{\mathcal{L}}, |T|)$ be the operator defined as in \eqref{Definition for the truncated operator}. Then for $\epsilon>0$ we have
\begin{align*}
    \|F_M(\sqrt{\mathcal{L}}, |T|)f\|_{L^2} &\leq C \left( N^{d/2+\epsilon} 2^{-M \epsilon/2} + N^{Q/2} 2^{-M/2} \right) \|F(N \cdot)\|_{N, 2}\, \|f\|_{L^1} .
\end{align*}
\end{proposition}

\begin{proof}
First we write
\begin{align*}
    F_M(\sqrt{\mathcal{L}}, |T|)f(z) &= \int_{\mathbb{S}} \mathcal{K}_{F_M(\sqrt{\mathcal{L}}, |T|)}(z, w) f(w) \, d\sigma(w) ,
\end{align*}
where
\begin{align*}
    \mathcal{K}_{F_M(\sqrt{\mathcal{L}}, |T|)}(z, w) &= \sum_{\substack{\ell, \ell'=0 \\ \ell \neq \ell'}}^{\infty} F_M(\sqrt{\lambda_{\ell, \ell'}}, |\ell-\ell'|) Y_w^{\ell, \ell'}(z) .
\end{align*}
Applying Minkowski's integral inequality we get
\begin{align*}
    \|F_M(\sqrt{\mathcal{L}}, |T|)f\|_{L^2} &\leq \int_{\mathbb{S}} |f(w)| \left( \int_{\mathbb{S}} |\mathcal{K}_{F_M(\sqrt{\mathcal{L}}, |T|)}(z, w)|^2 \, d\sigma(z) \right)^{1/2} d\sigma(w) .
\end{align*}
Hence, in order to estimate $\|F_M(\sqrt{\mathcal{L}}, |T|)\|_{L^1 \to L^2}$, it is enough to prove the following
\begin{align}
\label{Enough to prove in truncated restriction}
    \sup_{w \in \mathbb{S}}\|\mathcal{K}_{F_M(\sqrt{\mathcal{L}}, |T|)}(\cdot, w)\|_{L^2} &\leq C \left( N^{d/2+\epsilon} 2^{-M \epsilon/2} + N^{Q/2} 2^{-M/2} \right) \|F(N \cdot)\|_{N, 2} .
\end{align}
Note that, using orthogonality we get
\begin{align}
\label{Expression: Use of orthogonality in kernel estimate}
    \|\mathcal{K}_{F_M(\sqrt{\mathcal{L}}, |T|)}(\cdot, w)\|_{L^2}^2 &= \sum_{\substack{\ell, \ell'=0 \\ \ell \neq \ell'}}^{\infty} |F_M(\sqrt{\lambda_{\ell, \ell'}}, |\ell-\ell'|)|^2 \|Y_w^{\ell, \ell'}\|_{L^2}^2 .
\end{align}
For $i \in \mathbb{N}$, define
\begin{align}
\label{Definition of Hi set}
    H_i &= \{(\ell, \ell') \in \mathbb{N}_0^2 : (i-1)^2 \leq \lambda_{\ell, \ell'} \leq i^2 \} .
\end{align}
Now recall that $\lambda_{\ell, \ell'} = 4\ell \ell'+2(n-1)(\ell+\ell')$ and $a= \frac{n-1}{2}$. Then one can easily see that
\begin{align}
\label{Product of l and l prime equivalent to i^2 case}
    (i-1)^2 \leq \lambda_{\ell, \ell'} \leq i^2 \quad \text{if and only if} \quad \frac{(i-1)^2+4a^2}{4} \leq (\ell+a) (\ell'+a) \leq \frac{i^2+4a^2}{4} .
\end{align} 
Using the estimate \eqref{Estimate of the L2 norm} into \eqref{Expression: Use of orthogonality in kernel estimate} we obtain
\begin{align}
\label{Use of norm of Zonal harmonics}
    & \|\mathcal{K}_{F_M(\sqrt{\mathcal{L}}, |T|)}(\cdot, w)\|_{L^2}^2 \\
    &\nonumber \leq C \sum_{i=1}^{N} \sup_{(\ell, \ell') \in H_i} |F_M(\sqrt{\lambda_{\ell, \ell'}}, |\ell-\ell'|)|^2 \sum_{(\ell, \ell') \in H_i} (\ell+a)^{n-2} (\ell'+a)^{n-2} (\ell+a+\ell'+a) \\
    &\nonumber \leq C \sum_{i=1}^{N} \sup_{(\ell, \ell') \in H_i} |F_M(\sqrt{\lambda_{\ell, \ell'}}, |\ell-\ell'|)|^2 \left(\frac{i^2 +4a^2}{4} \right)^{n-2} \sum_{(\ell, \ell') \in H_i} (\ell+a+\ell'+a) .
\end{align}
Let us denote $H_i' = \{(\ell, \ell') \in H_i : \ell < \ell' \}$. Since $(\ell+a+\ell'+a)$ is symmetric in $\ell$ and $\ell'$, we have
\begin{align}
\label{L2 kernel estimate before l prime}
    & \|\mathcal{K}_{F_M(\sqrt{\mathcal{L}}, |T|)}(\cdot, w)\|_{L^2}^2 \\
    &\nonumber \leq C \sum_{i=1}^{N} \sup_{(\ell, \ell') \in H_i} |F(\sqrt{\lambda_{\ell, \ell'}})|^2 |\Theta(2^{M}  |\ell-\ell'|)|^2 \left(\frac{i^2 +4a^2}{4} \right)^{n-2} \sum_{(\ell, \ell') \in H_i'} (\ell+a+\ell'+a) \\
    &\nonumber \leq C \sum_{i=1}^{N} \sup_{(\ell, \ell') \in H_i} |F(\sqrt{\lambda_{\ell, \ell'}})|^2 \, i^{2n-4} \sum_{(\ell, \ell') \in H_{i,M}'} (\ell'+a) ,
\end{align}
where $H_{i,M}' = \{(\ell, \ell') \in H_i' : 2^{-M-1}N^2 \leq \ell'-\ell \leq 2^{-M+1}N^2 \}$.

\medskip

Note that if $(\ell, \ell') \in H_{i,M}'$, then we have the following upper and lower bounds for $\ell'-\ell$ and $(\ell+a) (\ell'+a)$:
\begin{align*}
    2^{-M-1}N^2 \leq \ell'-\ell \leq 2^{-M+1}N^2 \quad \text{and} \quad \frac{(i-1)^2+4a^2}{4} \leq (\ell+a) (\ell'+a) \leq \frac{i^2+4a^2}{4} ,
\end{align*}
Using these bounds one can obtain two quadratic inequalities in $\ell+a$ which upon solving give us the following estimate
\begin{align}
\label{boundforthesumlplusa}
    & A_{i, M, N} \leq (\ell+a) \leq B_{i, M, N} .
\end{align}
where
\begin{align*}
    A_{i, M, N} &= \frac{-2^{-M+1} N^2 + \sqrt{(2^{-M+1}N^2)^2+((i-1)^2+4a^2)}}{2} \\
   \text{and} \quad B_{i, M, N} &= \frac{-2^{-M-1} N^2 + \sqrt{(2^{-M-1}N^2)^2+(i^2+4a^2)}}{2} .
\end{align*}
On the other hand, if $(\ell, \ell') \in H_i'$ then 
\begin{align}
\label{On the other bound for l and lprime}
    0\leq \ell \leq i \quad \quad \text{and} \quad \quad \frac{(i-1)^2+4a^2}{4 (\ell+a)} \leq (\ell'+a) \leq \frac{i^2+4a^2}{4 (\ell+a)} .
\end{align}
Now we estimate 
\begin{align*}
    \mathfrak{J} := \sum_{(\ell, \ell') \in H_{i,M}'} (\ell'+a) .
\end{align*}
by considering the following two cases.

\medskip

\noindent \textbf{Case I: \texorpdfstring{$2^M \geq N^2$.}{}}
In this case $\ell' \leq \ell + 2$. So that $\ell'+a \leq C (\ell+a)$. Therefore using \eqref{boundforthesumlplusa} and \eqref{On the other bound for l and lprime} we obtain
\begin{align}
\label{Estimate of J first case}
    \mathfrak{J} &\leq C \sum_{A_{i, M, N} \leq (\ell+a) \leq \min\{B_{i, M, N}, (i+a)\}} \sum_{\frac{(i-1)^2+4a^2}{4 (\ell+a)} \leq (\ell'+a) \leq \frac{i^2+4a^2}{4 (\ell+a)}} (\ell+a) \\
    &\nonumber \leq C \sum_{\ell=0}^{i} (\ell+a) \frac{i+a}{\ell+a} \leq C~ i^2 2^{-M \epsilon} 2^{M \epsilon} .
\end{align}
Since $(\ell, \ell') \in H_{i,M}'$, we see that $2^M \leq 4 N^2$, otherwise $1 \leq \ell'-\ell \leq 2^{-M+1}N^2 \leq 1/2$, which is absurd. Hence from \eqref{Estimate of J first case} we get
\begin{align}
\label{Estimate of J first case last step}
    \mathfrak{J} &\leq C~ i^2 2^{-M \epsilon} N^{2 \epsilon} .
\end{align}

\medskip

\noindent \textbf{Case II: \texorpdfstring{$2^M < N^2$.}{}}
Similarly as in \eqref{Estimate of J first case} and using the fact $\ell'+a \leq \ell + a + 2^{-M+1}N^2$ we get
\begin{align}
\label{Estimate of J second case}
    \mathfrak{J} &\leq C \sum_{A_{i,M,N} \lesssim (\ell+a) \lesssim \min\{B_{i,M, N}, (i+a)\}} \sum_{\frac{(i-1)^2+4a^2}{4 (\ell+a)} \leq (\ell'+a) \leq \frac{i^2+4a^2}{4 (\ell+a)}} (\ell+a+ 2^{-M+1}N^2) \\
    &\nonumber \leq C \sum_{A_{i,M,N} \lesssim (\ell+a) \lesssim \min\{B_{i,M, N}, (i+a)\}} [(\ell+a)+ 2^{-M+1}N^2] \frac{i+a}{\ell+a} \\
    &\nonumber \leq C~ (i+a)^2 + (i+a) \, 2^{-M} N^2 \sum_{A_{i,M,N} \lesssim (\ell+a) \lesssim B_{i,M, N}} \frac{1}{\ell+a} \\
    &\nonumber \leq C~ i^2 \, 2^{-M \epsilon} N^{2 \epsilon} + i \, 2^{-M} N^2 .
\end{align}
Combining both the estimates \eqref{Estimate of J first case last step} and \eqref{Estimate of J second case} from \eqref{L2 kernel estimate before l prime} for all $w \in \mathbb{S}$ we obtain
\begin{align}
\label{Creating discrete norm in Restriction}
    \|\mathcal{K}_{F_M(\sqrt{\mathcal{L}}, |T|)}(\cdot, w)\|_{L^2}^2 & \leq C \sum_{i=1}^{N} \sup_{(\ell, \ell') \in H_i} |F(\sqrt{\lambda_{\ell, \ell'}})|^2 \, i^{2n-4} \left( i^2\, 2^{-M \epsilon} N^{2 \epsilon} + i\, 2^{-M} N^2 \right) \\
    &\nonumber \leq C \left(N^{2n-1+2 \epsilon} 2^{-M \epsilon} + N^{2n} 2^{-M} \right) \left( \frac{1}{N} \sum_{i=1}^{N} \sup_{\lambda \in [(i-1)^2, i^2]} |F(\sqrt{\lambda})|^2 \right) \\
    &\nonumber \leq C \left(N^{2n-1+2 \epsilon} 2^{-M \epsilon} + N^{2n} 2^{-M} \right) \left( \frac{1}{N} \sum_{i=1}^{N} \sup_{\lambda \in [\frac{(i-1)}{N}, \frac{i}{N}]} |F(N \lambda)|^2 \right) \\
    &\nonumber \leq C (N^{d+2\epsilon} 2^{-M \epsilon} + N^{Q} 2^{-M}) \|F(N \cdot)\|_{N, 2}^2 .
\end{align}
This completes the proof of \eqref{Enough to prove in truncated restriction}.
\end{proof}

The following proposition gives the $L^1 \to L^2$ estimates of the projection $\pi_{\ell, \ell'}$, which is required for the proof of Theorem \ref{Theorem: Bilinear Bochner-Riesz theorem with restricted f and g}. This type of estimates are already available in the literature, see \cite[Lemma 3.2]{Casarino_Peloso_Riesz_Mean_2011}. However, our result improves their projection estimates under the additional condition $|\ell-\ell'| \leq \kappa R^s$, for some $\kappa>0$ and $s\geq 0$.
\begin{proposition}
\label{Proposition: Projection estimate near diagonal}
Let $0\leq a <b$ be fixed. Then for any $s\geq 0$ and $\kappa, R>0$ we have
\begin{align}
\label{L1 to L2 projection estimate near diagonal}
    \Bigg\|\sum_{\substack{\lambda_{\ell, \ell'} \in [a,b] \\ |\ell-\ell'| \leq \kappa R^s}} \pi_{\ell, \ell'}f \Bigg\|_{L^{2}}^2 &\leq C b^{n-2}(\sqrt{b}+R^s) (\sqrt{b}-\sqrt{a}) R^s \, \|f\|_{L^1}^2 .
\end{align}
In particular, if $s \in [0,1]$, $a=R^2(1-2^{-j+1})^2$ and $b=R^2(1-2^{-j-1})^2$, then we have
\begin{align*}
    \Bigg\|\sum_{\substack{\lambda_{\ell, \ell'} \in [a,b] \\ |\ell-\ell'| \leq \kappa R^s}} \pi_{\ell, \ell'}f \Bigg\|_{L^{2}} &\leq C R^{(n-1+s/2)} 2^{-j/2} \, \|f\|_{L^1} .
\end{align*}
    
\end{proposition}

\begin{remark}
When $\ell=\ell'$, $s=0$, $a= R(1-2^{-j+1})$ and $b=R(1-2^{-j-1})$, then this result coincides with \cite[just before eq. 5.3]{Casarino_Peloso_Riesz_Mean_2011}.
    
\end{remark}

\begin{proof}
Applying orthogonality and Lemma \ref{Lemma: L1 to l2 projection estimate} yields
\begin{align}
\label{Projection estimate near diagonal case}
    \Bigg\|\sum_{\substack{\lambda_{\ell, \ell'} \in [a,b] \\ |\ell-\ell'| \leq \kappa R^s}} \pi_{\ell, \ell'}f \Bigg\|_{L^{2}}^2 = \sum_{\substack{\lambda_{\ell, \ell'} \in [a,b] \\ |\ell-\ell'| \leq \kappa R^s}} \|\pi_{\ell, \ell'}f\|_{L^{2}}^2
    & \leq C \sum_{\substack{\lambda_{\ell, \ell'} \in [a,b] \\ |\ell-\ell'| \leq \kappa R^s}} \lambda_{\ell, \ell'}^{n-2} (\ell+\ell') \, \|f\|_{L^1}^2 \\
    &\nonumber \leq C b^{n-2} \|f\|_{L^1}^2 \sum_{\substack{\lambda_{\ell, \ell'} \in [a,b] \\ |\ell-\ell'| \leq \kappa R^s}} (\ell+\ell') .
\end{align}
Denote $m= 2\ell+n-1 $ and $m'= 2\ell'+n-1$. Now setting $\Tilde{a}=a+(n-1)^2$ and $\Tilde{b}=b+(n-1)^2$, we see that
\begin{align*}
    a\leq \lambda_{\ell, \ell'} \leq b \quad \text{if and only if} \quad \Tilde{a} \leq m m' \leq \Tilde{b} .
\end{align*}
Take $h=m'-m$. Then further setting  $A_h = \frac{-h+\sqrt{h^2+4\Tilde{a}}}{2}$ and $ B_h = \frac{-h+\sqrt{h^2+4\Tilde{b}}}{2}$, we obtain 
\medskip
\begin{align*}
   \Tilde{a} \leq m (m+h) \leq \Tilde{b} \quad \text{if and only if} \quad A_h \leq m \leq B_h .
\end{align*}
Hence we get
\begin{align}
\label{Final estimate for sum in m}
    \sum_{\substack{\lambda_{\ell, \ell'} \in [a,b] \\ |\ell-\ell'| \leq \kappa R^s}} (\ell+\ell') \leq C \sum_{\substack{\Tilde{a} \leq m m' \leq \Tilde{b} \\ |m-m'|\leq \kappa R^s}} (m+m') &\leq C \sum_{|h| \leq \kappa R^s} \sum_{A_h \leq m \leq B_h} (m+ h) \\
    &\nonumber \leq C \sum_{|h| \leq \kappa R^s} (B_h + h) (B_h-A_h) \\
    &\nonumber \leq C (\sqrt{\Tilde{b}}-\sqrt{\Tilde{a}}) (\sqrt{\Tilde{b}} + R^{s}) R^s \\
    &\nonumber \leq C (\sqrt{b}+ R^s)(\sqrt{b}-\sqrt{a})R^s .
\end{align}
Finally, plugging the estimate \eqref{Final estimate for sum in m} into \eqref{Projection estimate near diagonal case} we got the required estimate \eqref{L1 to L2 projection estimate near diagonal}.
\end{proof}

As an application of the above proposition, we derive the following corollary, which will be instrumental in  proving Theorem \ref{Theorem: Bilinear Bochner-Riesz theorem with restricted f and g} and Theorem \ref{Theorem: Bilinear near diagonal case}.
\begin{corollary}
\label{Corollary: L1 to L2 estimate for diagonal operator}
Let $F : \mathbb{R} \to \mathbb{C}$ be a bounded Borel function such that $\supp{F} \subseteq [0, N]$. For $s\in [0,1]$, define
\begin{align}
\label{Definition of the operator Ls}
    \chi_s(|T|) F(\sqrt{\mathcal{L}})f(z) &:= \sum_{\substack{\ell, \ell'=0 \\ |\ell- \ell'| \leq \kappa N^s}}^{\infty} F(\sqrt{\lambda_{\ell, \ell'}})\, \pi_{\ell, \ell'}f(z) .
\end{align}
Then we have
\begin{align*}
    \|\chi_s(|T|) F(\sqrt{\mathcal{L}})f\|_{L^2} &\leq C N^{(Q/2-1+s/2)} \|F\|_{L^{\infty}} \|f\|_{L^1} .
\end{align*}
\end{corollary}

\begin{proof}
Using orthogonality and applying Proposition \ref{Proposition: Projection estimate near diagonal} with $a=0$ and $b=N^2$, we have
\begin{align*}
    \|\chi_s(|T|) F(\sqrt{\mathcal{L}})f\|_{L^2}^2 = \sum_{\substack{\ell, \ell'=0 \\ |\ell- \ell'| \leq \kappa N^s}}^{\infty} |F(\sqrt{\lambda_{\ell, \ell'}})|^2 \|\pi_{\ell, \ell'} f\|_{L^2}^2 &\leq \|F\|_{L^{\infty}}^2 \sum_{\substack{\lambda_{\ell, \ell} \in [0,N^2] \\ |\ell- \ell'| \leq \kappa N^s}} \|\pi_{\ell, \ell'} f\|_{L^2}^2 \\
    &\leq C \|F\|_{L^{\infty}}^2 \, N^{2(n-1+s/2)} \|f\|_{L^1}^2 .
\end{align*}
    This yields the desired estimate. 
\end{proof}

\section{Weighted Plancherel estimates}
\label{Section: Weighted Plancherel estimates}
In this section we will discuss various weighted Plancherel estimates for the sub-Laplacian $\mathcal{L}$ on complex sphere $\mathbb{S}$. The discussion is organized into two parts: the linear and the bilinear weighted Plancherel estimates.

\subsection{Linear weighted Plancherel estimates}
Let us begin this subsection by recalling a known weighted Plancherel estimate that is available in the literature.
\begin{proposition}\cite[Theorem 6.1 (v)]{Cowling_Kilima_Sikora_sublaplacian_2011}
\label{Proposition: Linear weighted Plancherel with weight}
Let $m : \mathbb{R} \to \mathbb{C}$ such that $\supp{m} \subseteq [0,N]$. Then for all $z \in \mathbb{S}$ and $0\leq \gamma <1/2$ we have
\begin{align*}
    \Bigg( \int_{\mathbb{S}} |\mathcal{K}_{m(\sqrt{\mathcal{L}})}(z,w) (1+N \varpi(z,w))^{\gamma}|^2 \, d\sigma(w) \Bigg)^{1/2} & \leq C N^{Q/2} \|m(N\cdot)\|_{N,2} .
\end{align*}
    
\end{proposition}

Applying the estimate \eqref{Cowling Sikora norm relation} in the above proposition, we immediately obtain the following corollary, which will be used in the proof of Theorem \ref{Theorem: Bilinear Bochner-Riesz Main theorem}.
\begin{corollary}
\label{Corollary: Linear weighted Plancherel}
Let $m : \mathbb{R} \to \mathbb{C}$ such that $\supp{m} \subseteq [0,N]$. Then for all $z \in \mathbb{S}$, $0\leq \gamma <1/2$ and $\beta>1/2$ we have
\begin{align*}
    \Bigg( \int_{\mathbb{S}} |\mathcal{K}_{m(\sqrt{\mathcal{L}})}(z,w) (1+N \varpi(z,w))^{\gamma}|^2 \, d\sigma(w) \Bigg)^{1/2} & \leq C N^{Q/2} \left(\|m(N\cdot)\|_{L^2} + N^{-\beta} \|m(N\cdot)\|_{L^2_{\beta}} \right) .
\end{align*}
    
\end{corollary}

We have the following improvement of the preceding result for the operator $\chi_s(|T|) m(\sqrt{\mathcal{L}})$, defined as in \eqref{Definition of the operator Ls}.
\begin{proposition}
\label{Proposition: Improved weighted Plancherel for near}
Let $m : \mathbb{R} \to \mathbb{C}$ such that $\supp{m} \subseteq [0,N]$. Then for all $z \in \mathbb{S}$, $0\leq \gamma <1/2$, $s \in [0,1]$ and $\beta>1/2$ we have
\begin{align*}
    \Bigg( \int_{\mathbb{S}} |\mathcal{K}_{\chi_s(|T|) m(\sqrt{\mathcal{L}})}(z,w) (1+N \varpi(z,w))^{\gamma}|^2 \, d\sigma(w) \Bigg)^{1/2} & \leq C N^{\frac{d+s}{2}} \left(\|m(N\cdot)\|_{L^2} + N^{-\beta} \|m(N\cdot)\|_{L^2_{\beta}} \right) .
\end{align*}
    
\end{proposition}

\begin{proof}
Let us first recall $H_i' = \{(\ell, \ell') \in H_i : \ell < \ell' \}$.
Note that our proof follows in same lines as in \cite[Theorem 6.1. (v) and eq. (23) of Proposition 5.2]{Cowling_Kilima_Sikora_sublaplacian_2011} along with the fact \eqref{Cowling Sikora norm relation}. However, in order to get the required improvement, following the proof of \cite[Theorem 6.1. (v) and eq. (23) of Proposition 5.2]{Cowling_Kilima_Sikora_sublaplacian_2011} it is enough to prove the following estimate: for $0\leq \gamma <1/2$,
\begin{align}
\label{Required to prove in New operator}
    \sum_{\substack{(\ell, \ell') \in H_i' \\ \ell'-\ell \leq \kappa N^s}} (\ell'+a)^{1-2\gamma} &\leq C N^{2+s-4\gamma}, \quad \quad \text{where} \quad a= \frac{n-1}{2} .
\end{align}
The above can be estimated similarly as in \eqref{Final estimate for sum in m}. Note $(\ell, \ell') \in H_i$ if and only if $(i-1)^2 \leq \lambda_{\ell, \ell'} \leq i^2$. Therefore, following the same notation as in the proof of Proposition \ref{Proposition: Projection estimate near diagonal}, in fact taking $a=(i-1)^2$ and $b=i^2$ with $i\leq N$ we see that
\begin{align}
\label{Number of integer in Ah and Bh}
    \# \left( \mathbb{Z} \cap [A_h, B_h] \right) \lesssim 1 .
\end{align}
Hence estimating similar to \eqref{Final estimate for sum in m} and using the fact \eqref{Number of integer in Ah and Bh} for $0\leq \gamma <1/2$, $s \in [0,1]$ we get
\begin{align*}
    \sum_{\substack{(\ell, \ell') \in H_i' \\ \ell'-\ell \leq \kappa N^s}} (\ell'+a)^{1-2\gamma} &\leq C \sum_{0\leq h \leq k N^s} \sum_{A_h \leq m \leq B_h} (m+h)^{1-2\gamma} 
    \leq C \sum_{0\leq h \leq k N^s} (B_h + N^s)^{1-2\gamma} \\
    &\leq C (N + N^s)^{1-2\gamma} N^{s}
    \leq C N^{1-2\gamma} N^{s} 
    \leq C N^{2+s-4\gamma} .
\end{align*}
    This gives the required estimate \eqref{Required to prove in New operator}.
\end{proof}

\begin{proposition}
\label{Prop: Linear L1 norm of kernel}
Let $\beta>d/2$. Then for all $m : \mathbb{R} \to \mathbb{C}$ supported in $[0,R]$ and $z \in \mathbb{S}$ we have
\begin{align*}
    \|\mathcal{K}_{m(\sqrt{\mathcal{L}})}(z,\cdot)\|_{L^1(\mathbb{S})} &\leq C \|m(R\cdot)\|_{L^2_{\beta}(\mathbb{R})} .
\end{align*}
    
\end{proposition}

\begin{proof}
The above proposition is well known in the literature, however, not explicitly mentioned in the form stated above. For instance, with the help of the estimates available in \cite{Cowling_Kilima_Sikora_sublaplacian_2011} and adapting the ideas from \cite[Proposition 4.3 (i) and 4.4]{Martini_Sharp_Multiplier_Kohn_Laplacian_2017}, one can obtain the required estimate.
\end{proof}

\begin{remark}
Note that from the above proposition we can easily obtain $L^p$-boundedness of the compactly supported spectral multipliers and in particular for Bochner-Riesz multipliers. In fact, if $m : \mathbb{R} \to \mathbb{C}$ be a bounded Borel function supported in $[0,R]$ for $R>0$, then using Proposition \ref{Prop: Linear L1 norm of kernel} for $1\leq p \leq \infty$ and $\beta>d/2$ we get
\begin{align}
\label{Boundedness of spectral multiplier}
    \|m(\mathcal{L})f\|_{L^p(\mathbb{S})} &\leq C \|m(R\cdot)\|_{L^2_{\beta}(\mathbb{R})} \|f\|_{L^p(\mathbb{S})} .
\end{align}
Consequently, we also obtain the boundedness of Bochner-Riesz multipliers associated with the sub-Laplacian $\mathcal{L}$. In fact, for $1\leq p \leq \infty$ whenever $\delta>(d-1)/2$ we have
\begin{align*}
    \|S_R^{\delta}(\sqrt{\mathcal{L}})f\|_{L^p(\mathbb{S})} &\leq C \|f\|_{L^p(\mathbb{S})} .
\end{align*}
We would like to emphasize that, for $1<p<\infty$ the above result \eqref{Boundedness of spectral multiplier} was already established in \cite{Cowling_Kilima_Sikora_sublaplacian_2011} for general multipliers, not necessarily compactly supported. There they have proved the $L^p$boundedness of the spectral multipliers for $1<p<\infty$ and weak $(1,1)$ estimates. Here under the additional assumption on the multiplier, that is, if $m$ is compactly support, we also obtain the $L^1$ and $L^{\infty}$-boundedness.

\end{remark}

The following proposition is one of the main contributions of this subsection, which is about the weighted Plancherel estimates with large power of weights for the operator $F_M(\sqrt{\mathcal{L}}, |T|)$, see \eqref{Definition for the truncated operator}. The principle novelty here is we prove our estimate holds for all $\mathfrak{N} \in \mathbb{N}_0$ unlike Proposition \ref{Proposition: Linear weighted Plancherel with weight} where the estimate was only shown for $0\leq \gamma<1/2$.

\begin{proposition}
\label{Proposition: Weighted Plancherel for large N case}
Let $F_M$ be as defined in \eqref{Definition of truncated function}. Then for all $\mathfrak{N} \in \mathbb{N}_0$ and for all $w \in \mathbb{S}$ we have
\begin{align}
\label{Weighted Plancherel large power of weight}
    \left( \int_{\mathbb{S}} |\varpi(z,w)^{\mathfrak{N}} \mathcal{K}_{F_M(\sqrt{\mathcal{L}}, |T|)}(z,w)|^2 \, dz \right)^{1/2} &\leq C_N N^{Q/2-\mathfrak{N}} 2^{M \mathfrak{N}} \|F(N \cdot)\|_{N, 2} .
\end{align}
    
\end{proposition}

Before we give the proof we will discuss an important lemma that encapsulates information necessary for the proof. In particular, this lemma describes the action of power of weights $\varpi$ over the zonal spherical harmonics.

\begin{lemma}
\label{Lemma: Shift of Zonal harmonics}
For any $N \in \mathbb{N}_0$, we have
\begin{align*}
 (1-|\langle z,w \rangle|^2)^N Y_w^{\ell, \ell'}(z) &=   C_{\ell, \ell'} Y_w^{\ell, \ell'}(z) + \sum_{\substack{j=-N \\ j \neq 0}}^{N} C_{\ell+j, \ell'+j} Y_w^{\ell+j, \ell'+j}(z) ,
\end{align*}
where
\begin{align*}
    C_{\ell, \ell'}= (1-\beta_{\ell, \ell'})^{N} + \mathcal{E}_{\ell, \ell'} \quad \text{with} \quad \mathcal{E}_{\ell, \ell'} \lesssim (1-\beta_{\ell, \ell'})^N ,
\end{align*}
and $C_{\ell+j, \ell'+j}=0$ if either $(\ell+j)<0$ or $(\ell'+j)<0$.
    
\end{lemma}

\begin{proof}
Recall from \eqref{Shifts of zonal spherical harmonics} for $\ell, \ell' \in \mathbb{N}_0$, we have
\begin{align*}
    |\langle z, w \rangle|^2 Y_w^{\ell,\ell'}(z) = \alpha_{\ell,\ell'} Y_w^{\ell+1,\ell'+1}(z) + \beta_{\ell,\ell'} Y_w^{\ell,\ell'}(z) + \gamma_{\ell,\ell'} Y_w^{\ell-1,\ell'-1}(z) .
\end{align*}
Let us now write 
\begin{align*}
    |\langle z,w \rangle|^2 Y_w^{\ell, \ell'}(z) &= (D + S_{+} + S_{-})Y_w^{\ell, \ell'}(z) ,
\end{align*}
where
\begin{align}
\label{Action of D and S}
    D Y_w^{\ell, \ell'}(z) &= \beta_{\ell,\ell'} Y_w^{\ell,\ell'}(z) , \\
    \label{Action of S+ and S-}
    S_{+} Y_w^{\ell, \ell'}(z) = \alpha_{\ell,\ell'} Y_w^{\ell+1,\ell'+1}(z) \quad &\text{and} \quad S_{-} Y_w^{\ell, \ell'}(z) = \gamma_{\ell,\ell'} Y_w^{\ell-1,\ell'-1}(z) .
\end{align}
Since the operators $(1-D)$, $S_{+}$ and $S_{-}$ do not commute, we see that
\begin{align}
\label{Acting weight written in terms of words}
    (1-|\langle z,w \rangle|^2)^N Y_w^{\ell, \ell'}(z) &= (1-D-S_{+}-S_{-})^N Y_w^{\ell, \ell'}(z) \\
    &\nonumber = \sum_{k=0}^N (-1)^k \sum_{\substack{\text{words} \ W \ \\ \text{containing} \ u \ \text{copies of} \ S_{+}, \\ v \ \text{copies of} \ S_{-} \ \text{with} \ u+v=k \\ \text{and} \ N-k \ \text{copies of} \ 1-D}} W Y_w^{\ell, \ell'}(z) .
\end{align}
In view of the action of $(1-D)$ \eqref{Action of D and S}, $S_{+}$ and $S_{-}$ \eqref{Action of S+ and S-} over $Y_w^{\ell, \ell'}(z)$, the above expression can be written as
\begin{align}
\label{Writing j zero and non zero parts}
    \sum_{\substack{j=-N \\ j \neq 0}}^N C_{\ell+j, \ell'+j} Y_w^{\ell+j, \ell'+j}(z) + C_{\ell, \ell'} Y_w^{\ell, \ell'}(z) ,
\end{align}
for some constants $C_{\ell+j, \ell'+j}$ for $j=0, \pm 1, \ldots, \pm N$. Note that if either $(\ell+j)<0$ or $(\ell'+j)<0$, then $C_{\ell+j, \ell'+j}=0$, which is due to \eqref{Shifts of zonal spherical harmonics} and \eqref{Definition of gamma}.

Here we are interested in the coefficient $C_{\ell, \ell'}$. Note that, when $k$ is odd, the terms in above summand \eqref{Acting weight written in terms of words} do not contribute to the coefficient $C_{\ell, \ell'}$. Precisely, $C_{\ell, \ell'} \neq 0$, only if $k$ is even. To see this, let us set $s=u-v$. Note that this measure how much $Y_w^{\ell, \ell'}(z)$ shifted after applying $u$ and $v$ many $S_{+}$ and $S_{-}$ respectively. Since $u+v=k$, we have
\begin{align*}
    s= k-2v .
\end{align*}
See that $C_{\ell, \ell'} \neq 0$, only if after applying all $(1-D)$, $S_{+}$ and $S_{-}$ to $Y_w^{\ell, \ell'}(z)$, the resulting element is same as $Y_w^{\ell, \ell'}(z)$, that is, eventually there is no shift to $Y_w^{\ell, \ell'}(z)$. This happens only when $s=0$. Hence $k$ has to be even and also $u=v$. This means in order to get $C_{\ell, \ell'} \neq 0$, number of $S_{+}$ and $S_{-}$ has to be same, say $r:=u=v=k/2$.

Therefore, whenever $C_{\ell, \ell'} \neq 0$ for $r \in \{0, 1, \ldots, \lfloor N/2 \rfloor\}$ we get
\begin{align}
\label{Only the interested coefficient cllprime}
    C_{\ell, \ell'} Y_w^{\ell, \ell'}(z) &= \sum_{r=0}^{\lfloor N/2 \rfloor} \sum_{\substack{\text{words} \ W \ \\ \text{containing} \ r \ \text{copies of} \ S_{+}, S_{-} \\ \text{and} \ N-2r \ \text{copies of} \ 1-D}} W Y_w^{\ell, \ell'}(z) .
\end{align}
Observe that, if $r=0$, then $W=(1-D)^{N}$ and hence from \eqref{Action of D and S} we have 
\begin{align}
\label{Coefficient for r=0 case}
    (1-D)^{N} Y_w^{\ell, \ell'}(z) &= (1-\beta_{\ell,\ell'})^{N} Y_w^{\ell, \ell'}(z) .
\end{align}
If $r=1$, then the word $W$ contain $1$ copy of $S_{+}$, $1$ copy of $S_{-}$ and $(N-2)$ copies of $(1-D)$. Consequently, in view of \eqref{Action of S+ and S-} and \eqref{Action of D and S} we get
\begin{align*}
    W Y_w^{\ell, \ell'}(z) &= \mathscr{C}_{\ell, \ell'} Y_w^{\ell, \ell'}(z) ,
\end{align*}
where $\mathscr{C}_{\ell, \ell'}$ must contain a factor of either $\alpha_{\ell, \ell'} \gamma_{\ell+1, \ell'+1}$ (if $S_{+}$ appear before $S_{-}$) or $\gamma_{\ell, \ell'} \alpha_{\ell-1, \ell'-1}$ (if $S_{-}$ appear before $S_{+}$) and the other factors are of the form $(1-\beta_{\ell, \ell'})$, $(1-\beta_{\ell+1, \ell'+1})$ or $(1-\beta_{\ell-1, \ell'-1})$ and for these later factors, their total numbers are add up to $N-2$. Note that when either $\ell=0$ or $\ell'=0$, the terms $\gamma_{\ell, \ell'} \alpha_{\ell-1, \ell'-1}$ and $(1-\beta_{\ell-1, \ell'-1})$ will not appear, because of \eqref{Definition of gamma}.

Now from \eqref{Definition of alpha} and \eqref{Definition of gamma} we have
\begin{align*}
    \alpha_{\ell, \ell'} \gamma_{\ell+1, \ell'+1} &= \frac{\ell'+1}{(n+\ell+\ell')} \frac{\ell+1}{(n+\ell+\ell'+1)} \frac{n+\ell-1}{(n+\ell+\ell')} \frac{n+\ell'-1}{(n+\ell+\ell'-1)} .
\end{align*}
On the other hand from \eqref{Definition of beta} we also get
\begin{align}
\label{Definition of 1-beta}
    (1-\beta_{\ell, \ell'})^2 & = \left(\frac{2\ell \ell'+(n-1)(\ell+\ell')+(n-1)(n-2)}{(n+\ell+\ell'-1)^2-1} \right)^2 .
\end{align}
From the above two expressions, for $\ell, \ell' \in \mathbb{N}_0$ we see that
\begin{align}
\label{Equivalence of alpha gamma-1}
    \alpha_{\ell, \ell'} \gamma_{\ell+1, \ell'+1} \sim (1-\beta_{\ell, \ell'})^2 .
\end{align}
Also note for $\ell, \ell' \in \mathbb{N}_0$ it satisfies
\begin{align}
\label{Equivalence of 1 -beta}
    (1-\beta_{\ell, \ell'}) \sim (1-\beta_{\ell+1, \ell'+1}) .
\end{align}
Therefore using the above observations \eqref{Equivalence of alpha gamma-1} and \eqref{Equivalence of 1 -beta} we obtain
\begin{align}
\label{Final expression for r=1 case}
    \mathscr{C}_{\ell, \ell'} &\lesssim (1-\beta_{\ell, \ell'})^{N-2} \alpha_{\ell, \ell'} \gamma_{\ell+1, \ell'+1} \quad \text{or} \quad (1-\beta_{\ell, \ell'})^{N-2} \gamma_{\ell, \ell'} \alpha_{\ell-1, \ell'-1} \\
    &\nonumber \lesssim (1-\beta_{\ell, \ell'})^N .
\end{align}
Now for the remaining cases $2 \leq r \leq \lfloor N/2 \rfloor$, the word $W$ contain $r$ copies of $S_{+}$, $S_{-}$ and $(N-2r)$ copies of $(1-D)$. Hence from \eqref{Action of S+ and S-} and \eqref{Action of D and S} we see
\begin{align*}
    W Y_w^{\ell, \ell'}(z) &= \mathscr{D}_{\ell, \ell'} Y_w^{\ell, \ell'}(z) ,
\end{align*}
where $\mathscr{D}_{\ell, \ell'}$ must contains factor of the form of $\alpha_{\ell+j, \ell'+j} \gamma_{\ell+j+1, \ell'+j+1}$ for some $|j| \leq r$ and they are $r$ many; while the other factors are of the form $(1-\beta_{\ell+j, \ell'+j})$ for some $|j| \leq r$ and for these later factors, their total numbers are add up to $N-2r$. Also note when either $(\ell+j)<0$ or $(\ell'+j)<0$ the terms $\alpha_{\ell+j, \ell'+j} \gamma_{\ell+j+1, \ell'+j+1}$ and $(1-\beta_{\ell+j, \ell'+j})$ will not appear.

Again from \eqref{Definition of alpha} and \eqref{Definition of gamma} we have
 \begin{align*}
    \alpha_{\ell+j, \ell'+j} \gamma_{\ell+j+1, \ell'+j+1} &= \frac{\ell'+j+1}{(n+\ell+\ell'+2j)} \frac{\ell+j+1}{(n+\ell+\ell'+2j+1)} \frac{n+\ell+j-1}{(n+\ell+\ell'+2j)} \frac{n+\ell'+j-1}{(n+\ell+\ell'+2j-1)} .
 \end{align*}
Hence in view of the above expression and \eqref{Definition of 1-beta}, for $|j| \leq r \leq \lfloor N/2 \rfloor$ and $(\ell+j), (\ell'+j) \in \mathbb{N}_0$ we see that
\begin{align*}
    \alpha_{\ell+j, \ell'+j} \gamma_{\ell+j+1, \ell'+j+1} \sim_{N} (1-\beta_{\ell, \ell'})^2 , \quad \quad \text{and} \quad \quad (1-\beta_{\ell, \ell'}) \sim_N (1-\beta_{\ell+j, \ell'+j}) .
\end{align*}
Consequently, we obtain
\begin{align}
\label{Coefficient for r non zero case}
    \mathscr{D}_{\ell, \ell'} &\lesssim (1-\beta_{\ell, \ell'})^{2r} (1-\beta_{\ell, \ell'})^{N-2r} \lesssim (1-\beta_{\ell, \ell'})^N .
\end{align}
Finally, combining the $r=0$ \eqref{Coefficient for r=0 case} and $r\geq 1$ (\ref{Final expression for r=1 case}, \ref{Coefficient for r non zero case}) cases, from \eqref{Only the interested coefficient cllprime} we get
\begin{align*}
    C_{\ell, \ell'}= (1-\beta_{\ell, \ell'})^{N} + \mathcal{E}_{\ell, \ell'} \quad \text{with} \quad \mathcal{E}_{\ell, \ell'} \lesssim_N (1-\beta_{\ell, \ell'})^N .
\end{align*}
Therefore in view of \eqref{Acting weight written in terms of words} and \eqref{Writing j zero and non zero parts} the proof of the lemma is completed.
\end{proof}

Now we are ready to prove Proposition \ref{Proposition: Weighted Plancherel for large N case}.

\begin{proof}[Proof of Proposition \ref{Proposition: Weighted Plancherel for large N case}]
Let us set $K(z,w) := \mathcal{K}_{F_M(\sqrt{\mathcal{L}}, |T|)}(z,w)$. Fix $\mathfrak{N} \in \mathbb{N}_0$ and we decompose $K(z,w) = \sum_{j=0}^{2\mathfrak{N}} K_j(z,w) $ where 
\begin{align*}
    K_j(z,w) &= \sum_{\substack{\ell, \ell'=0 \\ \ell \neq \ell' \\ \ell+\ell' \equiv j (\hspace{-3mm} \mod 2\mathfrak{N}+1)}}^{\infty} F(\sqrt{\lambda_{\ell, \ell'}}) \Theta(2^M |\ell-\ell'|) Y_w^{\ell, \ell'}(z) \quad \quad \text{for} \quad j=0, 1, \ldots, 2\mathfrak{N} .
\end{align*}
An application of Cauchy-Schwartz inequality gives
\begin{align}
\label{Application of Cauchy-Schwartz in weighted}
    \|\varpi(\cdot,w)^{\mathfrak{N}} K(\cdot,w)\|_{L^2}^2 =  \Big\|\varpi(\cdot,w)^{\mathfrak{N}} \sum_{j=0}^{2\mathfrak{N}} K_j(\cdot,w) \Big\|_{L^2}^2 &\leq (2\mathfrak{N}+1) \sum_{j=0}^{2\mathfrak{N}} \|\varpi(\cdot,w)^{\mathfrak{N}} K_j(\cdot,w)\|_{L^2}^2 \\
    &\nonumber= (2\mathfrak{N}+1) \sum_{j=0}^{2\mathfrak{N}} \left\langle \varpi(\cdot,w)^{2 \mathfrak{N}} K_j(\cdot,w), K_j(\cdot,w) \right\rangle .
\end{align}
Remember that $\varpi(z,w) = (1-|\langle z,w \rangle|^2)^{1/2}$. Then for each $j \in \{0, 1, \ldots, 2\mathfrak{N}\}$ we have
\begin{align}
\label{To show with weigh kernel}
    \varpi(z,w)^{2 \mathfrak{N}} K_j(z,w) &= \sum_{\substack{\ell, \ell'=0 \\ \ell \neq \ell' \\ \ell+\ell' \equiv j (\hspace{-3mm} \mod 2\mathfrak{N}+1)}}^{\infty} F(\sqrt{\lambda_{\ell, \ell'}}) \Theta(2^M |\ell-\ell'|) (1-|\langle z,w \rangle|^2)^{\mathfrak{N}} Y_w^{\ell, \ell'}(z) .
\end{align}
Recall that Lemma \ref{Lemma: Shift of Zonal harmonics} gives
\begin{align*}
    (1-|\langle z,w \rangle|^2)^{\mathfrak{N}} Y_w^{\ell, \ell'}(z) &= C_{\ell, \ell'} Y_w^{\ell, \ell'}(z) + \sum_{\substack{j=-\mathfrak{N} \\ j \neq 0}}^{\mathfrak{N}} C_{\ell+j, \ell'+j} Y_w^{\ell+j, \ell'+j}(z) ,
\end{align*}
where $C_{\ell, \ell'}= (1-\beta_{\ell, \ell'})^{\mathfrak{N}} + \mathcal{E}_{\ell, \ell'}$ with $\mathcal{E}_{\ell, \ell'} \lesssim (1-\beta_{\ell, \ell'})^{\mathfrak{N}}$ and $C_{\ell+j, \ell'+j}=0$ if either $(\ell+j)<0$ or $(\ell'+j)<0$.

\medskip
With the help of the above expression, from \eqref{To show with weigh kernel} we obtain
\begin{align}
\label{Wrting Kernel into main in ortho}
    \varpi(z,w)^{2 \mathfrak{N}} K_j(z,w) &= K_j^{main}(z,w) + K_j^{ortho}(z,w) ,
\end{align}
where
\begin{align*}
    K_j^{main}(z,w) &= \sum_{\substack{\ell, \ell'=0 \\ \ell \neq \ell' \\ \ell+\ell' \equiv j (\hspace{-3mm} \mod 2\mathfrak{N}+1)}}^{\infty} F(\sqrt{\lambda_{\ell, \ell'}}) \Theta(2^M |\ell-\ell'|) C_{\ell, \ell'} Y_w^{\ell, \ell'}(z) ,
\end{align*}
and $K_j^{ortho}(\cdot,w)$ is orthogonal to $K_j(\cdot, w)$ in $L^2(\mathbb{S})$.

\medskip
Consequently, plugging \eqref{Wrting Kernel into main in ortho} into the expression \eqref{Application of Cauchy-Schwartz in weighted} yields
\begin{align}
\label{Writing K as sum of Kj into main}
    \|\varpi(\cdot,w)^{\mathfrak{N}} K(\cdot,w)\|_{L^2}^2 &\leq (2\mathfrak{N}+1) \sum_{j=0}^{2\mathfrak{N}} \left\langle K_j^{main}(\cdot,w), K_j(\cdot,w) \right\rangle .
\end{align}
Using the fact $C_{\ell, \ell'} \lesssim (1-\beta_{\ell, \ell'})^{\mathfrak{N}}$, we obtain
\begin{align}
\label{Use of the fact of coefficient}
    \left\langle K_j^{main}(\cdot,w), K_j(\cdot,w) \right\rangle &\lesssim \sum_{\substack{\ell, \ell'=0 \\ \ell \neq \ell' \\ \ell+\ell' \equiv j (\hspace{-3mm} \mod 2\mathfrak{N}+1)}}^{\infty} |F(\sqrt{\lambda_{\ell, \ell'}})|^2 |\Theta(2^M |\ell-\ell'|)|^2 (1-\beta_{\ell, \ell'})^{\mathfrak{N}} \|Y_w^{\ell, \ell'}\|_{L^2}^2 .
\end{align}
From \eqref{Definition of 1-beta} we have
\begin{align}
\label{Definition of eta equals to 1 - beta}
    (1-\beta_{\ell, \ell'}) & = \frac{2\ell \ell'+(n-1)(\ell+\ell')+(n-1)(n-2)}{(n+\ell+\ell'-1)^2-1} =: \eta(\ell, \ell')^2 ,
\end{align}
Note that on the support of $\Theta$ and $F$ we get,
\begin{align*}
    (1-\beta_{\ell, \ell'}) &= \frac{\eta(\ell, \ell')^2 |\ell-\ell'|^2}{|\ell-\ell'|^2} \lesssim \frac{\lambda_{\ell, \ell'}}{|\ell-\ell'|^2} \lesssim \frac{N^2}{2^{-2M} N^4} \sim \frac{2^{2M}}{N^2} ,
\end{align*}
where we have used the fact that $\eta(\ell, \ell')^2 |\ell-\ell'|^2 \lesssim \lambda_{\ell, \ell'}$.

Hence from \eqref{Use of the fact of coefficient} we obtain
\begin{align}
\label{After use of cutoff theta}
    \left\langle K_j^{main}(\cdot,w), K_j(\cdot,w) \right\rangle &\lesssim \left(\frac{2^{2M}}{N^2} \right)^{\mathfrak{N}}  \sum_{\ell, \ell'=0 }^{\infty} |F(\sqrt{\lambda_{\ell, \ell'}})|^2 \|Y_w^{\ell, \ell'}\|_{L^2}^2 .
\end{align}
Therefore in view of \eqref{Writing K as sum of Kj into main} and \eqref{After use of cutoff theta} in order to get the required estimate \eqref{Weighted Plancherel large power of weight}, it is enough to prove the following:
\begin{align}
\label{To prove for discrete norm in weighted Plancherel}
    \sum_{\ell, \ell'=0 }^{\infty} |F(\sqrt{\lambda_{\ell, \ell'}})|^2 \|Y_w^{\ell, \ell'}\|_{L^2}^2 &\leq C N^{Q} \|F(N \cdot)\|_{N, 2}^2 .
\end{align}
Using the estimate \eqref{Estimate of the L2 norm}, estimating similarly as in \eqref{Use of norm of Zonal harmonics} and \eqref{L2 kernel estimate before l prime} yields
\begin{align}
\label{Taking out the sup norm}
    \sum_{\ell, \ell'=0 }^{\infty} |F(\sqrt{\lambda_{\ell, \ell'}})|^2 \|Y_w^{\ell, \ell'}\|_{L^2}^2 &\lesssim \sum_{i=1}^{N} \sup_{(\ell, \ell') \in H_i} |F(\sqrt{\lambda_{\ell, \ell'}})|^2 \, i^{2n-4} \sum_{(\ell, \ell') \in H_{i}'} (\ell'+a) ,
\end{align}
where $H_i' = \{(\ell, \ell') \in H_i : \ell < \ell' \}$.

Consequently, in view of the fact \eqref{On the other bound for l and lprime} we get
\begin{align*}
    \sum_{(\ell, \ell') \in H_{i}'} (\ell'+a) &\lesssim \sum_{\ell=0}^{i} \sum_{\frac{(i-1)^2+4a^2}{4 (\ell+a)} \leq (\ell'+a) \leq \frac{i^2+4a^2}{4 (\ell+a)}} (\ell'+a) \lesssim \sum_{\ell=0}^{\infty} \Big(\frac{i^2+4a^2}{4 (\ell+a)} \Big) \Big(\frac{i+a}{\ell+a}\Big) \lesssim i^3 .
\end{align*}
Finally, plugging the above estimate into \eqref{Taking out the sup norm} and then estimating similarly as in \eqref{Creating discrete norm in Restriction} we obtain \eqref{To prove for discrete norm in weighted Plancherel}. Hence the proof of the Proposition \ref{Proposition: Weighted Plancherel for large N case} is also completed.
\end{proof}

We now turn to the bilinear analogue of the above estimates. 
\subsection{Bilinear weighted Plancherel estimates} 
The principle new contribution of this subsection comes from our proof of the bilinear version of the weighted Plancherel estimates (Proposition \ref{Proposition: Bilinear weighted Plancherel with weight}). In order to do so, we will first start by defining discrete norm on $\mathbb{R}^2$.

\medskip

Let $N_1, N_2 \in \mathbb{N}$. For $F: \mathbb{R} \times \mathbb{R} \to \mathbb{C}$ supported in $[0,1] \times [0,1]$, the discrete norm on $\mathbb{R}^2$ is denoted by $\|F\|_{N_1,N_2,2}$ and is defined by
\begin{align*}
    \|F\|_{N_1,N_2,2} = \left(\frac{1}{N_1 N_2} \sum_{i=1}^{N_1} \sum_{j=1}^{N_2} \sup_{\substack{\lambda_1 \in [(i-1)/N_1, i/N_1] \\ \lambda_2 \in [(j-1)/N_2, j/N_2]}} |F(\lambda_1,\lambda_2)|^2 \right)^{1/2}.
\end{align*}

\medskip

Below we obtain the bilinear analogue of the Proposition \ref{Proposition: Linear weighted Plancherel with weight}. This will play a crucial role in the forthcoming bilinear weighted Plancherel estimates and the proof of Theorem \ref{Theorem: Bilinear Bochner-Riesz Main theorem} at $(p_1, p_2,p)=(\infty, \infty, \infty)$. From \eqref{Definition of bilinear kernel} recall that $\mathcal{K}_{m(\sqrt{\mathcal{L}_1}, \sqrt{\mathcal{L}_2})}^{bi}$ denote the integral kernel corresponding to the bilinear multiplier $m(\sqrt{\mathcal{L}_1}, \sqrt{\mathcal{L}_2})$.

\begin{proposition}
\label{Proposition: Bilinear weighted Plancherel with weight}
Let $m : \mathbb{R}^2 \to \mathbb{C}$ such that $\supp{m} \subseteq [0,N]^2$. Then for all $z \in \mathbb{S}$ and $0\leq \alpha_1, \alpha_2 <1/2$ we have
\begin{align}
\label{Required estimate for bilinear weighted Plancherel}
    & \Bigg( \int_{\mathbb{S}} \int_{\mathbb{S}} |\mathcal{K}_{m(\sqrt{\mathcal{L}_1}, \sqrt{\mathcal{L}_2})}^{bi}(z,w,u) (1+N \varpi(z,w))^{\alpha_1} (1+N \varpi(z,u))^{\alpha_2}|^2 \, d\sigma(w) \, d\sigma(u) \Bigg)^{1/2} \\
    &\nonumber \leq C N^Q \|m(N\cdot, N \cdot)\|_{N,N,2} .
\end{align}
    
\end{proposition}

\begin{proof}
Let us denote $f(w,u) :=\mathcal{K}_{m(\sqrt{\mathcal{L}_1},\sqrt{\mathcal{L}_2})}^{bi}(z,w,u)$ and decompose it as follows
\begin{align}
\label{Decompositing f into f ij part}
    f(w,u) & = \sum_{i,j=0}^2 f_{i,j}(w,u),
\end{align}
where
\begin{align}
\label{Definition of F ij part}
    f_{i,j}(w,u) & = \sum_{\substack{\ell_1,\ell_1', \ell_2,\ell_2' =0 \\ \\ \ell_1+\ell_1' \equiv i (\hspace{-3mm} \mod 3) \\ \ell_2+\ell_2' \equiv j (\hspace{-3mm} \mod 3)}}^{\infty} m(\sqrt{\lambda_{\ell_1,\ell_1'}}, \sqrt{\lambda_{\ell_2,\ell_2'}}) \overline{Y_z^{\ell_1,\ell_1'}(w)}\,  \overline{Y_z^{\ell_2,\ell_2'}(u)} .
\end{align}
Let us denote
\begin{align*}
    \mathfrak{C}_{i,j} &:= \{(\ell_1, \ell_1', \ell_2, \ell_2') \in \mathbb{N}_0^4 : \  \ell_1+\ell_1' \equiv i (\hspace{-3mm} \mod 3),\  \ell_2+\ell_2' \equiv j (\hspace{-3mm} \mod 3) \} .
\end{align*}
For $\gamma_1, \gamma_2 \geq 0$, we set
\begin{align}
\label{Definition of M gamma 1 and Gmma 2 part}
    \mathfrak{M}^{\gamma_1, \gamma_2}f(w,u) &:= \sum_{\substack{\ell_1,\ell_1', \ell_2,\ell_2' =0}}^{\infty} \eta(\ell_1,\ell_1')^{\gamma_1} \eta(\ell_2,\ell_2')^{\gamma_2} m(\sqrt{\lambda_{\ell_1,\ell_1'}}, \sqrt{\lambda_{\ell_2,\ell_2'}}) \overline{Y_z^{\ell_1,\ell_1'}(w)}\,  \overline{Y_z^{\ell_2,\ell_2'}(u)}, 
\end{align}
where for $i=1,2$ we have $\eta(\ell_i,\ell_i')= (1-\beta_{\ell_i,\ell_i'})^{1/2}$, see \eqref{Definition of eta equals to 1 - beta}.

Consequently, in view of \eqref{Decompositing f into f ij part} and orthogonality we obtain
\begin{align}
\label{Decomposition of M into M times fij part}
     \|\mathfrak{M}^{\gamma_1, \gamma_2} f\|_{L^2(d\sigma(w),d\sigma(u))}^2 &= \sum_{i,j=0}^2 \|\mathfrak{M}^{\gamma_1, \gamma_2} f_{i,j} \|_{L^2(d\sigma(w),d\sigma(u))}^2.
\end{align}
 We now claim that for each $i,j \in \{0,1,2\}$, the following estimate holds
\begin{align}
\label{To prove the claim first}
    \|\varpi(z,\cdot)\varpi(z,\cdot) f_{i,j}\|_{L^2(d\sigma(w),d\sigma(u))}^2 &\leq \|\mathfrak{M}^{1,1} f_{i,j}\|_{L^2(d\sigma(w),d\sigma(u))}^2 \\
    \medskip
    \label{To prove the claim second}
    \|\varpi(z,\cdot) f_{i,j}\|_{L^2(d\sigma(w),d\sigma(u))}^2 &\leq \|\mathfrak{M}^{1,0} f_{i,j}\|_{L^2(d\sigma(w),d\sigma(u))}^2 \\
    \medskip
    \label{To prove the claim third}
    \|\varpi(z,\cdot) f_{i,j}\|_{L^2(d\sigma(w),d\sigma(u))}^2 &\leq \|\mathfrak{M}^{0,1} f_{i,j}\|_{L^2(d\sigma(w),d\sigma(u))}^2 .
\end{align}
Assuming the claim for the moment, let us complete the proof of the Proposition \ref{Proposition: Bilinear weighted Plancherel with weight}. Using \eqref{Decompositing f into f ij part}, \eqref{Decomposition of M into M times fij part} and the above claims we deduce that
\begin{align}
\label{Estimate for gamma 1 1 case first}
    \|\varpi(z,\cdot)\varpi(z,\cdot) f\|_{L^2(d\sigma(w),d\sigma(u))}^2 &\leq \|\mathfrak{M}^{1,1} f\|_{L^2(d\sigma(w),d\sigma(u))}^2 \\
    \medskip \label{Estimate for gamma 1 1 case second}
    \|\varpi(z,\cdot) f\|_{L^2(d\sigma(w),d\sigma(u))}^2 &\leq \|\mathfrak{M}^{1,0} f\|_{L^2(d\sigma(w),d\sigma(u))}^2 \\
    \medskip \label{Estimate for gamma 1 1 case third}
    \|\varpi(z,\cdot) f\|_{L^2(d\sigma(w),d\sigma(u))}^2 &\leq \|\mathfrak{M}^{0,1} f\|_{L^2(d\sigma(w),d\sigma(u))}^2 .
\end{align}
Note that trivially we have 
\begin{align}
\label{Estimate for gamma 0 0 case}
    \|f\|_{L^2(d\sigma(w),d\sigma(u))}^2 &\leq \|\mathfrak{M}^{0,0} f\|_{L^2(d\sigma(w),d\sigma(u))}^2 .
\end{align}
Now we interpolate the above four estimates \eqref{Estimate for gamma 1 1 case first}, \eqref{Estimate for gamma 1 1 case second}, \eqref{Estimate for gamma 1 1 case third} and \eqref{Estimate for gamma 0 0 case} as follows.  Fix $z \in \mathbb{S}$ and let $\alpha_1 \geq 0$.  Consider the linear operator 
\begin{align*}
    \mathfrak{T} : l^2_{\alpha_1} \to L^2_{\alpha_1}(\mathbb{S} \times \mathbb{S}) 
\end{align*}
defined by 
\begin{align*}
    \mathfrak{T} ((a_{\ell_1,\ell_1', \ell_2, \ell_2'})) = \sum_{\substack{\ell_1,\ell_1', \ell_2,\ell_2' = 0}}^{\infty} a_{\ell_1,\ell_1', \ell_2, \ell_2'}  \overline{Y_z^{\ell_1,\ell_1'}(\cdot)} \, \overline{Y_z^{\ell_2,\ell_2'}(\cdot)} ,
\end{align*}
where $l^2_{\alpha_1}$ and $L^2_{\alpha_1}(\mathbb{S} \times \mathbb{S})$ are the Banach spaces defined by 
\begin{align*}
    l^2_{\alpha_1} := \left\{ (a_{\ell_1,\ell_1', \ell_2, \ell_2'}) : \sum_{\substack{\ell_1,\ell_1', \ell_2,\ell_2' = 0}}^{\infty} |a_{\ell_1,\ell_1', \ell_2, \ell_2'}|^2 \eta(\ell_1,\ell_1')^{2\alpha_1} \omega_{2n-1}^{-2} d(\ell_1,\ell_1') d(\ell_2,\ell_2') < \infty \right\},
\end{align*}
and 
\begin{align*}
    L^2_{\alpha_1}(\mathbb{S}) := \left\{ g \in L^0(\mathbb{S} \times \mathbb{S}) : \int_{\mathbb{S}} \int_{\mathbb{S}} \varpi(z, w)^{2\alpha_1} |g(w, u)|^2 \, d\sigma(w) \, d\sigma(u) < \infty \right\} ,
\end{align*} 
where $L^0$ denote the space of all measurable functions.

Note that in view of \eqref{L2 norm of zonal harmonics} the inequality \eqref{Estimate for gamma 0 0 case} tells that the linear operator $\mathfrak{T}$ maps continuously $l^2_{0} $ into $L^2_{0}(\mathbb{S} \times \mathbb{S})$. On the other hand the inequality \eqref{Estimate for gamma 1 1 case second} implies that  $\mathfrak{T}$ maps continuously $l^2_{1} $ into $L^2_{1}(\mathbb{S} \times \mathbb{S})$. Therefore, by the Stein-Weiss complex interpolation with change of measure (see \cite{stein_Weiss_Interpolation_change_of_measure_1958}), it follows that for every $\alpha_1 \in [0,1]$, the operator $ \mathfrak{T} : l^2_{\alpha_1} \to L^2_{\alpha_1}(\mathbb{S} \times \mathbb{S})$ is bounded. Hence for $\alpha_1 \in [0,1]$ we obtain
\begin{align}
\label{After interpolation alpha 0 case}
    \|\varpi(z,\cdot)^{\alpha_1} f\|_{L^2(d\sigma(w),d\sigma(u))}^2 &\leq C_{\alpha_1} \|\mathfrak{M}^{\alpha_1,0} f\|_{L^2(d\sigma(w),d\sigma(u))}^2 .
\end{align}
Similarly, performing an interpolation between \eqref{Estimate for gamma 1 1 case first} and \eqref{Estimate for gamma 1 1 case third} gives, for every $\alpha_1 \in [0,1]$,
\begin{align}
\label{After interpolation alpha 1 case}
    \|\varpi(z,\cdot)^{\alpha_1} \varpi(z,\cdot) f\|_{L^2(d\sigma(w),d\sigma(u))}^2 &\leq C_{\alpha_1} \|\mathfrak{M}^{\alpha_1, 1} f\|_{L^2(d\sigma(w),d\sigma(u))}^2 .
\end{align}
Furthermore, for any $\alpha_1 \in [0,1]$, a subsequent interpolation between \eqref{After interpolation alpha 0 case} and \eqref{After interpolation alpha 1 case} with $0\leq \alpha_2 \leq 1$ yields
\begin{align}
\label{Final interpolation result}
     \|\varpi(z,\cdot)^{\alpha_1}\varpi(z,\cdot)^{\alpha_2} f\|_{L^2(d\sigma(w),d\sigma(u))}^2 &\leq C_{\alpha_1, \alpha_2} \|\mathfrak{M}^{\alpha_1,\alpha_2}f\|_{L^2(d\sigma(w),d\sigma(u))}^2 .
\end{align}
Note that from \eqref{Definition of M gamma 1 and Gmma 2 part} we have
\begin{align}
\label{L2 norm of M gamma 1 and Gamma2}
    & \|\mathfrak{M}^{\alpha_1,\alpha_2}f\|_{L^2(d\sigma(w),d\sigma(u))}^2 \\
    &\nonumber = \sum_{\substack{\ell_1,\ell_1', \ell_2,\ell_2' =0}}^{\infty} \eta(\ell_1,\ell_1')^{2\alpha_1} \eta(\ell_2,\ell_2')^{2\alpha_2} |m(\sqrt{\lambda_{\ell_1,\ell_1'}}, \sqrt{\lambda_{\ell_2,\ell_2'}})|^2 \|Y_z^{\ell_1,\ell_1'}\|_{L^2}^2 \|Y_z^{\ell_2,\ell_2'}\|_{L^2}^2 .
\end{align}
From \eqref{Definition of Hi set} first recall the definition of $H_i = \{(\ell, \ell') \in \mathbb{N}_0^2 : (i-1)^2 \leq \lambda_{\ell, \ell'} \leq i^2 \}$. As $m$ is supported on $[0,N]^2$, therefore $m$ vanishes whenever $\lambda_{\ell_1,\ell_1'}> N^2$ and $\lambda_{\ell_2,\ell_2'} > N^2$. Hence,  from \eqref{L2 norm of M gamma 1 and Gamma2} and using \cite[eq. (21) of Proposition 5.2]{Cowling_Kilima_Sikora_sublaplacian_2011} for $0\leq \alpha_1, \alpha_2 <1/2$, we deduce that
\begin{align}
\label{Use of Zonal bound in bilinear weighted}
    & \|\mathfrak{M}^{\alpha_1, \alpha_2}f\|_{L^2(d\sigma(w),d\sigma(u))}^2 \\
    &\nonumber \leq \sum_{i,j=1}^{N} \max \{|m(\sqrt{\lambda_{\ell_1,\ell_1'}}, \sqrt{\lambda_{\ell_2,\ell_2'}})|^2 : (\ell_1,\ell_1') \in H_i, (\ell_2,\ell_2')\in H_j \} \\
    &\nonumber \hspace{3cm} \left(\sum_{(\ell_1,\ell_1') \in H_i}\eta(\ell_1,\ell_1')^{2\alpha_1} \|Y_z^{\ell_1,\ell_1'}\|_{L^2}^2 \right) \left(\sum_{(\ell_2,\ell_2') \in H_j}\eta(\ell_2,\ell_2')^{2\alpha_2} \|Y_z^{\ell_2,\ell_2'}\|_{L^2}^2 \right) \\
    &\nonumber \leq C {N}^{(2n-1-2\alpha_1)} {N}^{(2n-1-2\alpha_2)} \sum_{i,j=1}^{N} \max \{|m(\sqrt{\lambda_{\ell_1,\ell_1'}}, \sqrt{\lambda_{\ell_2,\ell_2'}})|^2 : (\ell_1,\ell_1') \in H_i, (\ell_2,\ell_2')\in H_j \} .
\end{align}
Recall that $f(w,u) =\mathcal{K}_{m(\sqrt{\mathcal{L}_1},\sqrt{\mathcal{L}_2})}^{bi}(z,w,u)$. Consequently, plugging the above estimate into \eqref{Final interpolation result} $0\leq \alpha_1, \alpha_2<1/2$ we obtain
\begin{align}
\label{Bilinear final kernel estimate with weight}
    &\int_{\mathbb{S}} \int_{\mathbb{S}} |\mathcal{K}_{m(\sqrt{\mathcal{L}_1}, \sqrt{\mathcal{L}_2})}^{bi}(z,w,u)|^2 \varpi(z,w)^{2\alpha_1} \varpi(z,u)^{2\alpha_2}  \, d\sigma(w)\,  d\sigma(u) \\
    &\nonumber \leq C N^{(2n-1-2\alpha_1)} N^{(2n-1-2\alpha_2)} \sum_{i,j=1}^{N} \max\{|m(\sqrt{\lambda_{\ell_1,\ell_1'}}, \sqrt{\lambda_{\ell_2,\ell_2'}})|^2 : (\ell_1,\ell_1') \in H_i, (\ell_2,\ell_2')\in H_j \} \\
    &\nonumber \leq C N^{(2n-2\alpha_1)} N^{(2n-2\alpha_2)} \frac{1}{N^2} \sum_{i,j=1}^{N} \sup_{\substack{\lambda_1 \in [i-1,i) \\ \lambda_2 \in [j-1,j) }} |m(\lambda_1,\lambda_2)|^2 \\
    &\nonumber = C N^{2Q-2\alpha_1-2\alpha_2} \|m(N \cdot, N \cdot)\|_{N,N,2}^2 .
\end{align}
where the last inequality follows by the same argument  as in \eqref{Creating discrete norm in Restriction}.

Finally, interpolating between the two instances of the above estimate corresponding to the cases $\alpha_1, \alpha_2=0$ and arbitrary $\alpha_1, \alpha_2 \in [0,1/2)$, we get the required estimate \eqref{Required estimate for bilinear weighted Plancherel}.

\medskip
Therefore it only remains to establish the three claims \eqref{To prove the claim first}, \eqref{To prove the claim second} and \eqref{To prove the claim third}. Here we only show how to prove the claim \eqref{To prove the claim first}, since the arguments for \eqref{To prove the claim second} and \eqref{To prove the claim third} are analogous and in fact simpler.

Recall that $\varpi(x,y) = \sqrt{1-|\langle x,y \rangle|^2}$. Then we have
\begin{align}
\label{Plugging the definition of weight into L2 xpression}
    \|\varpi(z,\cdot)\varpi(z,\cdot) f_{i,j}\|_{L^2(d\sigma(w),d\sigma(u))}^2 & = \langle \varpi(z,\cdot)^2\varpi(z,\cdot)^2 f_{i,j}, f_{i,j} \rangle \\
    &\nonumber = \|f_{i,j}\|_{L^2(d\sigma(w),d\sigma(u))}^2 - \|\langle z, \cdot \rangle f_{i,j}\|_{L^2(d\sigma(w),d\sigma(u))}^2 \\
    &\nonumber \hspace{0.5cm} - \|\langle z, \cdot \rangle f_{i,j}\|_{L^2(d\sigma(w),d\sigma(u))}^2 + \|\langle z, \cdot \rangle \langle z, \cdot \rangle f_{i,j} \|_{L^2(d\sigma(w),d\sigma(u))}^2 .
\end{align}
From \eqref{Definition of F ij part}, we first note that 
\begin{align}
\label{Expression with beta0 part}
    \|f_{i,j}\|_{L^2(d\sigma(w),d\sigma(u))}^2 &= \sum_{(\ell_1, \ell_1', \ell_2, \ell_2') \in \mathfrak{C}_{i,j}} |m(\sqrt{\lambda_{\ell_1,\ell_1'}}, \sqrt{\lambda_{\ell_2,\ell_2'}})|^2 \|Y_z^{\ell_1,\ell_1'}\|_{L^2}^2 \|Y_z^{\ell_2,\ell_2'}\|_{L^2}^2.
\end{align}
Also observe that
\begin{align}
\label{Expression with z w part}
     \langle z, w \rangle f_{i,j}(w,u) & = \sum_{(\ell_1, \ell_1', \ell_2, \ell_2') \in \mathfrak{C}_{i,j}} m(\sqrt{\lambda_{\ell_1,\ell_1'}}, \sqrt{\lambda_{\ell_2,\ell_2'}}) \overline{\langle w,z \rangle Y_z^{\ell_1,\ell_1'}(w)} \,  \overline{Y_z^{\ell_2,\ell_2'}(u)} .
\end{align}
Therefore application of the identity \eqref{Shifts of zonal spherical harmonics} yields
\begin{align*}
    \|\langle \cdot, z \rangle Y_z^{\ell_1,\ell_1'} \|_{L^2}^2 &= \left\langle |\langle \cdot, z \rangle|^2 Y_z^{\ell_1,\ell_1'},  Y_z^{\ell_1,\ell_1'} \right\rangle = \beta_{\ell_1,\ell_1'} \|Y_z^{\ell_1,\ell_1'}\|_{L^2}^2 .
\end{align*}
Consequently, from \eqref{Expression with z w part} we obtain
\begin{align}
\label{Expression with beta1 part}
    \|\langle z,\cdot \rangle f_{i,j} \|_{L^2(d\sigma(w),d\sigma(u))}^2 &= \sum_{(\ell_1, \ell_1', \ell_2, \ell_2') \in \mathfrak{C}_{i,j}} |m(\sqrt{\lambda_{\ell_1,\ell_1'}}, \sqrt{\lambda_{\ell_2,\ell_2'}})|^2 \beta_{\ell_1,\ell_1'} \|Y_z^{\ell_1,\ell_1'}\|_{L^2}^2 \|Y_z^{\ell_2,\ell_2'}\|_{L^2}^2 .
\end{align}
In a similar manner, we also have the following:
\begin{align}
\label{Expression with beta2 part}
    \|\langle z,\cdot \rangle f_{i,j} \|_{L^2(d\sigma(w),d\sigma(u))}^2 &= \sum_{(\ell_1, \ell_1', \ell_2, \ell_2') \in \mathfrak{C}_{i,j}} |m(\sqrt{\lambda_{\ell_1,\ell_1'}}, \sqrt{\lambda_{\ell_2,\ell_2'}})|^2 \beta_{\ell_2,\ell_2'} \|Y_z^{\ell_1,\ell_1'}\|_{L^2}^2 \|Y_z^{\ell_2,\ell_2'}\|_{L^2}^2 ,
\end{align}
and
\begin{align}
\label{Expression with Beta 1 and beta 2 part}
    \|\langle z, \cdot \rangle \langle z, \cdot \rangle f_{i,j} \|_{L^2(d\sigma(w),d\sigma(u))}^2 &= \sum_{(\ell_1, \ell_1', \ell_2, \ell_2') \in \mathfrak{C}_{i,j}} |m(\sqrt{\lambda_{\ell_1,\ell_1'}}, \sqrt{\lambda_{\ell_2,\ell_2'}})|^2 \beta_{\ell_1,\ell_1'} \beta_{\ell_2,\ell_2'} \|Y_z^{\ell_1,\ell_1'}\|_{L^2}^2 \|Y_z^{\ell_2,\ell_2'}\|_{L^2}^2 .
\end{align}
Finally, collecting the identities \eqref{Expression with beta0 part}, \eqref{Expression with beta1 part}, \eqref{Expression with beta2 part} and \eqref{Expression with Beta 1 and beta 2 part} and substituting them into \eqref{Plugging the definition of weight into L2 xpression}, we obtain
\begin{align*}
    & \|\varpi(z,\cdot)\varpi(z,\cdot) f_{i,j}\|_{L^2(d\sigma(w),d\sigma(u))}^2 \\
    &= \sum_{(\ell_1, \ell_1', \ell_2, \ell_2') \in \mathfrak{C}_{i,j}} |m(\sqrt{\lambda_{\ell_1,\ell_1'}}, \sqrt{\lambda_{\ell_2,\ell_2'}})|^2 (1- \beta_{\ell_1,\ell_1'}- \beta_{\ell_2,\ell_2'} + \beta_{\ell_1,\ell_1'}\beta_{\ell_2,\ell_2'}) \|Y_z^{\ell_1,\ell_1'}\|_{L^2}^2 \|Y_z^{\ell_2,\ell_2'}\|_{L^2}^2 \\
    &= \sum_{(\ell_1, \ell_1', \ell_2, \ell_2') \in \mathfrak{C}_{i,j}} |m(\sqrt{\lambda_{\ell_1,\ell_1'}}, \sqrt{\lambda_{\ell_2,\ell_2'}})|^2 \eta(\ell_1,\ell_1')^2 \eta(\ell_2,\ell_2')^2 \|Y_z^{\ell_1,\ell_1'}\|_{L^2}^2 \|Y_z^{\ell_2,\ell_2'}\|_{L^2}^2 \\
    &= \|\mathfrak{M}^{1,1} f_{i,j}\|_{L^2(d\sigma(w),d\sigma(u))}^2 .
\end{align*}
    This completes the proof of the required claim \eqref{To prove the claim first}.
\end{proof}

\begin{proposition}
\label{Proposition: Bilinear weighted Plancherel with weight and s operator}
Let $m : \mathbb{R}^2 \to \mathbb{C}$ such that $\supp{m} \subseteq [0,N]^2$. For $s \in [0,1]$ and $\kappa>0$, we define the following bilinear kernel as
\begin{align}
\label{Definition of bilinear kernel with s operator}
    \mathcal{K}_{m(\sqrt{\mathcal{L}_1}, \sqrt{\mathcal{L}_2})}^{bi, \chi_s}(z, w, u) &:= \sum_{\substack{\ell_1,\ell_1', \ell_2,\ell_2' =0\\ |\ell_1-\ell_1'|\leq \kappa N^s, |\ell_2-\ell_2'|\leq \kappa N^s }}^{\infty} m(\sqrt{\lambda_{\ell_1, \ell_1'}}, \sqrt{\lambda_{\ell_2, \ell_2'}}) \overline{Y_z^{\ell_1,\ell_1'}(w)}\,  \overline{Y_z^{\ell_2,\ell_2'}(u)} .
\end{align}
Then for all $z \in \mathbb{S}$ and $0\leq \alpha_1, \alpha_2 <1/2$ we have
\begin{align*}
    & \Bigg( \int_{\mathbb{S}} \int_{\mathbb{S}} |\mathcal{K}_{m(\sqrt{\mathcal{L}_1}, \sqrt{\mathcal{L}_2})}^{bi, \chi_s}(z,w,u) (1+N \varpi(z,w))^{\alpha_1} (1+N \varpi(z,u))^{\alpha_2}|^2 \, d\sigma(w) \, d\sigma(u) \Bigg)^{1/2} \\
    &\nonumber \leq C N^{Q-1+s} \|m(N\cdot, N \cdot)\|_{N,N,2} .
\end{align*}
    
\end{proposition}

\begin{proof}
For the proof of this proposition one can proceed analogously as of Proposition \ref{Proposition: Bilinear weighted Plancherel with weight}. For $\gamma_1, \gamma_2 \geq 0$ and $s \in [0,1]$, we set
\begin{align*}
    \mathfrak{M}^{\gamma_1, \gamma_2}_s f(w,u) &:= \sum_{\substack{\ell_1,\ell_1', \ell_2,\ell_2' =0\\ |\ell_1-\ell_1'|\leq \kappa N^s, |\ell_2-\ell_2'|\leq \kappa N^s }}^{\infty} \eta(\ell_1,\ell_1')^{\gamma_1} \eta(\ell_2,\ell_2')^{\gamma_2} m(\sqrt{\lambda_{\ell_1,\ell_1'}}, \sqrt{\lambda_{\ell_2,\ell_2'}}) \overline{Y_z^{\ell_1,\ell_1'}(w)}\,  \overline{Y_z^{\ell_2,\ell_2'}(u)} . 
\end{align*}
Also from \eqref{Definition of Hi set} recall $H_i = \{(\ell, \ell') \in \mathbb{N}_0^2 : (i-1)^2 \leq \lambda_{\ell, \ell'} \leq i^2 \}$. Then in place of \eqref{Use of Zonal bound in bilinear weighted}, here we have
\begin{align}
\label{M s alpha 1 alpha 2 estimate}
    & \|\mathfrak{M}^{\alpha_1,\alpha_2}_s f\|_{L^2(d\sigma(w),d\sigma(u))}^2 \\
    &\nonumber \leq \sum_{i,j=1}^{N} \max \{|m(\sqrt{\lambda_{\ell_1,\ell_1'}}, \sqrt{\lambda_{\ell_2,\ell_2'}})|^2 : (\ell_1,\ell_1') \in H_i , (\ell_2,\ell_2')\in H_j, |\ell_1-\ell_1'| \leq \kappa N^s, |\ell_2-\ell_2'| \leq \kappa N^s \} \\
    &\nonumber \hspace{3cm} \times \Bigg(\sum_{\substack{(\ell_1,\ell_1') \in H_i \\ |\ell_1-\ell_1'| \leq \kappa N^s}} \eta(\ell_1,\ell_1')^{2\alpha_1} \|Y_z^{\ell_1,\ell_1'}\|_{L^2}^2 \Bigg) \Bigg(\sum_{\substack{(\ell_2,\ell_2') \in H_j \\ |\ell_2-\ell_2'| \leq \kappa N^s }}\eta(\ell_2,\ell_2')^{2\alpha_2} \|Y_z^{\ell_2,\ell_2'}\|_{L^2}^2 \Bigg) .
\end{align}
From \eqref{Definition of eta equals to 1 - beta}, we have the following estimate (see also \cite[p. 627]{Cowling_Kilima_Sikora_sublaplacian_2011})
\begin{align*}
    \eta(\ell,\ell')^{2} &\leq C \frac{(\ell+a)(\ell' +a)}{(\ell+a +\ell'+a)} \quad \quad \text{where} \quad a=\frac{n-1}{2} .
\end{align*}
Hence, in view of the above estimate, estimate \eqref{Estimate of the L2 norm} and the fact \eqref{Product of l and l prime equivalent to i^2 case} for $0\leq \gamma <1/2$ we obtain
\begin{align*}
    \sum_{\substack{(\ell,\ell') \in H_i \\ |\ell-\ell'| \leq \kappa N^s}} \eta(\ell,\ell')^{2\gamma} \|Y_z^{\ell,\ell'}\|_{L^2}^2 &\lesssim i^{2(\gamma+n-2)} \sum_{\substack{(\ell,\ell') \in H_i' \\ |\ell-\ell'| \leq \kappa N^s}} (\ell'+a)^{1-2\gamma} ,
\end{align*}
where $H_i' = \{(\ell, \ell') \in H_i : \ell < \ell' \}$.

Consequently, using the estimate \eqref{Required to prove in New operator} for $1 \leq i \leq N$ and $0\leq \gamma <1/2$ yield
\begin{align*}
    \sum_{\substack{(\ell,\ell') \in H_i \\ |\ell-\ell'| \leq \kappa N^s}} \eta(\ell,\ell')^{2\gamma} \|Y_z^{\ell,\ell'}\|_{L^2}^2 &\lesssim N^{2(\gamma+n-2)} N^{2+s-4\gamma} \lesssim N^{Q-2+s-2 \gamma} .
\end{align*}
Therefore with the help of above estimate for $0\leq \alpha_1, \alpha_2 <1/2$, from \eqref{M s alpha 1 alpha 2 estimate} we have
\begin{align*}
    & \|\mathfrak{M}^{\alpha_1,\alpha_2}_s f\|_{L^2(d\sigma(w),d\sigma(u))}^2 \\
    &\lesssim N^{2(Q-2+s -\alpha_1 - \alpha_2)} \sum_{i,j=1}^{N} \max \{|m(\sqrt{\lambda_{\ell_1,\ell_1'}}, \sqrt{\lambda_{\ell_2,\ell_2'}})|^2 : (\ell_1,\ell_1') \in H_i , (\ell_2,\ell_2')\in H_j \} .
\end{align*}
Now proceeding similarly as in \eqref{Bilinear final kernel estimate with weight} of Proposition \ref{Proposition: Bilinear weighted Plancherel with weight} we get the required estimate.    
\end{proof}

\section{Proof of theorem \ref{Theorem: Bilinear Bochner-Riesz Main theorem}}
\label{Section: Proof of first main theorem}
This section concerns about the proof of our first main Theorem \ref{Theorem: Bilinear Bochner-Riesz Main theorem}. The main tools required in the proofs are \emph{restriction type estimates} and \emph{weighted Plancherel estimates} established in Section \ref{Section: Restriction type estimates} and \ref{Section: Weighted Plancherel estimates}.

Recall that, we have
\begin{align*}
    \mathcal{B}_R^{\alpha}(f,g)(z) & = \sum_{\substack{\ell_1,\ell_1', \ell_2,\ell_2' =0}}^{\infty} \left(1- \frac{\sqrt{\lambda_{\ell_1, \ell_1'}} + \sqrt{\lambda_{\ell_2, \ell_2'}}}{R} \right)_{+}^{\alpha} \pi_{\ell_1, \ell_1'} f(z) \,  \pi_{\ell_2, \ell_2'} g(z) .
\end{align*}
If $\ell_1=\ell_1'=\ell_2=\ell_2'=0$, then $\lambda_{\ell_1, \ell_1'}=\lambda_{\ell_2, \ell_2'}=0$. In this case applying H\"older's inequality for $1/p_1+1/p_2=1/p$ yields
\begin{align*}
    \|\pi_{0, 0} f \,  \pi_{0, 0} g\|_{L^p} &\leq C \|f\|_{L^{p_1}} \|g\|_{L^{p_2}} .
\end{align*}
Therefore enough to consider the case when at least one of $\ell_i$ or $\ell_i'$ for $i=1,2$ is not zero. Note that in this case, we have $\lambda_{\ell, \ell'} \geq 1$, which forces $R \geq 1$. Henceforth, in the rest of the proof we will always assume $R \geq 1$ and for brevity we continue with the same notation $\mathcal{B}_R^{\alpha}$.

Now first we decompose the bilinear Bochner-Riesz multiplier $\Big(1- \tfrac{\sqrt{\lambda_{\ell_1, \ell_1'}} + \sqrt{\lambda_{\ell_2, \ell_2'}}}{R} \Big)_{+}^{\alpha}$ dyadically as follows. Fix $\alpha>0$, and let $\Psi$ be a function in $C_c^{\infty}(1/2,2)$ satisfying $t^{\alpha} = \sum_{j \in \mathbb{Z}} 2^{-j\alpha} \Psi(2^j t)$  for $ t>0  $. Then we have
\begin{align}
\label{Partition of unity for 1-t}
    (1-t)_{+}^{\alpha} = \sum_{j\in \mathbb{Z}} 2^{-j\alpha} \Psi(2^j (1-t)), \quad \quad \text{for} \quad  t\in [0,1) .
\end{align}
Since $\supp{\Psi} \subseteq (\frac{1}{2}, 2)$, for $t\geq 0$, notice that the term corresponding to $j<0$ in the above summand vanishes. Thus,  it is enough to restrict to the case when $j \geq 0$. Therefore from \eqref{Partition of unity for 1-t} we can decompose
\begin{align*}
    \left(1- \frac{\sqrt{\lambda_1} + \sqrt{\lambda_2}}{R} \right)_{+}^{\alpha} &= \sum_{j=0}^{\infty} 2^{-j\alpha} \Psi_j\left( \frac{\sqrt{\lambda_1} }{R}, \frac{\sqrt{\lambda_2} }{R} \right) ,
\end{align*}
where $\Psi_j(\eta_1, \eta_2) :=\Psi(2^j(1-\eta_1-\eta_2))$ for $j \geq 0$.

Consequently, we can write
\begin{align}
\label{Writing B alpha into Bo and Bj}
    \mathcal{B}^{\alpha}_R(f,g) (z) &= \sum_{j = 0} ^{\infty} 2^{-j\alpha} \mathcal{B}_{R,j}(f,g)(z) ,
\end{align}
where for $j \geq 0$ we have
\begin{align}
\label{Definition of dyadic bilinear Bochner-Riesz}
    \mathcal{B}_{R,j}(f,g)(z) &= \sum_{\substack{\ell_1,\ell_1', \ell_2,\ell_2' =0}}^{\infty} \Psi_j\left( \tfrac{\sqrt{\lambda_{\ell_1, \ell_1'}}}{R}, \tfrac{\sqrt{\lambda_{\ell_2, \ell_2'}}}{R} \right) \pi_{\ell_1, \ell_1'} f(z) \,  \pi_{\ell_2, \ell_2'} g(z) .
\end{align}

Therefore, in order to prove Theorem \ref{Theorem: Bilinear Bochner-Riesz Main theorem}, we have to show that, for $1/p=1/p_1+1/p_2$ with $1\leq p_1, p_2 \leq \infty$ whenever $\alpha>\alpha(p_1, p_2)$ we have
\begin{align}
\label{Estimate for B R 1 case}
    \|\mathcal{B}_{R}^{\alpha}(f,g)\|_{L^p(\mathbb{S})} &\leq C \|f\|_{L^{p_1}(\mathbb{S})} \|g\|_{L^{p_2}(\mathbb{S})} .
\end{align}

\medskip

In order to estimate $\mathcal{B}_{R}^{\alpha}$ we prove the required estimate \eqref{Estimate for B R 1 case} at some particular points $(p_1, p_2, p)=(1,1,1/2)$, $(1,2,2/3)$, $(2,2,1)$, $(\infty, \infty, \infty)$, $(2, \infty, 2)$ and $(1, \infty, 1)$. Note that interchanging $f$ and $g$ one easily get the estimate for the other points $(p_1, p_2, p)=(2,1,2/3)$, $(\infty, 2, 2)$ and $(\infty, 1, 1)$. Finally, using the bilinear interpolation between these nine points we obtain the required estimate \eqref{Estimate for B R 1 case}, see \cite[Section 4.3]{Bernicot_Grafakos_Song_Yan_Bilinear_Bochner_Riesz_2015}.

First let us start with the estimate of \eqref{Estimate for B R 1 case} for the points $(p_1, p_2, p)=(1,1,1/2)$, $(1,2,2/3)$, $(2,2,1)$.

\medskip
\noindent \textbf{Estimate of \eqref{Estimate for B R 1 case} at $(p_1, p_2, p)=(1,1,1/2)$, $(1,2,2/3)$, $(2,2,1)$:}
In view of \eqref{Writing B alpha into Bo and Bj}, to prove \eqref{Estimate for B R 1 case} for these four points it is enough to show, for $j\geq 0$ whenever $\alpha>\alpha(p_1, p_2)$ we have
\begin{align}
\label{Required to proof after dyadic decomposition}
    \|\mathcal{B}_{R,j}(f,g)\|_{L^p(\mathbb{S})} &\leq C 2^{j \alpha(p_1, p_2)+j\epsilon} \|f\|_{L^{p_1}(\mathbb{S})} \|g\|_{L^{p_2}(\mathbb{S})} ,
\end{align}
where $\epsilon>0$ is sufficiently small.

\subsection{Proof of (\ref{Required to proof after dyadic decomposition}) at \texorpdfstring{$(p_1, p_2, p)=(1,1,1/2)$}{}}

Let us set $\eta_1 = \tfrac{\sqrt{\lambda_{\ell_1, \ell_1'}}}{R}$ and $\eta_2 = \tfrac{\sqrt{\lambda_{\ell_2, \ell_2'}}}{R}$. Note that we have $0\leq \eta_1, \eta_2 \leq 1$. Let us fix $\eta_1$ and view $\Psi_j(\eta_1, \cdot)$ as a function of $\eta_2$ supported in $[0,1]$. For each fixed $\eta_1$, the function $\eta_2 \to \Psi_j(\eta_{1}, \eta_2)$ is a smooth function on $\mathbb{R}$ supported in $(-\infty, 1)$. Choose a compactly supported smooth function $\zeta$ such that it is $1$ on $[-1,1]$ and $0$ outside  $[-2,2]$. Consequently, the function $\eta_2 \to \zeta(\eta_2) \Psi_j(\eta_{1}, \eta_2)$ become a smooth function on $\mathbb{R}$ which is supported on $(-2,2)$ and coincides with the function $\Psi_j(\eta_{1}, \eta_2)$ as a function of $\eta_2$ on $[0, \infty)$. Then we extend the function $\eta_2 \to \zeta(\eta_2) \Psi_j(\eta_{1}, \eta_2)$ periodically as a $4$-periodic function to the whole of $\mathbb{R}$. Therefore, it admits a Fourier series expansion:
\begin{align}
\label{Fourier series decomposition with chi tilde}
    \zeta(\eta_2) \Psi_j(\eta_1, \eta_{2}) &= \sum_{l \in \mathbb{Z}} \phi_{j,l}(\eta_1) e^{i \pi l \eta_{2}/2} ,
\end{align}
where for $l \in \mathbb{Z}$, the Fourier coefficients $\phi_{j,l}$ are given by
\begin{align}
\label{Definition of Fourier coefficient}
    \phi_{j,l}(\eta_1) &= \frac{1}{4} \int_{-2}^2 \zeta(\eta_2) \Psi_j(\eta_1, \eta_{2}) e^{-i \pi l \eta_{2}/2}\, d\eta_{2} .
\end{align}
Now using the support of $\Psi_j(\eta_1, \cdot)$ from \eqref{Definition of Fourier coefficient} we have the following estimate of $\phi_{j,l}$:
\begin{align}
\label{Convergence of sum of l using phi}
    \|\phi_{j,l}\|_{L^p_s([0,1])} &\leq C \, (1 + |l|)^{-(1 + \beta)} 2^{j (\beta+s)} \quad \text{for all} \quad \beta, s \geq 0, \ 1\leq p\leq \infty .
\end{align}
Recall that the function $\eta_2 \mapsto \zeta(\eta_2) \Psi_j(\eta_{1}, \eta_2)$ agrees with the function $\Psi_j(\eta_{1}, \eta_2)$ on $[0, \infty)$, therefore from \eqref{Fourier series decomposition with chi tilde} for $\eta_1, \eta_2 \in [0, \infty)$ we have
\begin{align}
\label{Fourier series decomposition}
    \Psi_j(\eta_1, \eta_{2}) &= \sum_{l \in \mathbb{Z}} \phi_{j,l}(\eta_1) e^{i \pi l \eta_{2}/2} \zeta(\eta_2) =: \sum_{l \in \mathbb{Z}} \phi_{j,l}(\eta_1) \psi_l(\eta_2) ,
\end{align}
where $\psi_l(\eta_2)= e^{i \pi l \eta_{2}/2} \zeta(\eta_2)$.

Hence in view of the above decomposition we write
\begin{align}
\label{Decomposition of Bochner-Riesz into Fourier series}
    \mathcal{B}_{R,j}(f,g)(z) &= \sum_{l \in \mathbb{Z}} \left\{ \sum_{\ell_1,\ell_1' =0}^{\infty} \phi_{j,l}\left(\frac{\sqrt{\lambda_{\ell_1, \ell_1'}}}{R}\right) \pi_{\ell_1, \ell_1'} f(z) \right\} \left\{ \sum_{\ell_2,\ell_2' =0}^{\infty} \psi_l\left(\frac{\sqrt{\lambda_{\ell_2, \ell_2'}}}{R}\right) \pi_{\ell_2, \ell_2'} g(z) \right\} \\
    &\nonumber=: \sum_{l \in \mathbb{Z}} \phi_{j,l}(R^{-1} \sqrt{\mathcal{L}})f(z)\, \psi_{l}(R^{-1} \sqrt{\mathcal{L}})g(z) .
\end{align}
Before proceed further note for $0<p<1$, we have the following inequality
\begin{align}
\label{Triangle inequality for p less than 1 case}
    \|f+g\|_{L^p}^p &\leq \|f\|_{L^p}^p + \|g\|_{L^p}^p .
\end{align}
Using the above inequality \eqref{Triangle inequality for p less than 1 case} and H\"older's inequality we see that
\begin{align}
\label{Use of triangle for p less 1 and Holder}
    \|\mathcal{B}_{R,j}(f,g)\|_{L^{1/2}} &\leq \left(\sum_{l \in \mathbb{Z}} \|\phi_{j,l}(R^{-1} \sqrt{\mathcal{L}})f\|_{L^1}^{1/2} \, \|\psi_{l}(R^{-1} \sqrt{\mathcal{L}})g\|_{L^1}^{1/2} \right)^2 .
\end{align}
Let $\mathcal{K}_{\phi_{j,l}(R^{-1} \sqrt{\mathcal{L}})}$ denote the integral kernel of the operator $\phi_{j,l}(R^{-1} \sqrt{\mathcal{L}})$, that is,
\begin{align*}
    \phi_{j,l}(R^{-1} \sqrt{\mathcal{L}})f(z) &= \int_{\mathbb{S}} \mathcal{K}_{\phi_{j,l}(R^{-1} \sqrt{\mathcal{L}})}(z, w) f(w) \, d\sigma(w) .
\end{align*}
Note that $\supp \phi_{j,l}(R^{-1} \cdot) \subseteq [0, R]$. Therefore, using Proposition \ref{Prop: Linear L1 norm of kernel} for $s>d/2$ we obtain
\begin{align}
\label{L1 norm estimate of phijl operator}
    \|\phi_{j,l}(R^{-1} \sqrt{\mathcal{L}})f\|_{L^1} &\leq \|f\|_{L^1} \left(\sup_{w \in \mathbb{S}} \int_{\mathbb{S}} |\mathcal{K}_{\phi_{j,l}(R^{-1} \sqrt{\mathcal{L}})}(z, w)| \, d\sigma(z) \right)  \leq C \|\phi_{j,l}\|_{L^2_{s}} \|f\|_{L^1} ,
\end{align}
and similarly for $s>d/2$ we also get
\begin{align}
\label{L1 norm of psil}
    \|\psi_{l}(R^{-1} \sqrt{\mathcal{L}})g\|_{L^1} &\leq C \|\psi_{l}\|_{L^2_s} \|g\|_{L^1} .
\end{align}
On the other hand, from \eqref{Convergence of sum of l using phi} for any $\beta>0$ we see that
\begin{align}
\label{Sobolev norm estimate of multiplier}
    \|\phi_{j,l}\|_{L^2_{s}} \leq C (1+|l|)^{-(1+\beta)} 2^{j(\beta+s)} ,
\end{align}
and since $\psi_l(\eta_2)= e^{i \pi l \eta_{2}/2} \zeta(\eta_2)$ we also have
\begin{align}
\label{Sobolev norm for psil}
    \|\psi_{l}\|_{L^2_s} \leq C (1+|l|)^{s} .
\end{align}
Finally, combining all the above estimates \eqref{L1 norm estimate of phijl operator} - \eqref{Sobolev norm for psil} and plugging them into \eqref{Use of triangle for p less 1 and Holder} for $s>d/2$ we obtain
\begin{align}
\label{Final estimate of BRj at 1,1}
    \|\mathcal{B}_{R,j}(f,g)\|_{L^{1/2}} &\leq C 2^{j(\beta+s)} \|f\|_{L^1} \|g\|_{L^1} \left(\sum_{l \in \mathbb{Z}} (1+|l|)^{(-1-\beta+s)/2} \right)^2 .
\end{align}
Note that the above series in $l$ converges if $\beta>1+s$. Since $s>d/2$, therefore $\beta>d/2+1$, and hence we can choose $\epsilon>0$ sufficiently small so that
\begin{align*}
    \|\mathcal{B}_{R,j}(f,g)\|_{L^{1/2}} &\leq C 2^{j(d+1+\epsilon)} \|f\|_{L^1} \|g\|_{L^1} ,
\end{align*}
which proves the claim \eqref{Required to proof after dyadic decomposition} for $\alpha(1,1)=d+1$.

\subsection{Proof of (\ref{Required to proof after dyadic decomposition}) at \texorpdfstring{$(p_1, p_2, p)=(1,2,2/3)$}{}}
The idea of the proof at this point follows the same line of argument used for the proof of \eqref{Required to proof after dyadic decomposition} at the point $(p_1, p_2, p)=(1,1,1/2)$. As in the previous argument, using \eqref{Triangle inequality for p less than 1 case} and H\"older's inequality, we get 
\begin{align*}
    \|\mathcal{B}_{R,j}(f,g)\|_{L^{2/3}} &\leq \left(\sum_{l \in \mathbb{Z}} \|\phi_{j,l}(R^{-1} \sqrt{\mathcal{L}})f\|_{L^1}^{2/3} \, \|\psi_{l}(R^{-1} \sqrt{\mathcal{L}})g\|_{L^2}^{2/3} \right)^{3/2} .
\end{align*}
Using orthogonality in place of \eqref{L1 norm of psil} and then estimating analogous to \eqref{Final estimate of BRj at 1,1} for $s>d/2$ we obtain
\begin{align*}
    \|\mathcal{B}_{R,j}(f,g)\|_{L^{2/3}} &\leq C 2^{j(\beta+s)} \|f\|_{L^1} \|g\|_{L^2} \left(\sum_{l \in \mathbb{Z}} (1+|l|)^{-2(1+\beta)/3}  \, \|\psi_{l}\|_{L^{\infty}}^{2/3} \right)^{3/2}
\end{align*}
Since $\psi_{l} \in C_c^{\infty}$ we have $\|\psi_{l}\|_{L^{\infty}} < \infty$. Therefore the above sum in $l$ converges provided $\beta>1/2$. Now for $s>d/2$ and $\beta>1/2$, we can choose $\epsilon>0$ such that
\begin{align*}
    \|\mathcal{B}_{R,j}(f,g)\|_{L^{2/3}} &\leq C 2^{j (d+1)/2+j\epsilon} \|f\|_{L^1} \|g\|_{L^2} ,
\end{align*}
which proves the claim \eqref{Required to proof after dyadic decomposition} for $\alpha(1,2)=(d+1)/2$.

\subsection{Proof of (\ref{Required to proof after dyadic decomposition}) at \texorpdfstring{$(p_1, p_2, p)=(2,2,1)$}{}}
Recall \eqref{Decomposition of Bochner-Riesz into Fourier series} and then using H\"older's inequality we get
\begin{align}
\label{Use of Holder at 2,2,1 point}
    \|\mathcal{B}_{R,j}(f,g)\|_{L^{1}} &\leq \sum_{l \in \mathbb{Z}} \|\phi_{j,l}(R^{-1} \sqrt{\mathcal{L}})f\|_{L^2} \, \|\psi_{l}(R^{-1} \sqrt{\mathcal{L}})g\|_{L^2} \\
    &\nonumber \leq C \|f\|_{L^2} \|g\|_{L^2} \sum_{l \in \mathbb{Z}} \|\phi_{j,l}\|_{L^{\infty}} \|\psi_{l}\|_{L^{\infty}} ,
\end{align}
where the last inequality was obtained via the $L^2$-boundedness of the spectral multiplier.

Now using the fact \eqref{Convergence of sum of l using phi} for any $\epsilon>0$, from \eqref{Use of Holder at 2,2,1 point} we get
\begin{align*}
    \|\mathcal{B}_{R,j}(f,g)\|_{L^{1}} &\leq C 2^{j\epsilon} \|f\|_{L^2} \|g\|_{L^2} ,
\end{align*}
which proves the claim \eqref{Required to proof after dyadic decomposition} for $\alpha(2,2)=0$.

\medskip
This completes the proof of \eqref{Estimate for B R 1 case} for the points $(p_1, p_2, p)=(1,1,1/2)$, $(1,2,2/3)$, $(2,2,1)$. Therefore, it remains to prove \eqref{Estimate for B R 1 case} at $(p_1, p_2, p)=(2, \infty, 2)$, $(\infty, \infty, \infty)$ and $(1,\infty,1)$.

\medskip
\noindent \textbf{Estimate of \eqref{Estimate for B R 1 case} at $(p_1, p_2, p)=(2, \infty, 2)$, $(\infty, \infty, \infty)$ and $(1,\infty,1)$:}
Now we turn to the estimate of \eqref{Estimate for B R 1 case} for the points $(p_1, p_2, p)=(2, \infty, 2)$, $(\infty, \infty, \infty)$ and $(1,\infty,1)$. Let us first write
\begin{align}
\label{Decomposition of BR1 into main and error}
    \mathcal{B}_{R}^{\alpha}(f,g) (z) &= \sum_{j = 0} ^{\lfloor\log_2 R \rfloor} 2^{-j\alpha} \mathcal{B}_{R,j}(f,g)(z) + \sum_{j = \lfloor\log_2 R\rfloor+1}^{\infty} 2^{-j\alpha} \mathcal{B}_{R,j}(f,g)(z) \\
    &\nonumber =: \sum_{j = 0} ^{\lfloor\log_2 R \rfloor} 2^{-j\alpha} \mathcal{B}_{R,j}(f,g)(z) + \mathcal{E}_{R}^{\alpha}(f,g)(z) .
\end{align}
Hence, to prove \eqref{Estimate for B R 1 case} at $(p_1, p_2, p)=(2, \infty, 2)$, $(\infty, \infty, \infty)$ and $(1,\infty,1)$, it is enough to show, whenever $\alpha>\alpha(p_1, p_2)$ we have
\begin{align}
\label{To prove for Error at two points}
    \|\mathcal{E}_{R}^{\alpha}(f,g)\|_{L^p(\mathbb{S})} &\leq C \|f\|_{L^{p_1}(\mathbb{S})} \|g\|_{L^{p_2}(\mathbb{S})} ,
\end{align}
and for $0\leq j \leq \lfloor\log_2 R \rfloor$ we have
\begin{align}
\label{To prove main at two points}
    \|\mathcal{B}_{R,j}(f,g)\|_{L^p(\mathbb{S})} &\leq C 2^{j \alpha(p_1, p_2)+j\epsilon} \|f\|_{L^{p_1}(\mathbb{S})} \|g\|_{L^{p_2}(\mathbb{S})} ,
\end{align}
for some $\epsilon>0$ sufficiently small.

Let us start with the estimate of \eqref{To prove main at two points}. In the subsequent steps we further reduce the required estimate \eqref{To prove main at two points} to the estimate \eqref{New claim for bilinear} given below.

\medskip
\noindent \textbf{Reduction of \eqref{New claim for bilinear} from \eqref{To prove main at two points}:}
In this case we need to further decompose the kernel of $\mathcal{B}^{\alpha}_{R,j}$ and the input functions $f$ and $g$. Let $\mathcal{K}_{R,j}$ denote the kernel corresponding to the operator $\mathcal{B}_{R,j}$. Then for $\gamma>0$, we decompose the kernel $\mathcal{K}_{R,j}$ further into following four parts depending on whether $w$ and $u$ lie inside or outside of the ball of radius $\frac{2^{j(1+\gamma)}}{R}$ centered at $z$, as follows
\begin{align*}
    \mathcal{K}_{R,j}^{in,in}(z,w,u) &= \mathcal{K}_{R,j}(z,w,u) \chi_{B(z, \tfrac{2^{j(1+\gamma)}}{R})}(w) \chi_{B(z, \tfrac{2^{j(1+\gamma)}}{R})}(u) \\
    \mathcal{K}_{R,j}^{in,out}(z,w,u) &= \mathcal{K}_{R,j}(z,w,u) \chi_{B(z, \tfrac{2^{j(1+\gamma)}}{R})}(w) \chi_{B(z, \tfrac{2^{j(1+\gamma)}}{R})^c}(u) \\
    \mathcal{K}_{R,j}^{out,in}(z,w,u) &= \mathcal{K}_{R,j}(z,w,u) \chi_{B(z, \tfrac{2^{j(1+\gamma)}}{R})^c}(w) \chi_{B(z, \tfrac{2^{j(1+\gamma)}}{R})}(u) \\
    \mathcal{K}_{R,j}^{out,out}(z,w,u) &= \mathcal{K}_{R,j}(z,w,u) \chi_{B(z, \tfrac{2^{j(1+\gamma)}}{R})^c}(w) \chi_{B(z, \tfrac{2^{j(1+\gamma)}}{R})^c}(u) .
\end{align*}

Let $\mathcal{B}_{R,j}^{in,in}$, $\mathcal{B}_{R,j}^{in,out}$, $\mathcal{B}_{R,j}^{out,in}$ and $\mathcal{B}_{R,j}^{out,out}$ denote the bilinear multiplier associate to the kernel $\mathcal{K}_{R,j}^{in,in}$, $\mathcal{K}_{R,j}^{in,out}$, $\mathcal{K}_{R,j}^{out,in}$ and $\mathcal{K}_{R,j}^{out,out}$ respectively. Therefore it is enough to prove \eqref{To prove main at two points} for $\mathcal{B}_{R,j}^{in,in}$, $\mathcal{B}_{R,j}^{in,out}$, $\mathcal{B}_{R,j}^{out,in}$ and $\mathcal{B}_{R,j}^{out,out}$.

\medskip
Before proceeding further, we state the following two lemmas, whose proofs are essentially contained in \cite[Lemma 4.1]{Bagchi_Molla_Singh_Bilinear_Bochner_Riesz_Grushin} and \cite[Lemma 2.4]{Bagchi_Molla_Singh_Bilinear_Bochner_Riesz_Grushin} respectively. Since the proofs are follows from the almost same idea, we omit the details here.

\begin{lemma}
\label{Lemma: Pointwise kernel estimate}
For any $\beta_1, \beta_2 \geq 0$ and $\epsilon>0$, we have
\begin{align*}
    |\mathcal{K}_{R,j}(z,w,u) (1+R \varrho(z,w))^{\beta_1} (1+R \varrho(z,u))^{\beta_2} | &\leq C R^{2Q}\, 2^{j(\beta_1 +\beta_2 +1/2+\epsilon)} .
\end{align*}
    
\end{lemma}

\begin{lemma}
\label{lemma: outside distance}
    Let $R> 0$. Then for any $N> Q$ we have
    \begin{align*}
        \int_{\varrho(z, w) \geq R} \frac{|f(w)| }{\big( 1 + \varrho(z, w)\big)^N} \, d\sigma(w) &\leq C R^{-N + Q} \, \mathcal{M}f(z) ,
    \end{align*}
    where $\mathcal{M}$ denotes the Hardy-Littlewood maximal function. 
    
    In particular, taking $f \equiv 1$ for $N>Q$, we get
    \begin{align*}
        \int_{\varrho(z, w) \geq R} \frac{d\sigma(w)}{\big( 1 + \varrho(z, w) \big)^N} \leq C R^{-N + Q}.
    \end{align*}
\end{lemma}

\medskip
\noindent \textbf{Estimate of $\mathcal{B}_{R,j}^{in,out}$, $\mathcal{B}_{R,j}^{out,in}$ and $\mathcal{B}_{R,j}^{out,out}$:}
With the help of Lemma \ref{lemma: outside distance} and the Lemma \ref{Lemma: Pointwise kernel estimate}, one can prove the estimate \eqref{To prove main at two points} with $\mathcal{B}_{R,j}$ replaced by $\mathcal{B}_{R,j}^{in,out}$, $\mathcal{B}_{R,j}^{out,in}$ and $\mathcal{B}_{R,j}^{out,out}$ similarly as in the Subsection 4.3 of \cite{Bagchi_Molla_Singh_Bilinear_Bochner_Riesz_Grushin}.

\medskip
\noindent \textbf{Estimate of $\mathcal{B}_{R,j}^{in, in}$:}
Note that we have to show, for $(p_1, p_2, p)=(2, \infty, 2)$, $(\infty, \infty, \infty)$ and $(1,\infty,1)$ and $0 \leq j \leq \lfloor\log_2 R \rfloor$ whenever $\alpha>\alpha(p_1, p_2)$ we have
\begin{align}
\label{To prove main at two points inside in in part}
    \|\mathcal{B}_{R,j}^{in, in}(f,g)\|_{L^p(\mathbb{S})} &\leq C 2^{j \alpha(p_1, p_2)+j\epsilon} \|f\|_{L^{p_1}(\mathbb{S})} \|g\|_{L^{p_2}(\mathbb{S})} ,
\end{align}
for some $\epsilon>0$ sufficiently small.

In this case we further localize the input functions $f$ and $g$, and consequently, the support of $\mathcal{K}_{R,j}^{in, in}$, which also gives us space localization.

Indeed, let $\gamma>0$ be the same as above. For fixed $j$ and $R$, we can choose points $\{a_m\}$ in $\mathbb{S}$ such that $\varrho(a_{m_1}, a_{m_2})>\tfrac{2^{j(1+\gamma)}}{10 R}$ for $m_1 \neq m_2$ and $\sup_a \inf_m \varrho(a,a_m) \leq \tfrac{2^{j(1+\gamma)}}{10 R}$. Associated to this $\{a_m\}$ we define disjoint sets
\begin{align*}
    S_{m,R}^j &:= B(a_m, \tfrac{2^{j(1+\gamma)}}{10 R}) \setminus \bigcup_{r<m} B(a_r, \tfrac{2^{j(1+\gamma)}}{10 R}) .
\end{align*}
Note that
\begin{align*}
    \supp{\mathcal{K}_{R,j}^{in, in}} \subseteq \mathcal{D}_{R,j} := \left\{(z,w,u) : \varrho(z,w) \leq \tfrac{2^{j(1+\gamma)}}{R}, \  \varrho(z,u) \leq \tfrac{2^{j(1+\gamma)}}{R} \right\} .
\end{align*}
Hence we see that
\begin{align*}
    \mathcal{D}_{R,j} \subseteq \bigcup_{\substack{m,m_1,m_2: \varrho(a_m, a_{m_1}) \leq 2 \cdot \tfrac{2^{j(1+\gamma)}}{R} \\ \varrho(a_m, a_{m_2}) \leq 2 \cdot \tfrac{2^{j(1+\gamma)}}{R} }} S_{m,R}^j \times (S_{m_1,R}^j \times S_{m_2,R}^j) .
\end{align*}
Therefore in view of the above observation we can write
\begin{align*}
    \mathcal{B}_{R,j}^{in, in}(f,g)(z) &= \sum_{m} \sum_{\substack{m_1: \varrho(a_m, a_{m_1}) \leq 2 \cdot \tfrac{2^{j(1+\gamma)}}{R} \\ m_2: \varrho(a_m, a_{m_2}) \leq 2 \cdot \tfrac{2^{j(1+\gamma)}}{R}}} \chi_{S_{m,R}^j}(z) \mathcal{B}_{R,j}^{in, in}(f_{m_1,R}^j,g_{m_2,R}^j)(z) ,
\end{align*}
where $f_{m_1,R}^j(z) = f(z) \chi_{S_{m_1,R}^j}(z)$ and $g_{m_2,R}^j(z) = g(z) \chi_{S_{m_2,R}^j}(z)$.

\medskip
Now let us make the following claim: for
\begin{align*}
    \varrho(a_m, a_{m_1}) \leq 2 \cdot \tfrac{2^{j(1+\gamma)}}{R} \quad \text{and} \quad \varrho(a_m, a_{m_2}) \leq 2 \cdot \tfrac{2^{j(1+\gamma)}}{R} ,
\end{align*}
whenever $\alpha>\alpha(p_1, p_2)$ we have
\begin{align}
\label{Claim: For Bj1 case}
    \|\chi_{S_{m,R}^j} \mathcal{B}_{R,j}^{in, in}(f_{m_1,R}^j,g_{m_2,R}^j)\|_{L^p} &\leq C 2^{j \alpha(p_1, p_2)+j\epsilon} \|f_{m_1,R}^j\|_{L^{p_1}} \|g_{m_2,R}^j\|_{L^{p_2}} ,
\end{align}
for some $\epsilon>0$ sufficiently small and $(p_1, p_2, p)=(2, \infty, 2)$, $(\infty, \infty, \infty)$ and $(1,\infty,1)$.

Assuming the above claim \eqref{Claim: For Bj1 case} for the moment, we can complete the proof of \eqref{To prove main at two points inside in in part}. In fact, first note that for $n \neq m$,
\begin{align*}
    B(a_m, \tfrac{2^{j(1+\gamma)}}{20 R}) \cap B(a_n, \tfrac{2^{j(1+\gamma)}}{20 R}) \neq \emptyset .
\end{align*}
Using the above fact we get the following bounded overlapping property of the balls:
\begin{align}
\label{Bounded overlapping property}
    \sup_{n} \#\{m : \varrho(a_{n}, a_{m}) \leq  2 \cdot \tfrac{2^{j(1+\gamma)}}{R}\} \leq C .
\end{align}
Let $(p_1, p_2, p)=(2, \infty, 2)$, $(\infty, \infty, \infty)$ and $(1,\infty,1)$. Therefore, since the sets $S_{m,R}^j$ are disjoint and using H\"older's inequality along with \eqref{Bounded overlapping property} we obtain
\begin{align}
\label{Proving claim using bounded overlap}
    \|\mathcal{B}_{R,j}^{in, in}(f,g)\|_{L^p}^p &= \Bigg\|\sum_{m} \sum_{\substack{m_1: \varrho(a_m, a_{m_1}) \leq 2 \cdot \tfrac{2^{j(1+\gamma)}}{R} \\ m_2: \varrho(a_m, a_{m_2}) \leq 2 \cdot \tfrac{2^{j(1+\gamma)}}{R}}} \chi_{S_{m,R}^j} \mathcal{B}_{R,j}^{in, in}(f_{m_1,R}^j,g_{m_2,R}^j)\Bigg\|_{L^p}^p \\
    &\nonumber \leq C \sum_{m} \sum_{\substack{m_1: \varrho(a_m, a_{m_1}) \leq 2 \cdot \tfrac{2^{j(1+\gamma)}}{R} \\ m_2: \varrho(a_m, a_{m_2}) \leq 2 \cdot \tfrac{2^{j(1+\gamma)}}{R}}} \Big\|\chi_{S_{m,R}^j} \mathcal{B}_{R,j}^{in, in}(f_{m_1,R}^j,g_{m_2,R}^j)\Big\|_{L^p}^p .
\end{align}
Now applying the claim \eqref{Claim: For Bj1 case}, whenever $\alpha>\alpha(p_1, p_2)$ and $\epsilon>0$, the above expression can be dominated by
\begin{align*}
    C 2^{j \alpha(p_1, p_2) p+j\epsilon p} \sum_{m} \Big(\sum_{m_1: \varrho(a_m, a_{m_1}) \leq 2 \cdot \tfrac{2^{j(1+\gamma)}}{R}} \|f_{m_1,R}^j\|_{L^{p_1}}^p \Big) \Big(\sum_{m_1: \varrho(a_m, a_{m_1}) \leq 2 \cdot \tfrac{2^{j(1+\gamma)}}{R}} \|g_{m_2,R}^j\|_{L^{p_2}}^p \Big) .
\end{align*}
Note that we have $p/p_1 + p/p_2=1$. Therefore using H\"older's inequality the above expression can be further dominated by
\begin{align*}
    & C 2^{j \alpha(p_1, p_2) p+j\epsilon p} \Big(\sum_{m} \sum_{m_1: \varrho(a_m, a_{m_1}) \leq 2 \cdot \tfrac{2^{j(1+\gamma)}}{R}} \|f_{m_1,R}^j\|_{L^{p_1}}^{p_1} \Big)^{\frac{p}{p_1}} \Big(\sum_{m} \sum_{m_1: \varrho(a_m, a_{m_1}) \leq 2 \cdot \tfrac{2^{j(1+\gamma)}}{R}} \|g_{m_2,R}^j\|_{L^{p_2}}^{p_2} \Big)^{\frac{p}{p_2}} \\
    &\leq C 2^{j \alpha(p_1, p_2) p+j\epsilon p} \|f\|_{L^{p_1}}^p \|g\|_{L^{p_2}}^p ,
\end{align*}
where in the last inequality we have used the fact that the sets $\{S_{m_i,R}^j\}_{m_i}$ are disjoint for $i=1,2$ and also \eqref{Bounded overlapping property}.

Hence taking the $1/p$'th power and combining the above estimate and \eqref{Proving claim using bounded overlap} we obtain
\begin{align*}
    \|\mathcal{B}_{R,j}^{in, in}(f,g)\|_{L^p} &\leq C 2^{j \alpha(p_1, p_2)+j\epsilon} \|f\|_{L^{p_1}} \|g\|_{L^{p_2}} .
\end{align*}
This completes the proof of \eqref{To prove main at two points inside in in part}, under the assumption that the claim \eqref{Claim: For Bj1 case} is true.

\medskip
Since the complex sphere $\mathbb{S}$ is transitive under the action of $\mathbb{U}(n)$, let $U_m \in \mathbb{U}(n)$, which takes $e=(1,0, \ldots, 0)$ to $a_m \in \mathbb{S}$. Let us set
\begin{align*}
    S_{e,R}^j = U_m^{-1} S_{m,R}^j , \quad S_{m_1,e}^{R,j} = U_m^{-1} S_{m_1,R}^j , \quad S_{m_2,e}^{R,j} = U_m^{-1} S_{m_2,R}^j .
\end{align*}
Since $\mathbb{S}$ is $\mathbb{U}(n)$-invariant, therefore we have
\begin{align*}
    \|\chi_{S_{m,R}^j} \mathcal{B}_{R,j}^{in, in}(f_{m_1,R}^j,g_{m_2,R}^j)\|_{L^p} &= \|\chi_{S_{e,R}^j} \mathcal{B}_{R,j}^{in, in}(f_{m_1,e}^{R,j},g_{m_2,e}^{R,j})\|_{L^p} ,
\end{align*}
where $f_{m_1,e}^{R,j}(z) = f(z) \chi_{S_{m_1,e}^{R,j}}(z)$ and $g_{m_2,e}^{R,j}(z) = g(z) \chi_{S_{m_2,e}^{R,j}}(z)$.

\medskip
Hence the above claim \eqref{Claim: For Bj1 case} is equivalent to the following claim: for 
\begin{align}
\label{distance condition}
    \varrho(e, a_{m_1}) \leq 2 \cdot \tfrac{2^{j(1+\gamma)}}{R} \quad \text{and} \quad \varrho(e, a_{m_2}) \leq 2 \cdot \tfrac{2^{j(1+\gamma)}}{R}
\end{align}
whenever $\alpha>\alpha(p_1, p_2)$ we have
\begin{align}
\label{Claim: For Bj1 reduced case}
    \|\chi_{S_{e,R}^j} \mathcal{B}_{R,j}^{in, in}(f_{m_1,e}^{R,j},g_{m_2,e}^{R,j})\|_{L^p} &\leq C 2^{j \alpha(p_1, p_2)+j\epsilon} \|f_{m_1,e}^{R,j}\|_{L^{p_1}} \|g_{m_2,e}^{R,j}\|_{L^{p_2}} ,
\end{align}
for $\epsilon>0$ sufficiently small and $(p_1, p_2, p)=(2, \infty, 2)$, $(\infty, \infty, \infty)$ and $(1,\infty,1)$.

\medskip
Note that $S_{e,R}^j \subseteq B(e, \tfrac{2^{j(1+\gamma)}}{10 R})$. Then from \eqref{distance condition} we also get
\begin{align*}
    S_{m_1,e}^{R,j} \subseteq B(e, 3 \cdot \tfrac{2^{j(1+\gamma)}}{ R}) \quad \text{and} \quad S_{m_2,e}^{R,j} \subseteq B(e, 3 \cdot \tfrac{2^{j(1+\gamma)}}{ R}) .
\end{align*}
Therefore all the sets $S_{e,R}^j$, $S_{m_1,e}^{R,j}$ and $S_{m_2,e}^{R,j}$ are contained in $B_{R}^j := B(e, 3 \cdot \tfrac{2^{j(1+\gamma)}}{ R})$. Hence to prove the claim in \eqref{Claim: For Bj1 reduced case}, it is enough to prove: for $\epsilon>0$ sufficiently small, whenever $\alpha>\alpha(p_1, p_2)$
\begin{align}
\label{Reduce it to the same support}
    \|\chi_{B_{R}^j} \mathcal{B}_{R,j}^{in, in}(f_{R,j},g_{R,j})\|_{L^p} &\leq C 2^{j \alpha(p_1, p_2)+j\epsilon} \|f_{R,j}\|_{L^{p_1}} \|g_{R,j}\|_{L^{p_2}},
\end{align}
where for functions $f_{R,j}$, $g_{R,j}$ such that $\supp{f_{R,j}}, \supp{g_{R,j}} \subseteq B_{R}^j$ and $(p_1, p_2, p)=(2, \infty, 2)$, $(\infty, \infty, \infty)$ and $(1,\infty,1)$.

Let us write
\begin{align*}
    \chi_{B_{R}^j}(z) \mathcal{B}_{R,j}^{in, in}(f_{R,j},g_{R,j})(z) &= \chi_{B_{R}^j}(z) \mathcal{B}_{R,j}(f_{R,j},g_{R,j})(z) - \chi_{B_{R}^j}(z) (\mathcal{B}_{R,j}-\mathcal{B}_{R,j}^{in, in})(f_{R,j},g_{R,j})(z) .
\end{align*}
For the second term in the right hand side of above equality, we write
\begin{align*}
    & \chi_{B_{R}^j}(z) (\mathcal{B}_{R,j}-\mathcal{B}_{R,j}^{in, in})(f_{R,j},g_{R,j})(z) \\
    &= \chi_{B_{R}^j}(z) \int_{\mathbb{S}} \int_{\mathbb{S}} (\mathcal{K}_{R,j}-\mathcal{K}_{R,j}^{in, in})(z,w,u) f_{R,j}(w) g_{R,j}(u) \, d\sigma(w) \, d\sigma(u) \\
    &=: \chi_{B_{R}^j}(z) \int_{\mathbb{S}} \int_{\mathbb{S}} \widetilde{K}_{R,j}(z,w,u) f_{R,j}(w) g_{R,j}(u) \, d\sigma(w) \, d\sigma(u) .
\end{align*}
Note $\supp{f_{R,j}}, \supp{g_{R,j}} \subseteq B_{R}^j$, hence whenever $z \in B_{R}^j$, $w \in \supp{f_{R,j}}$ and $u \in \supp{g_{R,j}}$, we get
\begin{align*}
    \supp{\mathcal{K}_{R,j}} \subseteq \left\{(z,w,u) : \varrho(z,w) \leq 6 \cdot \tfrac{2^{j(1+\gamma)}}{R}, \varrho(z,u) \leq 6 \cdot \tfrac{2^{j(1+\gamma)}}{R} \right\} .
\end{align*}
Since $\widetilde{K}_{R,j} = \mathcal{K}_{R,j}-\mathcal{K}_{R,j}^{in, in}$, therefore we have 
\begin{align}
\label{Support of K r j tilde}
    \supp{\widetilde{K}_{R,j}} &\subseteq \left\{(z,w,u) : \tfrac{2^{j(1+\gamma)}}{R}< \varrho(z,w) \leq 6 \cdot \tfrac{2^{j(1+\gamma)}}{R}; \  \varrho(z,u) \leq 6 \cdot \tfrac{2^{j(1+\gamma)}}{R} \right\} \\
    &\nonumber \hspace{1cm} \cup \left\{(z,w,u) : \varrho(z,w) \leq 6 \cdot \tfrac{2^{j(1+\gamma)}}{R};\  \tfrac{2^{j(1+\gamma)}}{R}< \varrho(z,u) \leq 6 \cdot \tfrac{2^{j(1+\gamma)}}{R} \right\} .
\end{align}
Consequently, in the view of the above observation \eqref{Support of K r j tilde}, the $L^p$-norm of $\chi_{B_{R}^j} (\mathcal{B}_{R,j}-\mathcal{B}_{R,j}^{in, in})$ can be estimated similar to the estimates of $\mathcal{B}_{R,j}^{in,out}$, $\mathcal{B}_{R,j}^{out,in}$.

\medskip

Therefore in order to show \eqref{Reduce it to the same support}, it is enough to prove the following: for $\epsilon>0$ sufficiently small and $0 \leq j \leq \lfloor\log_2 R \rfloor$ whenever $\alpha>\alpha(p_1, p_2)$ we have
\begin{align}
\label{New claim for bilinear}
    \|\chi_{B_{R}^j} \mathcal{B}_{R,j}(f_{R,j},g_{R,j})\|_{L^p} &\leq C 2^{j \alpha(p_1, p_2)+j\epsilon} \|f_{R,j}\|_{L^{p_1}} \|g_{R,j}\|_{L^{p_2}} ,
\end{align}
where $\supp{f_{R,j}}, \supp{g_{R,j}} \subseteq B_{R}^j$ and $(p_1, p_2, p)=(2, \infty, 2)$, $(\infty, \infty, \infty)$ and $(1,\infty,1)$.

\medskip
This completes the proof of the reduction of \eqref{New claim for bilinear} from \eqref{To prove main at two points}. Therefore now to prove \eqref{Estimate for B R 1 case} at $(p_1, p_2, p)=(2, \infty, 2)$ and $(1,\infty,1)$, it suffices to show \eqref{To prove for Error at two points} and \eqref{New claim for bilinear} at these two points.

\subsection{Proof of (\ref{To prove for Error at two points}) and (\ref{New claim for bilinear}) at \texorpdfstring{$(p_1, p_2, p)=(2, \infty, 2)$}{}}

Recall from Theorem \ref{Theorem: Bilinear Bochner-Riesz Main theorem} that in this case $\alpha(2, \infty)=(d-1)/2$. Let us first start with the proof of the claim \eqref{New claim for bilinear}.

\subsubsection{Proof of (\ref{New claim for bilinear}) at \texorpdfstring{$(p_1, p_2, p)=(2,\infty,2)$}{}}
\label{Subsection: Proof of main theorem at 2 infinity}
To prove the required estimate at this point, we start by recalling
\begin{align*}
    \mathcal{B}_{R,j}(f_{R,j},g_{R,j})(z) &= \int_{\mathbb{S}} \int_{\mathbb{S}} \mathcal{K}_{R,j}(z, w, u) f_{R,j}(w) g_{R,j}(u) \, d\sigma(w) \, d\sigma(u) ,
\end{align*}
where $\mathcal{K}_{R,j}$ denote the integral kernel corresponding to the bilinear Bochner-Riesz operator $\mathcal{B}_{R,j}$ (see \eqref{Definition of dyadic bilinear Bochner-Riesz}).

Now using H\"older's inequality we get
\begin{align*}
    & |\chi_{B_{R}^j}(z) \mathcal{B}_{R,j}(f_{R,j},g_{R,j})(z)| \\
    &\leq \left(\int_{\mathbb{S}} \frac{|g_{R,j}(u)|^2}{\varpi(z,u)^{2\beta}} \,d\sigma(u) \right)^{1/2} \left(\int_{\mathbb{S}} \varpi(z,u)^{2\beta} \left|\int_{\mathbb{S}} \mathcal{K}_{R,j}(z, w, u) f_{R,j}(w) \, d\sigma(w) \right|^2 \,d\sigma(u) \right)^{1/2} .
\end{align*}
Then
\begin{align}
\label{After taking L2 norm in @ infinity}
    \|\chi_{B_{R}^j} \mathcal{B}_{R,j}(f_{R,j},g_{R,j})\|_{L^2} &\leq \left\{\sup_{z \in \mathbb{S}} \left(\int_{\mathbb{S}} \frac{|g_{R,j}(u)|^2}{\varpi(z,u)^{2\beta}} \,d\sigma(u) \right)^{1/2} \right\} \\
    &\nonumber \hspace{0.5cm} \times \left\{ \int_{\mathbb{S}} \int_{\mathbb{S}} \varpi(z,u)^{2\beta} \left|\int_{\mathbb{S}} \mathcal{K}_{R,j}(z, w, u) f_{R,j}(w) \, d\sigma(w) \right|^2 \,d\sigma(u) \,d\sigma(z)\right\}^{1/2} .
\end{align}
Observe that since $\supp{g_{R,j}} \subseteq B_{R}^j$, an application of Lemma \ref{Lemma: Integral of weight} with $0\leq \beta<1/2$ implies
\begin{align}
\label{Integration of g with weight}
    \sup_{z \in \mathbb{S}} \left(\int_{\mathbb{S}} \frac{|g_{R,j}(u)|^2}{\varpi(z,u)^{2\beta}} \,d\sigma(u) \right)^{1/2} &\leq C \|g_{R,j}\|_{L^{\infty}} \left(\frac{2^{j(1+\gamma)}}{ R} \right)^{Q/2-\beta} .
\end{align}
Now in order to estimate the second factor in the right hand side of \eqref{After taking L2 norm in @ infinity}, first in view of \eqref{Writing projection as inner product} and \eqref{Linear kernel of spectral multiplier} we write
\begin{align}
\label{Rewriting the kernel in different form}
    & \int_{\mathbb{S}} \mathcal{K}_{R,j}(z, w, u) f_{R,j}(w) \, d\sigma(w) \\
    &\nonumber = \int_{\mathbb{S}} \sum_{\substack{\ell_1,\ell_1', \ell_2,\ell_2' =0}}^{\infty} \Psi_j(R^{-1}\sqrt{\lambda_{\ell_1,\ell_1'}}, R^{-1}\sqrt{\lambda_{\ell_2,\ell_2'}}) \overline{Y_z^{(\ell_1,\ell_1')}(w)}\,  \overline{Y_z^{(\ell_2,\ell_2')}(u)} f_{R,j}(w) \, d\sigma(w) \\
    &\nonumber = \sum_{\ell_2,\ell_2' =0}^{\infty} \left(\sum_{\ell_1,\ell_1' =0}^{\infty} \Psi_j(R^{-1}\sqrt{\lambda_{\ell_1,\ell_1'}}, R^{-1}\sqrt{\lambda_{\ell_2,\ell_2'}}) \, \pi_{\ell_1, \ell_1'} f_{R,j}(z) \right) \overline{Y_z^{(\ell_2,\ell_2')}(u)} \\
    &\nonumber =: \sum_{\ell_2,\ell_2' =0}^{\infty} F_{j, z}(R^{-1}\sqrt{\lambda_{\ell_2,\ell_2'}}) \, \overline{Y_z^{(\ell_2,\ell_2')}(u)} \\
    &\nonumber =: \mathcal{K}_{F_{j, z}(R^{-1}\sqrt{\mathcal{L}})}(z, u) .
\end{align}
Note that from Corollary \ref{Corollary: Linear weighted Plancherel} for $0\leq \beta<1/2$ and $\mu>1/2$ we obtain
\begin{align}
\label{Calculating weighted Plancherel}
    \int_{\mathbb{S}} \varpi(z,u)^{2\beta} |\mathcal{K}_{F_{j, z}(R^{-1}\sqrt{\mathcal{L}})}(z, u) |^2 \,d\sigma(u) & \leq C R^{Q-2\beta} \left( \|F_{j, z}\|_{L^2}^2 + R^{-2\mu} \|F_{j, z}\|_{L^2_{\mu}}^2 \right) .
\end{align}
Therefore combining the above two observations \eqref{Rewriting the kernel in different form} and \eqref{Calculating weighted Plancherel} and using Plancherel theorem we get
\begin{align}
\label{Weighted kernel estimate with last f hat}
     &\left\{ \int_{\mathbb{S}} \int_{\mathbb{S}} \varpi(z,u)^{2\beta} \left|\int_{\mathbb{S}} \mathcal{K}_{R,j}(z, w, u) f_{R,j}(w) \, d\sigma(w) \right|^2 \,d\sigma(u) \,d\sigma(z)\right\}^{1/2} \\
     &\nonumber = \left\{ \int_{\mathbb{S}} \left( \int_{\mathbb{S}} \varpi(z,u)^{2\beta} |\mathcal{K}_{F_{j, z}(R^{-1}\sqrt{\mathcal{L}})}(z, u) |^2 \,d\sigma(u) \right) \,d\sigma(z) \right\}^{1/2} \\
     &\nonumber \leq C R^{Q/2-\beta} \left\{ \int_{\mathbb{S}} \left( \int_{\mathbb{R}} |\widehat{F_{j, z}}(\xi_2)|^2 \, d\xi_2 + R^{-2{\mu}} \int_{\mathbb{R}} (1+\xi_2^2)^{\mu} |\widehat{F_{j, z}}(\xi_2)|^2 \, d\xi_2 \right) \,d\sigma(z) \right\}^{1/2} .
\end{align}
For any function $m$ in $\mathbb{R}^2$, let $\mathcal{F}_2 m(\cdot, \xi_2)$ denote the Fourier transform of $m$ at $\xi_2$ with respect to the second variable. Then using the definition of $F_{j, z}$ from \eqref{Rewriting the kernel in different form} and an application of orthogonality yields
\begin{align}
\label{Fjs hat with sobolev}
    & \int_{\mathbb{R}} (1+\xi_2^2)^{\mu} \int_{\mathbb{S}} |\widehat{F_{j, z}}(\xi_2)|^2 \,d\sigma(z) \, d\xi_2 \\
    &\nonumber = \int_{\mathbb{R}} (1+\xi_2^2)^\mu\int_{\mathbb{S}} \left| \sum_{\ell_1,\ell_1' =0}^{\infty} \mathcal{F}_2 \Psi_j(R^{-1}\sqrt{\lambda_{\ell_1,\ell_1'}}, \xi_2) \, \pi_{\ell_1, \ell_1'} f_{R,j}(z) \right|^2 \,d\sigma(z) \, d\xi_2 \\
    &\nonumber = \int_{\mathbb{R}} (1+\xi_2^2)^{\mu} \sum_{\ell_1,\ell_1' =0}^{\infty} |\mathcal{F}_2 \Psi_j(R^{-1}\sqrt{\lambda_{\ell_1,\ell_1'}}, \xi_2)|^2 \|\pi_{\ell_1, \ell_1'} f_{R,j}\|_{L^2}^2 \, d\xi_2 \\
    &\nonumber = \sum_{\ell_1,\ell_1' =0}^{\infty} \left(\int_{\mathbb{R}} (1+\xi_2^2)^{\mu} |\mathcal{F}_2 \Psi_j(R^{-1}\sqrt{\lambda_{\ell_1,\ell_1'}}, \xi_2)|^2 \, d\xi_2 \right) \|\pi_{\ell_1, \ell_1'} f_{R,j}\|_{L^2}^2 \\
    &\nonumber \leq C 2^{2j \mu-j} \|f_{R,j}\|_{L^2}^2 ,
\end{align}
where in the last inequality we have used the following fact 
\begin{align*}
    \int_{\mathbb{R}} (1+\xi_2^2)^{\mu} |\mathcal{F}_2 \Psi_j(R^{-1}\sqrt{\lambda_{\ell_1,\ell_1'}}, \xi_2)|^2 \, d\xi_2 &= \|\Psi_j(R^{-1}\sqrt{\lambda_{\ell_1,\ell_1'}}, \cdot)\|_{L^2_{\mu}}^2 \leq C \, 2^{2j \mu-j} .
\end{align*}
In particular from \eqref{Fjs hat with sobolev} we also get
\begin{align}
\label{Fjs hat estimate without s factor}
    \int_{\mathbb{R}} \int_{\mathbb{S}} |\widehat{F_{j, z}}(\xi_2)|^2 \,d\sigma(z) \, d\xi_2 &\leq C 2^{-j} \|f_{R,j}\|_{L^2}^2 .
\end{align}
Combining both the above estimates \eqref{Fjs hat with sobolev}, \eqref{Fjs hat estimate without s factor} and plugging them into \eqref{Weighted kernel estimate with last f hat} for $0\leq \beta<1/2$ and $\mu>1/2$ we obtain 
\begin{align}
\label{weighted kernel estimate for 2 infinity}
    & \left\{ \int_{\mathbb{S}} \int_{\mathbb{S}} \varpi(z,u)^{2\beta} \left|\int_{\mathbb{S}} \mathcal{K}_{R,j}(z, w, u) f_{R,j}(w) \, d\sigma(w) \right|^2 \,d\sigma(u) \,d\sigma(z)\right\}^{1/2} \\
    &\nonumber\leq C R^{Q/2-\beta} 2^{-j/2} \|f_{R,j}\|_{L^2} (1+ R^{-2\mu} 2^{2j \mu})^{1/2} .
\end{align}

Recall that, in this case we have $0 \leq j \leq \lfloor\log_2 R \rfloor$. Therefore substituting the estimates \eqref{Integration of g with weight} and \eqref{weighted kernel estimate for 2 infinity} into \eqref{After taking L2 norm in @ infinity} we get
\begin{align*}
    \|\chi_{B_{R}^j} \mathcal{B}_{R,j}(f_{R,j},g_{R,j})\|_{L^2} &\leq C \|f_{R,j}\|_{L^2} \|g_{R,j}\|_{L^{\infty}} \left(\frac{2^{j(1+\gamma)}}{ R} \right)^{Q/2-\beta} R^{Q/2-\beta} 2^{-j/2} .
\end{align*}
Since $0\leq \beta<1/2$, choosing $\beta$ close to $1/2$ and $\gamma>0$ sufficiently small yields
\begin{align}
\label{Final estimate of case I in 2 infinity}
    \|\chi_{B_{R}^j} \mathcal{B}_{R,j}(f_{R,j},g_{R,j})\|_{L^2} &\leq C 2^{j (d-1)/2+j \epsilon} \|f_{R,j}\|_{L^2} \|g_{R,j}\|_{L^{\infty}} ,
\end{align}
for some $\epsilon>0$.

This proves the claim \eqref{New claim for bilinear} at $(p_1, p_2, p)=(2, \infty, 2)$ with $\alpha(2, \infty)=(d-1)/2$.

\subsubsection{Proof of \eqref{To prove for Error at two points} at \texorpdfstring{$(p_1, p_2, p)=(2,\infty,2)$}{}}
Recall that 
\begin{align*}
    \mathcal{B}_{R,j}(f,g)(z) &= \int_{\mathbb{S}} \int_{\mathbb{S}} \mathcal{K}_{R,j}(z, w, u) f(w) g(u) \, d\sigma(w) \, d\sigma(u) .
\end{align*}
Applying H\"older's inequality and taking $L^2$-norm we obtain
\begin{align}
\label{Use of Holder in error part}
    \|\mathcal{B}_{R,j}(f, g)\|_{L^2} &\leq \left\{\sup_{z \in \mathbb{S}} \left(\int_{\mathbb{S}} \frac{|g(u)|^2}{\varpi(z,u)^{2\beta}} \,d\sigma(u) \right)^{1/2} \right\} \\
    &\nonumber \hspace{1cm} \times \left\{ \int_{\mathbb{S}} \int_{\mathbb{S}} \varpi(z,u)^{2\beta} \left|\int_{\mathbb{S}} \mathcal{K}_{R,j}(z, w, u) f(w) \, d\sigma(w) \right|^2 \,d\sigma(u) \,d\sigma(z)\right\}^{1/2} .
\end{align}
Note that $\varpi (z,u) \geq 1$, therefore
\begin{align}
\label{Integration of g for error}
    \sup_{z \in \mathbb{S}} \left(\int_{\mathbb{S}} \frac{|g(u)|^2}{\varpi(z,u)^{2\beta}} \,d\sigma(u) \right)^{1/2} &\leq C \|g\|_{L^{\infty}} .
\end{align}
Since in this case $j> \lfloor\log_2 R \rfloor$, from \eqref{weighted kernel estimate for 2 infinity} for $0\leq \beta<1/2$ and $\mu>1/2$ we get
\begin{align}
\label{Weighted Plancherel calculation for j bigger}
    \left\{ \int_{\mathbb{S}} \int_{\mathbb{S}} \varpi(z,u)^{2\beta} \left|\int_{\mathbb{S}} \mathcal{K}_{R,1,j}(z, w, u) f(w) \, d\sigma(w) \right|^2 \,d\sigma(u) \,d\sigma(z)\right\}^{1/2} &\leq C 2^{j(Q/2-\beta)} 2^{-j/2} R^{-\mu} 2^{j \mu} \|f\|_{L^2} .
\end{align}
Then plugging the above two estimates \eqref{Integration of g for error} and \eqref{Weighted Plancherel calculation for j bigger} into \eqref{Use of Holder in error part}, for $\mu>1/2$ and choosing $\beta$ very close to $1/2$ yields
\begin{align*}
    \|\mathcal{B}_{R,j}(f,g)\|_{L^2} &\leq C 2^{j(d-1)/2 +j \mu +j \epsilon_1} R^{-\mu} \|f\|_{L^2} \|g\|_{L^{\infty}} ,  
\end{align*}
for some $\epsilon_1>0$.

Consequently, since $R>1$ and taking $\alpha>\alpha(2, \infty)=(d-1)/2$ we obtain
\begin{align*}
    \|\mathcal{E}_{R}^{\alpha}(f,g)\|_{L^2} &\leq \sum_{j=\lfloor\log_2 R \rfloor+1}^{\infty} 2^{-j \alpha} \|\mathcal{B}_{R,j}(f,g)\|_{L^2} \\
    &\leq C R^{-\mu} \|f\|_{L^2} \|g\|_{L^{\infty}} \sum_{j=\lfloor\log_2 R \rfloor+1}^{\infty} 2^{-j[ \alpha-(d-1)/2-\mu-\epsilon_1]} \\
    &\leq C R^{-[ \alpha-(d-1)/2-\epsilon_1]} \|f\|_{L^2} \|g\|_{L^{\infty}} \\
    &\leq C \|f\|_{L^2} \|g\|_{L^{\infty}} ,
\end{align*}
where we can choose $\epsilon_1>0$ sufficiently small so that $\alpha-(d-1)/2-\epsilon_1>0$.

This completes the proof \eqref{To prove for Error at two points} for the point $(p_1, p_2, p)=(2,\infty, 2)$.

\subsection{Proof of (\ref{To prove for Error at two points}) and (\ref{New claim for bilinear}) at \texorpdfstring{$(p_1, p_2, p)=(\infty, \infty, \infty)$}{}}
At this point we have $\alpha(\infty, \infty)=d-1/2$.

\subsubsection{Proof of (\ref{New claim for bilinear}) at \texorpdfstring{$(p_1, p_2, p)=(\infty,\infty,\infty)$}{}}
Applying H\"older's inequality we have
\begin{align}
\label{Use of twice Holder in infinity}
    |\chi_{B_{R}^j}(z) \mathcal{B}_{R,j}(f_{R,j},g_{R,j})(z)| &\leq \left(\int_{\mathbb{S}} \int_{\mathbb{S}}  \varpi(z,w)^{2\beta_1} \varpi(z,u)^{2\beta_2} | \mathcal{K}_{R,j}(z, w, u)|^2 \,d\sigma(w) \, \,d\sigma(u) \right)^{1/2} \\
    &\nonumber \hspace{1cm} \left(\int_{\mathbb{S}} \frac{|f_{R,j}(w)|^2}{\varpi(z,w)^{2\beta_1}} \,d\sigma(w) \right)^{1/2} \left(\int_{\mathbb{S}} \frac{|g_{R,j}(u)|^2}{\varpi(z,u)^{2\beta_2}} \,d\sigma(u) \right)^{1/2} ,
\end{align}
From Proposition \ref{Proposition: Bilinear weighted Plancherel with weight} we have
\begin{align}
\label{Weighted Plancherel in psi j part for main before discrete norm}
    \left(\int_{\mathbb{S}} \int_{\mathbb{S}}  \varpi(z,w)^{2\beta_1} \varpi(z,u)^{2\beta_2} | \mathcal{K}_{R,j}(z, w, u)|^2 \,d\sigma(w) \, \,d\sigma(u) \right)^{1/2} & \leq C R^{Q-\beta_1-\beta_2} \|\Psi_j(R\cdot, R \cdot)\|_{\lceil R \rceil, \lceil R \rceil, 2} .
\end{align}
Recall that
\begin{align*}
    \Psi_j(R\eta_1,R\eta_2) &= \Psi\bigl(2^j(1-\eta_1-\eta_2)\bigr), \qquad \operatorname{supp}\Psi\subseteq (1/2,2).
\end{align*}
Then $\Psi_j(\eta_1,\eta_2)\neq 0$ gives
\begin{align*}
    1-2^{-j+1} < \eta_1+\eta_2 < 1-2^{-j-1}.
\end{align*}
Thus, $\operatorname{supp}\Psi_j$ is contained in a strip of width $2^{-j}$ around the line $\eta_1+\eta_2=1$.

Let
\begin{align*}
    Q_{k_1,k_2} &= \left[\frac{k_1-1}{\lceil R \rceil},\frac{k_1}{\lceil R \rceil}\right] \times \left[\frac{k_2-1}{\lceil R \rceil},\frac{k_2}{\lceil R \rceil}\right].
\end{align*}
Then
\begin{align*}
    \|\Psi_j(R\cdot,R\cdot)\|_{\lceil R \rceil, \lceil R \rceil,2}^2 &= \frac{1}{\lceil R \rceil^2} \sum_{k_1,k_2=1}^{\lceil R \rceil} \sup_{(\eta_1,\eta_2)\in Q_{k_1,k_2}} \left| \Psi_j(R\eta_1,R\eta_2) \right|^2.
\end{align*}

\noindent \textbf{Case 1: $R\leq 2^j$.}
In this case $2^{-j}\leq R^{-1}$. Thus, the width of the $\supp \Psi_j$ is smaller than the side length $R^{-1}$ of each square $Q_{k_1,k_2}$. Since the $\supp \Psi_j$ has length $O(1)$ inside the unit square, it intersects at most $O(R)$ such squares. Hence
\begin{align*}
    \#\left\{ (k_1,k_2): Q_{k_1,k_2}\cap\operatorname{supp}\Psi_j\neq\emptyset \right\} \leq C R.
\end{align*}
Hence
\begin{align*}
    \|\Psi_j(R\cdot,R\cdot)\|_{\lceil R \rceil ,\lceil R \rceil,2}^2 &\leq C \frac{R}{\lceil R \rceil^2}\|\Psi\|_{L^\infty}^2 \leq C R^{-1}\|\Psi\|_{L^\infty}^2 .
\end{align*}

\noindent \textbf{Case 2: $R\geq 2^j$.}
In this case $2^{-j}\geq R^{-1}$. Thus, the width of the $\supp \Psi_j$ is larger than the side length $R^{-1}$ of each square $Q_{k_1,k_2}$. Since the $\supp \Psi_j$ has length $O(1)$, its area is $O(2^{-j})$. On the other hand, each square $Q_{k_1,k_2}$ has area $R^{-2}$. Consequently,
\begin{align*}
    \#\left\{ (k_1,k_2): Q_{k_1,k_2}\cap\operatorname{supp}\Psi_j\neq\emptyset \right\} &\leq C \frac{2^{-j}}{R^{-2}} = C R^2 2^{-j}.
\end{align*}
Therefore,
\begin{align*}
    \|\Psi_j(R\cdot,R\cdot)\|_{\lceil R \rceil, \lceil R \rceil,2}^2 &\leq C \frac{R^2 2^{-j}}{\lceil R \rceil^2}\|\Psi\|_{L^\infty}^2 \leq C\, 2^{-j}\|\Psi\|_{L^\infty}^2 .
\end{align*}
Combining the two cases, we obtain
\begin{align}
\label{Discrete norm computation}
    \|\Psi_j(R\cdot,R\cdot)\|_{\lceil R \rceil, \lceil R \rceil,2} &\leq C \max\left\{R^{-1/2},2^{-j/2}\right\}\|\Psi\|_{L^\infty} .
\end{align}
Consequently, from \eqref{Weighted Plancherel in psi j part for main before discrete norm} we obtain
\begin{align}
\label{Weighted Plancherel in psi j part for main}
    \left(\int_{\mathbb{S}} \int_{\mathbb{S}}  \varpi(z,w)^{2\beta_1} \varpi(z,u)^{2\beta_2} | \mathcal{K}_{R,j}(z, w, u)|^2 \,d\sigma(w) \, \,d\sigma(u) \right)^{1/2} &\leq C R^{Q-\beta_1-\beta_2} \max\{2^{-j/2}, R^{-1/2}\} .
\end{align}
On the other hand, similarly as in \eqref{Integration of g with weight}, for $0\leq \beta_1 <1/2$ yields
\begin{align}
\label{Integral of wight with f case for main}
    \left(\int_{\mathbb{S}} \frac{|f_{R,j}(w)|^2}{\varpi(z,w)^{2\beta_1}} \,d\sigma(w) \right)^{1/2} & \leq C \|f_{R, j}\|_{L^{\infty}} \left(\tfrac{2^{j(1+\gamma)}}{ R} \right)^{Q/2-\beta_1} ,
\end{align}
and for $0\leq \beta_2 <1/2$ we also get
\begin{align}
\label{Integral of weight with g case for main}
    \left(\int_{\mathbb{S}} \frac{|g_{R,j}(u)|^2}{\varpi(z,u)^{2\beta_2}} \,d\sigma(u) \right)^{1/2} &\leq C \|g_{R, j}\|_{L^{\infty}} \left(\tfrac{2^{j(1+\gamma)}}{ R} \right)^{Q/2-\beta_2} .
\end{align}
Consequently, since $0 \leq j \leq \lfloor\log_2 R \rfloor$ and $R>1$, plugging the above three estimates into \eqref{Use of twice Holder in infinity} for $0 \leq \beta_1, \beta_2<1/2$ we obtain
\begin{align*}
    |\chi_{B_{R}^j}(z) \mathcal{B}_{R,j}(f_{R,j},g_{R,j})(z)| &\leq C R^{Q-\beta_1-\beta_2} 2^{-j/2} \left(\tfrac{2^{j(1+\gamma)}}{ R} \right)^{Q-\beta_1-\beta_2} \|f_{R, j}\|_{L^{\infty}} \|g_{R, j}\|_{L^{\infty}} \\
    &\leq C 2^{-j/2} 2^{j \epsilon} 2^{j (Q-\beta_1-\beta_2)+ j \gamma (Q-\beta_1-\beta_2)} \|f_{R, j}\|_{L^{\infty}} \|g_{R, j}\|_{L^{\infty}} \\
    &\leq C 2^{j (d-1/2+\epsilon)} \|f_{R,j}\|_{L^{\infty}} \|g_{R,j}\|_{L^{\infty}} ,
\end{align*}
for some $\epsilon>0$, provided we choose $\beta_1$ and $\beta_2$ sufficiently close to $1/2$ and $\gamma>0$ sufficiently small.

This proves the claim \eqref{New claim for bilinear} for $\alpha(\infty,\infty)=d-1/2$.

\subsubsection{Proof of \eqref{To prove for Error at two points} at \texorpdfstring{$(p_1, p_2, p)=(\infty,\infty,\infty)$}{}}
Observe that, we have $\varpi(z, w) \geq 1$. Applying H\"older's inequality and proceeding similarly as in \eqref{Use of twice Holder in infinity}, \eqref{Weighted Plancherel in psi j part for main} for $j> \lfloor\log_2 R \rfloor$ and $0 \leq \beta_1, \beta_2<1/2$ yields
\begin{align*}
    |\mathcal{B}_{R,j}(f,g)(z)| &\leq \left(\int_{\mathbb{S}} \int_{\mathbb{S}}  \varpi(z,w)^{2\beta_1} \varpi(z,u)^{2\beta_2} | \mathcal{K}_{R,j}(z, w, u)|^2 \,d\sigma(w) \, \,d\sigma(u) \right)^{1/2} \\
    &\nonumber \hspace{2cm} \left(\int_{\mathbb{S}} |f(w)|^2 \,d\sigma(w) \right)^{1/2} \left(\int_{\mathbb{S}} |g(u)|^2 \,d\sigma(u) \right)^{1/2} \\
    &\leq C R^{Q-\beta_1-\beta_2-1/2} \|f\|_{L^{\infty}} \|g\|_{L^{\infty}} .
\end{align*}
Consequently, since $R>1$ and taking $\alpha>d-1/2$ we obtain
\begin{align}
\label{Final estimate second case infinity}
    \|\mathcal{E}_{R}^{\alpha}(f,g)\|_{L^{\infty}} &\leq \sum_{j=\lfloor\log_2 R \rfloor+1}^{\infty} 2^{-j \alpha} \|\mathcal{B}_{R,j}(f,g)\|_{L^{\infty}} \\
    &\nonumber \leq C R^{Q-\beta_1-\beta_2-1/2} \|f\|_{L^{\infty}} \|g\|_{L^{\infty}} \sum_{j=\lfloor\log_2 R \rfloor+1}^{\infty} 2^{-j \alpha} \\
    &\nonumber \leq C R^{-\alpha+(Q-\beta_1-\beta_2-1/2)} \|f\|_{L^{\infty}} \|g\|_{L^{\infty}} \\
    &\nonumber \leq C \|f\|_{L^{\infty}} \|g\|_{L^{\infty}} ,
\end{align}
provided we choose $\beta_1 +\beta_2$ very close to $1$.

\subsection{Proof of (\ref{To prove for Error at two points}) and (\ref{New claim for bilinear}) at \texorpdfstring{$(p_1, p_2, p)=(1, \infty, 1)$}{}}

Recall that from Theorem \ref{Theorem: Bilinear Bochner-Riesz Main theorem} in this case we have $\alpha(1, \infty)=Q/2$. We will begin with the proof of claim \eqref{New claim for bilinear}.

\subsubsection{Proof of (\ref{New claim for bilinear}) at \texorpdfstring{$(p_1, p_2, p)=(1,\infty,1)$}{}}
Using similar decomposition as in \eqref{Decomposition of Bochner-Riesz into Fourier series} here we have
\begin{align*}
    \mathcal{B}_{R,j}(f_{R,1,j},g_{R,j})(z) &= \sum_{l \in \mathbb{Z}} \phi_{j,l}(R^{-1} \sqrt{\mathcal{L}})f_{R,j}(z)\, \psi_{l}(R^{-1} \sqrt{\mathcal{L}})g_{R,j}(z) .
\end{align*}
Employing H\"older's inequality yields
\begin{align}
\label{Use of Holder in 1 infinity case}
    \|\chi_{B_{R}^j} \mathcal{B}_{R,j}(f_{R,j},g_{R,j})\|_{L^1} &\leq \sum_{l \in \mathbb{Z}} \|\phi_{j,l}(R^{-1} \sqrt{\mathcal{L}})f_{R,j}\|_{L^2} \|\psi_{l}(R^{-1} \sqrt{\mathcal{L}})g_{R,j}\|_{L^2} .
\end{align}
Applying Minkowski's integral inequality and using Corollary \ref{Corollary: Linear weighted Plancherel} with $\gamma=0$ for $\mu>1/2$ we have
\begin{align}
\label{Estimate of f in i inifinity}
    \|\phi_{j,l}(R^{-1} \sqrt{\mathcal{L}})f_{R,j}\|_{L^2} &\leq C R^{Q/2} \left(\|\phi_{j,l}\|_{L^2} + R^{-\mu} \|\phi_{j,l}\|_{L^2_{\mu}} \right) \|f_{R,j}\|_{L^1} \\
    &\nonumber \leq C R^{Q/2} \|\phi_{j,l}\|_{L^{\infty}} (1+ R^{-\mu} 2^{j \mu}) \|f_{R,j}\|_{L^1} .
\end{align}
On the other hand, using orthogonality and since $\supp{g_{R,j}} \subseteq B_{R}^j$, we obtain
\begin{align}
\label{Estimate of g in i infinity}
    \|\psi_{l}(R^{-1} \sqrt{\mathcal{L}})g_{R,j}\|_{L^2} &\leq \|\psi_{l}\|_{L^{\infty}} \|g_{R,j}\|_{L^2} \leq C \left(\frac{2^{j(1+\gamma)}}{ R} \right)^{Q/2} \|g_{R,j}\|_{L^{\infty}} .
\end{align}
Since in this case $0 \leq j \leq \lfloor\log_2 R \rfloor$, plugging \eqref{Estimate of f in i inifinity}, \eqref{Estimate of g in i infinity} into \eqref{Use of Holder in 1 infinity case} and using the fact \eqref{Convergence of sum of l using phi} we get
\begin{align*}
    \|\chi_{B_{R}^j} \mathcal{B}_{R,j}(f_{R,j},g_{R,j})\|_{L^1} &\leq C  R^{Q/2} \left(\frac{2^{j(1+\gamma)}}{ R} \right)^{Q/2} \left( \sum_{l \in \mathbb{Z}} \|\phi_{j,l}\|_{L^{\infty}} \right) \|f_{R,j}\|_{L^1} \|g_{R,j}\|_{L^{\infty}} \\
    &\leq C 2^{j Q/2 + j \epsilon} \|f_{R,j}\|_{L^1} \|g_{R,j}\|_{L^{\infty}} ,
\end{align*}
for some $\epsilon>0$, provided we choose $\gamma$ sufficiently small.

This proves the claim \eqref{New claim for bilinear} at $(p_1, p_2, p)=(1,\infty,1)$ with $\alpha(1, \infty)=Q/2$.

\subsubsection{Proof of \eqref{To prove for Error at two points} at \texorpdfstring{$(p_1, p_2, p)=(1,\infty,1)$}{}}
Recall from \eqref{Decomposition of Bochner-Riesz into Fourier series} that
\begin{align*}
    \mathcal{B}_{R,j}(f,g)(z) &= \sum_{l \in \mathbb{Z}} \phi_{j,l}(R^{-1} \sqrt{\mathcal{L}})f(z)\, \psi_{l}(R^{-1} \sqrt{\mathcal{L}})g(z) .
\end{align*}
In this case $j> \lfloor\log_2 R \rfloor$, hence the estimate \eqref{Estimate of f in i inifinity} provides
\begin{align}
\label{Estimate of phi using Placherel error}
    \|\phi_{j,l}(R^{-1} \sqrt{\mathcal{L}})f\|_{L^2} &\leq C 2^{j Q/2} \|\phi_{j,l}\|_{L^{\infty}} R^{-\mu} 2^{j \mu} \|f\|_{L^1} .
\end{align}
Since $\mathbb{S}$ is compact, we get
\begin{align}
\label{Estimate of psi in error}
     \|\psi_{l}(R^{-1} \sqrt{\mathcal{L}})g\|_{L^2} &\leq \|\psi_{l}\|_{L^{\infty}} \|g\|_{L^2} \leq C \|g\|_{L^{\infty}} .
\end{align}
Using H\"older;s inequality, plugging the above two estimates \eqref{Estimate of phi using Placherel error} and \eqref{Estimate of psi in error} and the fact \eqref{Convergence of sum of l using phi} we obtain
\begin{align}
\label{Estimate of B Rj for error}
    \|\mathcal{B}_{R,j}(f,g)\|_{L^1} &\leq \sum_{l \in \mathbb{Z}} \|\phi_{j,l}(R^{-1} \sqrt{\mathcal{L}})f\|_{L^2} \|\psi_{l}(R^{-1} \sqrt{\mathcal{L}})g\|_{L^2} \\
    &\nonumber \leq C 2^{j Q/2+j\mu} R^{-\mu} \left( \sum_{l \in \mathbb{Z}} \|\phi_{j,l}\|_{L^{\infty}} \right) \|f\|_{L^1} \|g\|_{L^{\infty}} \\
    &\nonumber \leq C 2^{j Q/2+j\mu} R^{-\mu} \|f\|_{L^1} \|g\|_{L^{\infty}}  .
\end{align}
Since $R>1$, and taking $\alpha>Q/2$ the above estimate \eqref{Estimate of B Rj for error} yields
\begin{align*}
    \|\mathcal{E}_{R}^{\alpha}(f,g)\|_{L^1} &\leq \sum_{j=\lfloor\log_2 R \rfloor+1}^{\infty} 2^{-j \alpha} \|\mathcal{B}_{R,j}(f,g)\|_{L^1} \\
    &\leq C R^{-\mu} \|f\|_{L^1} \|g\|_{L^{\infty}} \sum_{j=\lfloor\log_2 R \rfloor+1}^{\infty} 2^{-j[ \alpha-Q/2-\mu]} \\
    &\leq C R^{-( \alpha-Q/2)} \|f\|_{L^1} \|g\|_{L^{\infty}} \\
    &\leq C \|f\|_{L^1} \|g\|_{L^{\infty}} .
\end{align*}
This finishes the claim \eqref{To prove for Error at two points} at $(p_1, p_2, p)=(1,\infty,1)$ with $\alpha(1, \infty)=Q/2$.

Therefore the proof of \eqref{Estimate for B R 1 case} at the remaining two points $(p_1, p_2, p)=(2, \infty, 2)$, $(1,\infty,1)$ is also also completed.

\section{Proof of Theorem \ref{Theorem: Bilinear Bochner-Riesz theorem with restricted f and g}}
\label{Section: Proof of second main theorem}
In this section we will prove Theorem \ref{Theorem: Bilinear Bochner-Riesz theorem with restricted f and g}. Following the proof of Theorem \ref{Theorem: Bilinear Bochner-Riesz Main theorem}, it is enough to prove Theorem \ref{Theorem: Bilinear Bochner-Riesz theorem with restricted f and g} for $(p_1, p_2, p)=(1,1,1/2)$, $(1,2,2/3)$, $(2,2,1)$, $(2, \infty, 2)$ and $(1, \infty, 1)$. Notice that proof at the points $(p_1, p_2, p)=(2, \infty, 2)$ and $(2,2,1)$ can be seen from Theorem \ref{Theorem: Bilinear Bochner-Riesz Main theorem}. Therefore we only need to proof the theorem for the points $(p_1, p_2, p)=(1,1,1/2)$, $(1,2,2/3)$ and $(1, \infty, 1)$. Here we only present the proof at the point $(p_1, p_2, p)=(1,1,1/2)$, since the proof at the remaining two points $(p_1, p_2, p)=(1,2,2/3)$ and $(1, \infty, 1)$ can be deduced using the same idea as of $(1,1,1/2)$ (one can also see \cite[Theorem 1.3]{Bagchi_Molla_Singh_Bilinear_Metivier_2026} for more details). What follows, is the proof for the point $(p_1, p_2, p)=(1,1,1/2)$. It needs to be noted that at this point the proof of Theorem \ref{Theorem: Bilinear Bochner-Riesz theorem with restricted f and g} is more technical in nature and considerably different in comparison to proof at this point in Theorem \ref{Theorem: Bilinear Bochner-Riesz Main theorem} as we end up with an improved smoothness threshold $\alpha$ from $d+1$ to $d$ under the additional assumptions on the input functions.

\subsection{Proof at \texorpdfstring{$(p_1, p_2, p)=(1,1,1/2)$}{}}
Recall that from \eqref{Estimate for B R 1 case}, $f \in \mathfrak{A}_R$ and $g \in \mathfrak{B}_R$ share the same parity, that is, either
\begin{align*}
    \{f = f^{<} \quad \text{and} \quad g= g^{<}\} \quad \quad \text{or} \quad \quad \{f = f^{>} \quad \text{and} \quad g = g^{>}\} ,
\end{align*}
where
\begin{align*}
    f^{<} \in \bigoplus\limits_{\substack{\ell_1, \ell_1'=0 \\ \{|\ell_1-\ell_1'| \leq \kappa_1 R\}}}^{\infty} \mathcal{H}_{\ell_1, \ell_1'} \quad \text{and} \quad f^{>} \in \bigoplus\limits_{\substack{\ell_1, \ell_1'=0 \\ \{|\ell_1-\ell_1'| \geq \kappa_1' R^2\}}}^{\infty} \mathcal{H}_{\ell_1, \ell_1'} ,
\end{align*}
and similarly $g^{<}$ and $g^{>}$ are also defined.

Therefore, it is suffices to estimate $\mathcal{B}_{R}^{\alpha}(f^{<},g^{<})$ and $\mathcal{B}_{R}^{\alpha}(f^{>},g^{>})$. First we will estimate $\mathcal{B}_{R}^{\alpha}(f^{>},g^{>})$.

\medskip
\noindent \textbf{Estimate of $\mathcal{B}_{R}^{\alpha}(f^{>},g^{>})$:}
Let $\Omega_{\kappa_1', R}^{>}$, $\Omega_{\kappa_2', R}^{>}$ be two smooth functions on $\mathbb{R}$ such that
\begin{align*}
    \Omega_{\kappa_1', R}^{>}(\tau_1) = \left\{\begin{array}{ll}
        1, & \quad |\tau_1| \geq \kappa_1' R^2 \\
        0, & \quad |\tau_1| \leq \kappa_1' R^2/2 
    \end{array}\right. \quad \text{and} \quad 
    \Omega_{\kappa_2', R}^{>}(\tau_2) = \left\{\begin{array}{ll}
        1, & \quad |\tau_2| \geq \kappa_2' R^2 \\
        0, & \quad |\tau_2| \leq \kappa_2' R^2/2 
    \end{array}\right.
\end{align*}
and defined smoothly elsewhere.

With the aid of above smooth functions, we may express $\mathcal{B}_{R}^{\alpha}(f^{>},g^{>})$ as 
\begin{align}
\label{Rewriting Bochner-Riesz grater different way}
    & \mathcal{B}_{R}^{\alpha}(f^{>},g^{>})(z) = \sum_{\substack{\ell_1,\ell_1', \ell_2,\ell_2' =0}}^{\infty} \left(1- \tfrac{\sqrt{\lambda_{\ell_1, \ell_1'}} + \sqrt{\lambda_{\ell_2, \ell_2'}}}{R} \right)_{+}^{\alpha} \pi_{\ell_1, \ell_1'} f^{>}(z) \, \pi_{\ell_2, \ell_2'} g^{>}(z) \\
    &\nonumber = \sum_{\substack{\ell_1,\ell_1', \ell_2,\ell_2' =0}}^{\infty} \left(1- \tfrac{\sqrt{\lambda_{\ell_1, \ell_1'}} + \sqrt{\lambda_{\ell_2, \ell_2'}}}{R} \right)_{+}^{\alpha} \Omega_{\kappa_1',R}^{>}(|\ell_1-\ell_1'|) \Omega_{\kappa_2',R}^{>}(|\ell_2-\ell_2'|) \pi_{\ell_1, \ell_1'} f(z) \,  \pi_{\ell_2, \ell_2'} g(z) \\
    &\nonumber =: \mathcal{B}^{\alpha, >, >}_{R}(f,g)(z) .
\end{align}
In the same manner as in \eqref{Decomposition of BR1 into main and error}, we decompose
\begin{align}
\label{Decomposition of Bochner-Riesz greater}
    \mathcal{B}_{R}^{\alpha, >, >}(f,g)(z) &= \sum_{j = 0} ^{\lfloor\log_2 R \rfloor} 2^{-j\alpha} \mathcal{B}_{R,j}^{>, >}(f,g)(z) + \mathcal{E}_{R}^{\alpha, >, >}(f,g)(z) ,
\end{align}
where
\begin{align}
\label{Definition of error partfor restricted theorem}
    \mathcal{E}_{R}^{\alpha, >, >}(f,g)(z) &= \sum_{j = \lfloor\log_2 R \rfloor + 1}^{\infty} 2^{-j\alpha} \mathcal{B}_{R,j}^{>, >}(f,g)(z) ,
\end{align}
and
\begin{align*}
    \mathcal{B}_{R,j}^{>, >}(f,g)(z) &= \sum_{\substack{\ell_1,\ell_1', \ell_2,\ell_2' =0}}^{\infty} \Psi_j\left( \tfrac{\sqrt{\lambda_{\ell_1, \ell_1'}}}{R}, \tfrac{\sqrt{\lambda_{\ell_2, \ell_2'}}}{R} \right) \Omega_{\kappa_1',R}^{>}(|\ell_1-\ell_1'|) \Omega_{\kappa_2',R}^{>}(|\ell_2-\ell_2'|) \pi_{\ell_1, \ell_1'} f(z) \,  \pi_{\ell_2, \ell_2'} g(z) .
\end{align*}
Note that in Theorem \ref{Theorem: Bilinear Bochner-Riesz theorem with restricted f and g} we have $\alpha(1,1)=d$. Therefore enough to show: whenever $\alpha>d$ we have
\begin{align}
\label{To prove for Error for second theorem}
    \|\mathcal{E}_{R}^{\alpha, >, >}(f,g)\|_{L^{1/2}} &\leq C \|f\|_{L^{1}} \|g\|_{L^{1}} ,
\end{align}
and for $0 \leq j \leq \lfloor\log_2 R \rfloor$ we have
\begin{align}
\label{To prove main for second theorem}
    \|\mathcal{B}_{R,j}^{>, >}(f,g)\|_{L^{1/2}} &\leq C 2^{j d+j\epsilon} \|f\|_{L^{1}} \|g\|_{L^{1}} ,
\end{align}
for some $\epsilon>0$ sufficiently small.

\noindent \textbf{Estimate of (\ref{To prove main for second theorem}):}
Using the same reduction as in the derivation of \eqref{New claim for bilinear} from \eqref{To prove main at two points}, in this case in order to establish \eqref{To prove main for second theorem}, it is enough to show that, for $0 \leq j \leq \lfloor\log_2 R \rfloor$, whenever $\alpha>d$ we have
\begin{align*}
    \|\chi_{B_{R}^j} \mathcal{B}_{R,j}^{>, >}(f_{R,j},g_{R,j})\|_{L^{1/2}} &\leq C 2^{j d+j\epsilon} \|f_{R,j}\|_{L^{1}} \|g_{R,j}\|_{L^{1}}, \quad \text{where} \quad \supp{f_{R,j}}, \supp{g_{R,j}} \subseteq B_{R}^j ,
\end{align*}
for some $\epsilon>0$ sufficiently small.

Let us set \begin{align*}
    \phi_{j,l,R}^{>}(\sqrt{\eta_1}, \tau_1) = \phi_{j,l}(R^{-1}\sqrt{\eta_1}) \Omega_{\kappa_1',R}^{>}(\tau_1) \quad \text{and} \quad \psi_{l,R}^{>}(\sqrt{\eta_2}, \tau_2) = \psi_l(R^{-1}\sqrt{\eta_2}) \Omega_{\kappa_2',R}^{>}(\tau_2) .
\end{align*}
Then using the Fourier series expansion \eqref{Fourier series decomposition} we get
\begin{align}
\label{Fourier series decomposition in restricted}
    & \chi_{B_{R}^j}(z) \mathcal{B}_{R,j}^{>, >}(f_{R,j},g_{R,j})(z) \\
    &\nonumber = \chi_{B_{R}^j}(z) \sum_{l \in \mathbb{Z}} \left\{ \sum_{\ell_1,\ell_1' =0}^{\infty} \phi_{j,l}\left(\frac{\sqrt{\lambda_{\ell_1, \ell_1'}}}{R}\right) \Omega_{\kappa_1',R}^{>}(|\ell_1-\ell_1'|) \pi_{\ell_1, \ell_1'} f_{R,j}(z) \right\} \\
    &\nonumber \hspace{5cm} \left\{ \sum_{\ell_2,\ell_2' =0}^{\infty} \psi_l\left(\frac{\sqrt{\lambda_{\ell_2, \ell_2'}}}{R}\right) \Omega_{\kappa_2',R}^{>}(|\ell_2-\ell_2'|) \pi_{\ell_2, \ell_2'} g_{R,j}(z) \right\} \\
    &\nonumber = \chi_{B_{R}^j}(z) \sum_{l \in \mathbb{Z}} \phi_{j,l,R}^{>}(\sqrt{\mathcal{L}}, |T|)f_{R,j}(z)\, \psi_{l,R}^{>}(\sqrt{\mathcal{L}}, |T|)g_{R,j}(z) .
\end{align}
Using the cut-off function $\Theta$  as in \eqref{Definition of Theta} and  for $M_1, M_2 \in \mathbb{Z}$, we set
\begin{align*}
    \phi_{j,l,R}^{>, M_1}(\sqrt{\eta_1}, \tau_1) = \phi_{j,l,R}^{>}(\sqrt{\eta_1}, \tau_1) \Theta(2^{M_1} \tau_1) \quad \text{and} \quad \psi_{l,R}^{>, M_2}(\sqrt{\eta_2}, \tau_2) = \psi_{l,R}^{>}(\sqrt{\eta_2}, \tau_2) \Theta(2^{M_2}\tau_2) .
\end{align*}
This yields the following decomposition:
\begin{align}
\label{Decompositing in terms of theta inside proof}
    \phi_{j,l,R}^{ >}(\sqrt{\mathcal{L}}, |T|)f_{R,j}(z) &= \sum_{M_1=0}^{\infty} \phi_{j,l,R}^{>, M_1}(\sqrt{\mathcal{L}}, |T|)f_{R,j}(z) ,
\end{align}
where $\phi_{j,l,R}^{> M_1}(\sqrt{\mathcal{L}}, |T|)$ is defined as in \eqref{Definition for the truncated operator}, that is,
\begin{align}
\label{Definition of phi j r l M1}
    \phi_{j,l,R}^{>, M_1}(\sqrt{\mathcal{L}}, |T|)f_{R,j}(z) &= \sum_{\ell_1,\ell_1' =0 }^{\infty} \phi_{j,l,R}^{>, M_1}(\sqrt{\lambda_{\ell_1, \ell_1'}}, |\ell_1-\ell_1'|) \pi_{\ell_1, \ell_1'} f_{R,j}(z) .
\end{align}
Observe that in \eqref{Decompositing in terms of theta inside proof}, the term corresponds to $M_1$ is non-zero only when $M_1 \geq 0$. This can be seen as follows: since $\ell_1 \neq \ell_1'$, using the support of $\phi_{j,l}$ and $\Theta$ we can see
\begin{align*}
    R^2 \geq \lambda_{\ell_1, \ell_1'} = \frac{\lambda_{\ell_1, \ell_1'}}{|\ell_1-\ell_1'|} |\ell_1-\ell_1'| \geq \frac{\lambda_{\ell_1, \ell_1'}}{|\ell_1-\ell_1'|} 2^{-M_1-1} R^2 .
\end{align*}
On the other hand, we have $2 |\ell_1-\ell_1'| \leq \lambda_{\ell_1, \ell_1'}$, hence from the above estimate we obtain
\begin{align*}
    2 \leq \frac{\lambda_{\ell_1, \ell_1'}}{|\ell_1-\ell_1'|} \leq 2^{M_1+1} ,
\end{align*} 
which immediately implies that $M_1 \geq 0$.

Similarly as in \eqref{Decompositing in terms of theta inside proof}, we also have
\begin{align}
\label{Expression of psil R with theta}
    \psi_{l,R}^{>}(\sqrt{\mathcal{L}}, |T|)g_{R,j}(z) &= \sum_{M_2=0}^{\infty} \psi_{l,R}^{>, M_2}(\sqrt{\mathcal{L}}, |T|)g_{R,j}(z) .
\end{align}
Therefore plugging \eqref{Decompositing in terms of theta inside proof} and \eqref{Expression of psil R with theta} into \eqref{Fourier series decomposition in restricted}, we see that
\begin{align}
\label{Introducing cuoff of M1 and M2 term}
    \chi_{B_{R}^j}(z) \mathcal{B}_{R,j}^{>, >}(f_{R,j},g_{R,j})(z) &= \chi_{B_{R}^j}(z) \sum_{M_1=0}^{\infty} \sum_{M_2=0}^{\infty} \mathcal{B}_{R,j, M_1, M_2}^{>, >}(f_{R,j},g_{R,j})(z) ,
\end{align}
where
\begin{align}
\label{Definition of Brj greater greater}
    \mathcal{B}_{R,j, M_1, M_2}^{>, >}(f,g)(z) &= \sum_{l \in \mathbb{Z}} \phi_{j,l,R}^{>, M_1}(\sqrt{\mathcal{L}}, |T|)f(z)\, \psi_{l,R}^{>, M_2}(\sqrt{\mathcal{L}}, |T|)g(z) .
\end{align}
Note that without loss of generality we can assume $M_1<M_2$, since the case of $M_2\leq M_1$ can be treated in exactly  the same way. So from now we restrict our attention to the case $M_1 <M_2$. We continue with the same notation for the operator for brevity and for each $j \in \mathbb{Z}$, we decompose the sum of $M_1$ and $M_2$ as follows
\begin{align}
\label{Decomposition in terms of M1 and M2}
    & \chi_{B_{R}^j}(z) \mathcal{B}_{R,j}^{>, >}(f_{R,j},g_{R,j})(z) \\
    &\nonumber = \chi_{B_{R}^j}(z) \sum_{M_2=0}^{j} \sum_{M_1=0}^{M_2} \mathcal{B}_{R,j, M_1, M_2}^{>, >}(f_{R,j},g_{R,j})(z) + \chi_{B_{R}^j}(z) \sum_{M_2=j+1}^{\infty} \sum_{M_1=0}^{j} \mathcal{B}_{R,j, M_1, M_2}^{>, >}(f_{R,j},g_{R,j})(z) \\
    &\nonumber + \chi_{B_{R}^j}(z) \sum_{M_2=j+1}^{\infty} \sum_{M_1=j+1}^{M_2} \mathcal{B}_{R,j, M_1, M_2}^{>, >}(f_{R,j},g_{R,j})(z) \\
    &\nonumber =: S_1 + S_2 + S_3 .
\end{align}

In the following we estimate each $S_1$, $S_2$ and $S_3$ separately. Let us first start with the estimate of $S_3$, since it is is relatively easy to handle using Proposition \ref{Proposition: Truncated restriction type estimate}.
\subsubsection{Estimate of $S_3$}
\label{Estimate of S3 case}
Applying H\"older's inequality twice we obtain
\begin{align*}
    \|S_3\|_{L^{1/2}} &\leq C \left(\frac{2^{j(1+\gamma)}}{ R} \right)^{Q} \sum_{M_2=j+1}^{\infty} \sum_{M_1=j+1}^{M_2} \|\mathcal{B}_{R,j, M_1, M_2}^{>, >}(f_{R,j},g_{R,j})\|_{L^1} \\
    &\leq C \left(\frac{2^{j(1+\gamma)}}{ R} \right)^{Q} \sum_{M_2=j+1}^{\infty} \sum_{M_1=j+1}^{M_2} \sum_{l \in \mathbb{Z}} \|\phi_{j,l,R}^{>, M_1}(\sqrt{\mathcal{L}}, |T|)f_{R,j}\|_{L^2} \|\psi_{l,R}^{>, M_2}(\sqrt{\mathcal{L}}, |T|)g_{R,j}\|_{L^2} .
\end{align*}
Then using Proposition \ref{Proposition: Truncated restriction type estimate} and inequality \eqref{Discrete norm dominated by sup norm} we get
\begin{align}
\label{Use of proposition inside S3 estimate}
    \|S_3\|_{L^{1/2}} &\leq C \left(\frac{2^{j(1+\gamma)}}{ R} \right)^{Q} \sum_{M_2=j+1}^{\infty} \sum_{M_1=j+1}^{M_2} \sum_{l \in \mathbb{Z}} \left( R^{d/2+\epsilon} 2^{-M_1 \epsilon/2} + R^{Q/2} 2^{-M_1/2} \right) \|\phi_{j, l}\|_{L^\infty} \|f_{R,j}\|_{L^1} \\
    &\nonumber \hspace{6cm} \left( R^{d/2+\epsilon} 2^{-M_2 \epsilon/2} + R^{Q/2} 2^{-M_2/2} \right) \|\psi_{l}\|_{\infty} \|g_{R,j}\|_{L^1} \\
    &\nonumber \leq C 2^{j \epsilon_1} \sum_{M_2=j+1}^{\infty} \sum_{M_1=j+1}^{M_2} \Big(2^{j Q} R^{-1+2\epsilon} 2^{-M_1 \epsilon/2} 2^{-M_2 \epsilon/2} + 2^{j Q} R^{-1/2+\epsilon} 2^{-M_1 \epsilon/2} 2^{-M_2/2} \\
    &\nonumber \hspace{1cm} + 2^{j Q} R^{-1/2+\epsilon} 2^{-M_1/2} 2^{-M_2 \epsilon/2} + 2^{j Q} 2^{-M_1/2} 2^{-M_2/2} \Big) \left(\sum_{l \in \mathbb{Z}} \|\phi_{j, l}\|_{L^\infty} \right) \|f_{R,j}\|_{L^1} \|g_{R,j}\|_{L^1} \\
    &\nonumber \leq C 2^{j \epsilon_2} \Big(2^{j Q} R^{-1+2\epsilon} + 2^{j(Q+d)/2} R^{-1/2+\epsilon} + 2^{j d} \Big) \|f_{R,j}\|_{L^1} \|g_{R,j}\|_{L^1} ,
\end{align}
for some $\epsilon_1, \epsilon_2>0$ and where in the last inequality we have also used \eqref{Convergence of sum of l using phi}.

Since $0 \leq j \leq \lfloor\log_2 R\rfloor$, from the previous estimate we obtain
\begin{align*}
    \|S_3\|_{L^{1/2}} &\leq C 2^{j d + j \epsilon_3} \|f_{R,j}\|_{L^1} \|g_{R,j}\|_{L^1} ,
\end{align*}
for some $\epsilon_3>0$.

\subsubsection{Estimate of $S_1$}
\label{Estimate of S1 case}
Note that we have $\supp{f_{R,j}}, \supp{g_{R,j}} \subseteq B_{R}^j = B(e, 3 \cdot \tfrac{2^{j(1+\gamma)}}{ R})$. Then from Lemma \ref{Lemma: Ball contained in product of two Euclidean like ball} we get
\begin{align*}
    B(e, 3 \cdot \tfrac{2^{j(1+\gamma)}}{ R}) \subseteq B^{\varpi, \Im} \Big(e, C (\tfrac{2^{j(1+\gamma)}}{ R}, \tfrac{2^{2j(1+\gamma)}}{ R^{2}}) \Big) .
\end{align*}
Since in this case we have $0\leq M_1 \leq j$, we decompose $\supp{f_{R,j}}$ as follows:
\begin{align}
\label{Decomposition of the support of f}
    \supp{f_{R,j}} &= \bigcup_{m_1=1}^{N_{M_1}} B_{m_1, M_1, R}^{\varpi, \Im, j} ,
\end{align}
such that
\begin{align}
\label{Definition of B m1 set}
    B_{m_1, M_1, R}^{\varpi, \Im, j} \subseteq \Big\{z \in B_{R}^j : \varpi(z, a_{m_1}) \leq C \tfrac{2^{M_1(1+\gamma)}}{ R}, \ \  |\Im \langle z, e \rangle| \leq C \tfrac{2^{2j(1+\gamma)}}{ R^{2}} \Big\} ,
\end{align}
and for $m_1 \neq m_1'$ it satisfies $\varpi(a_{m_1}, a_{m_1'}) \geq C \tfrac{2^{M_1(1+\gamma)}}{2 R}$.

Moreover, since the sets $B_{m_1, M_1, R}^{\varpi, \Im, j}$ are well separated with respect to $\varpi$, the number $N_{M_1}$ of such pieces can be estimated by
\begin{align*}
    N_{M_1} &\lesssim \frac{(\tfrac{2^{j(1+\gamma)}}{ R})^{2(n-1)}}{(\tfrac{2^{M_1 (1+\gamma)}}{R})^{2(n-1)}} \lesssim 2^{(j-M_1)(1+\gamma)2(n-1)} .
\end{align*}
For each $1\leq m_1 \leq N_{M_1}$ and $\vartheta>0$, we define
\begin{align}
\label{Define B tilde set for m1}
    \widetilde{B}_{m_1, M_1, R}^{\varpi, \Im, j} := \Big\{z \in B_{R}^j : \varpi(z, a_{m_1}) \leq C \tfrac{2^{M_1(1+\gamma)} 2^{j\vartheta+1}}{ R}, \ \  |\Im \langle z, e \rangle| \leq C \tfrac{2^{2j(1+\gamma)}}{ R^{2}} \Big\} .
\end{align}
Similarly, for $0\leq M_1 \leq j$, we also decompose the support of $g_{R,j}$ as
\begin{align}
\label{Decomposition of support of g}
    \supp{g_{R,j}} = \bigcup_{m_2=1}^{N_{M_1}} B_{m_2, M_1, R}^{\varpi, \Im, j} ,
\end{align}
where $B_{m_2, M_1, R}^{\varpi, \Im, j}$ is defined similar to \eqref{Definition of B m1 set} and also define $\widetilde{B}_{m_2, M_1, R}^{\varpi, \Im, j}$ analogous to \eqref{Define B tilde set for m1}. Note that for the shake of brevity we are using the same notation $N_{M_1}$ for the number of pieces in the decomposition of support of $f_{R,j}$ and $g_{R,j}$.

Then in view of the decompositions of the supports of $f_{R,j}$ and $g_{R,j}$, we further decompose the functions $f_{R,j}$ and $g_{R,j}$ as follows:
\begin{align*}
    f_{R,j} = \sum_{m_1=1}^{N_{M_1}} f^{R,j}_{m_1, M_1} \qquad \text{and} \qquad g_{R,j} = \sum_{m_2=1}^{N_{M_1}} g^{R,j}_{m_2, M_1} ,
\end{align*}
where
\begin{align*}
    f^{R,j}_{m_1, M_1} = f_{R,j} \chi_{B_{m_1, M_1, R}^{\varpi, \Im, j}} \qquad \text{and} \qquad g^{R,j}_{m_2, M_1} = g_{R,j} \chi_{B_{m_2, M_1, R}^{\varpi, \Im, j}} .
\end{align*}
Consequently, we decompose $S_1$ as follows
\begin{align}
\label{Decomposition of S1 wrt B tilde}
    S_1 &= \sum_{M_2=0}^{j} \sum_{M_1=0}^{M_2} \sum_{m_1=1}^{N_{M_1}} \chi_{B_{R}^j}(z) (1-\chi_{\widetilde{B}_{m_1, M_1, R}^{\varpi, \Im, j}})(z) \mathcal{B}_{R,j, M_1, M_2}^{>, >}(f^{R,j}_{m_1, M_1} ,g_{R,j})(z) \\
    &\nonumber + \sum_{M_2=0}^{j} \sum_{M_1=0}^{M_2} \sum_{m_1=1}^{N_{M_1}} \sum_{m_2=1}^{N_{M_1}} \chi_{B_{R}^j}(z) \chi_{\widetilde{B}_{m_1, M_1, R}^{\varpi, \Im, j}}(z) \chi_{\widetilde{B}_{m_2, M_1, R}^{\varpi, \Im, j}}(z) \mathcal{B}_{R,j, M_1, M_2}^{>, >}(f^{R,j}_{m_1, M_1} ,g^{R,j}_{m_2, M_1})(z) \\
    &\nonumber + \sum_{M_2=0}^{j} \sum_{M_1=0}^{M_2} \sum_{m_1=1}^{N_{M_1}} \sum_{m_2=1}^{N_{M_1}} \chi_{B_{R}^j}(z) \chi_{\widetilde{B}_{m_1, M_1, R}^{\varpi, \Im, j}}(z) (1-\chi_{\widetilde{B}_{m_2, M_1, R}^{\varpi, \Im, j}})(z) \mathcal{B}_{R,j, M_1, M_2}^{>, >}(f^{R,j}_{m_1, M_1} ,g^{R,j}_{m_2, M_1})(z) \\
    &\nonumber =: S_{11} + S_{12} + S_{13} .
\end{align}
Now we estimate each of these pieces. Let us start with the estimate of $S_{11}$.

\subsubsection{Estimate of $S_{11}$}
\label{Estimate of S11 case}
Using H\"older's inequality we obtain
\begin{align}
\label{First use of Holder in S11 part}
    \|S_{11}\|_{L^{1/2}} &\leq C \left(\frac{2^{j(1+\gamma)}}{ R} \right)^{Q} \sum_{M_2=0}^{j} \sum_{M_1=0}^{M_2} \sum_{m_1=1}^{N_{M_1}} \|\chi_{B_{R}^j} (1-\chi_{\widetilde{B}_{m_1, M_1, R}^{\varpi, \Im, j}}) \mathcal{B}_{R,j, M_1, M_2}^{>, >}(f^{R,j}_{m_1, M_1},g_{R,j})\|_{L^1} .
\end{align}
Recalling the definition from \eqref{Definition of Brj greater greater}, and applying H\"older's inequality we get
\begin{align}
\label{Use of Holder in S11 part}
    & \|\chi_{B_{R}^j} (1-\chi_{\widetilde{B}_{m_1, M_1, R}^{\varpi, \Im, j}}) \mathcal{B}_{R,j, M_1, M_2}^{>, >}(f^{R,j}_{m_1, M_1},g_{R,j})\|_{L^1} \\
    &\nonumber \leq C \sum_{l \in \mathbb{Z}} \|\chi_{B_{R,j}} (1-\chi_{\widetilde{B}_{m_1, M_1, R}^{\varpi, \Im, j}}) \phi_{j,l,R}^{ >, M_1}(\sqrt{\mathcal{L}}, |T|)f^{R,j}_{m_1, M_1}\|_{L^2} \|\psi_{l,R}^{>, M_2}(\sqrt{\mathcal{L}}, |T|)g_{R,j}\|_{L^2} .
\end{align}
For each $l \in \mathbb{Z}$, we will first estimate the first factor in the right hand side of the above inequality. Let us write
\begin{align*}
    & \chi_{B_{R}^j}(z) (1-\chi_{\widetilde{B}_{m_1, M_1, R}^{\varpi, \Im, j}})(z) \phi_{j,l,R}^{>, M_1}(\sqrt{\mathcal{L}}, |T|)f^{R,j}_{m_1, M_1}(z) \\
    &= \int_{\mathbb{S}} \chi_{B_{R}^j}(z) (1-\chi_{\widetilde{B}_{m_1, M_1, R}^{\varpi, \Im, j}})(z) \mathcal{K}_{\phi_{j,l,R}^{>, M_1}(\sqrt{\mathcal{L}}, |T|)}(z,w) f^{R,j}_{m_1, M_1}(w) \, d\sigma(w) .
\end{align*}
Note that if $z \in \supp{\chi_{B_{R}^j} (1-\chi_{\widetilde{B}_{m_1, M_1, R}^{\varpi, \Im, j}})}$ and $w \in \supp{f^{R,j}_{m_1, M_1}}$, then
\begin{align*}
    \varpi(z, a_{m_1}) \geq C \tfrac{2^{M_1(1+\gamma)} 2^{j\vartheta+1}}{ R}, \quad \quad \varpi(w, a_{m_1}) \leq C \tfrac{2^{M_1(1+\gamma)}}{ R} .
\end{align*}
Now using Lemma \ref{Lemma: Triangle inequality for weight}, whenever $z \in \supp{\chi_{B_{R}^j} (1-\chi_{\widetilde{B}_{m_1, M_1, R}^{\varpi, \Im, j}})}$ and $w \in \supp{f^{R,j}_{m_1, M_1}}$ we obtain
\begin{align*}
    \varpi(z,w) &\geq \varpi(z, a_{m_1}) - \varpi(w, a_{m_1}) \geq C \tfrac{2^{M_1(1+\gamma)} 2^{j\vartheta}}{ R} ,
\end{align*}
for some $C>0$.

Applying Minkowski's integral inequality and Proposition \ref{Proposition: Weighted Plancherel for large N case} and the fact \eqref{Discrete norm dominated by sup norm} for any $N>0$ we get
\begin{align}
\label{Use of minkowski inequality in S1}
    & \|\chi_{B_{R}^j} (1-\chi_{\widetilde{B}_{m_1, M_1, R}^{\varpi, \Im, j}}) \phi_{j,l,R}^{>, M_1}(\sqrt{\mathcal{L}}, |T|)f^{R,j}_{m_1, M_1}\|_{L^2} \\
    &\nonumber \leq \int_{\mathbb{S}} |f^{R,j}_{m_1, M_1}(w)| \left(\int_{\mathbb{S}} |\chi_{B_{R}^j}(z) (1-\chi_{\widetilde{B}_{m_1, M_1, R}^{\varpi, \Im, j}})(z) \mathcal{K}_{\phi_{j,l,R}^{>, M_1}(\sqrt{\mathcal{L}}, |T|)}(z,w)|^2 \, d\sigma(z) \right)^{1/2} \, d\sigma(w) \\
    &\nonumber \leq C \int_{\mathbb{S}} |f^{R,j}_{m_1, M_1}(w)| \left[ \left( \tfrac{2^{M_1(1+\gamma)} 2^{j\vartheta}}{ R} \right)^{-N} \left(\int_{\mathbb{S}} |\varpi(z,w)^N \mathcal{K}_{\phi_{j,l,R}^{>, M_1}(\sqrt{\mathcal{L}}, |T|)}(z,w)|^2 \, d\sigma(z) \right)^{1/2} \right] \, d\sigma(w)  \\
    &\nonumber \leq C \left( \tfrac{2^{M_1(1+\gamma)} 2^{j\vartheta}}{ R} \right)^{-N} R^{(Q/2-N)} 2^{M_1 N} \|\phi_{j,l}\|_{L^{\infty}} \|f^{R,j}_{m_1, M_1}\|_{L^1} \\
    &\nonumber \leq C 2^{-j \vartheta N} 2^{-M_1 \gamma N} R^{Q/2} \|\phi_{j,l}\|_{L^{\infty}} \|f^{R,j}_{m_1, M_1}\|_{L^1} .
\end{align}
On the other hand from Proposition \ref{Proposition: Truncated restriction type estimate} and inequality \eqref{Discrete norm dominated by sup norm} we have
\begin{align*}
    \|\psi_{l,R}^{>, M_2}(\sqrt{\mathcal{L}}, |T|)g_{R,j}\|_{L^2} &\leq C \left( R^{d/2+\epsilon} 2^{-M_2 \epsilon/2} + R^{Q/2} 2^{-M_2/2} \right) \|\psi_{l}\|_{\infty} \|g_{R,j}\|_{L^1} .
\end{align*}
Hence combining the above two estimates and using \eqref{Convergence of sum of l using phi}, from \eqref{Use of Holder in S11 part} we get
\begin{align}
\label{L1 norm estimate in S11 part}
    & \|\chi_{B_{R}^j} (1-\chi_{\widetilde{B}_{m_1, M_1, R}^{\varpi, \Im, j}}) \mathcal{B}_{R,j, M_1, M_2}^{>, >}(f^{R,j}_{m_1, M_1},g_{R,j})\|_{L^1} \\
    &\nonumber \leq C 2^{-j \vartheta N} 2^{-M_1 \gamma N} R^{Q/2} \Big(\sum_{l \in \mathbb{Z}} \|\phi_{j,l}\|_{L^{\infty}} \Big) \|f^{R,j}_{m_1, M_1}\|_{L^1} \left( R^{d/2+\epsilon} 2^{-M_2 \epsilon/2} + R^{Q/2} 2^{-M_2/2} \right) \|g_{R,j}\|_{L^1} \\
    &\nonumber \leq C 2^{-j \vartheta N} \left( R^{(Q+d)/2+\epsilon} 2^{-M_2 \epsilon/2} + R^{Q} 2^{-M_2/2} \right) \|f^{R,j}_{m_1, M_1}\|_{L^1} \|g_{R,j}\|_{L^1} .
\end{align}
Therefore plugging the above estimate into \eqref{First use of Holder in S11 part} we obtain
\begin{align*}
    \|S_{11}\|_{L^{1/2}} &\leq C \left(\frac{2^{j(1+\gamma)}}{ R} \right)^{Q} \sum_{M_2=0}^{j} \sum_{M_1=0}^{M_2} \sum_{m_1=1}^{N_{M_1}} 2^{-j \vartheta N} \\
    &\hspace{3cm} \left( R^{(Q+d)/2+\epsilon} 2^{-M_2 \epsilon/2} + R^{Q} 2^{-M_2/2} \right) \|f^{R,j}_{m_1, M_1}\|_{L^1} \|g_{R,j}\|_{L^1} \\
    &\leq C 2^{j \epsilon_1} 2^{-j \vartheta N} 2^{j Q} \|f_{R,j}\|_{L^1} \|g_{R,j}\|_{L^1} ,
\end{align*} 
for some $\epsilon_1>0$.

Consequently, choosing $N>0$ sufficiently large such that $\vartheta N > Q+\epsilon_1$ we get
\begin{align*}
    \|S_{11}\|_{L^{1/2}} &\leq C 2^{-j \epsilon_2} \|f_{R,j}\|_{L^1} \|g_{R,j}\|_{L^1} ,
\end{align*}
for some $\epsilon_2>0$.

\subsubsection{Estimate of $S_{12}$}
\label{Estimate of S12 case}
Similar to \eqref{Decomposition of the support of f} and \eqref{Decomposition of support of g} we decompose $B_{R}^j$ as follows
\begin{align*}
    B_{R}^j &= \bigcup_{m=1}^{N_{M_1}} B_{m, M_1, R}^{\varpi, \Im, j} 
\end{align*}
such that
\begin{align*}
    B_{m, M_1, R}^{\varpi, \Im, j} \subseteq \{z \in B_{R}^j : \varpi(z, a_{m}) \leq C \tfrac{2^{M_1(1+\gamma)}}{ R},\ \  |\Im \langle z, e \rangle| \leq C \tfrac{2^{2j(1+\gamma)}}{ R^{2}} \} ,
\end{align*}
and for $m \neq m'$ it satisfies $\varpi(a_{m}, a_{m'}) \geq C \tfrac{2^{M_1(1+\gamma)}}{2 R}$. 

Therefore in view of the above decomposition we write
\begin{align*}
    S_{12} &= \sum_{M_2=0}^{j} \sum_{M_1=0}^{M_2} \sum_{m_1=1}^{N_{M_1}} \sum_{m_2=1}^{N_{M_1}} \sum_{m=1}^{N_{M_1}} \chi_{B_{m, M_1, R}^{\varpi, \Im, j}}(z) \chi_{\widetilde{B}_{m_1, M_1, R}^{\varpi, \Im, j}}(z) \chi_{\widetilde{B}_{m_2, M_1, R}^{\varpi, \Im, j}}(z) \mathcal{B}_{R,j, M_1, M_2}^{>, >}(f^{R,j}_{m_1, M_1} ,g^{R,j}_{m_2, M_1})(z) \\
    &= \sum_{M_2=0}^{j} \sum_{M_1=0}^{M_2} \sum_{m=1}^{N_{M_1}} \sum_{m_1: B_{m, M_1, R}^{\varpi, \Im, j} \cap \widetilde{B}_{m_1, M_1, R}^{\varpi, \Im, j} \neq \emptyset} \sum_{m_2: B_{m, M_1, R}^{\varpi, \Im, j} \cap \widetilde{B}_{m_2, M_1, R}^{\varpi, \Im, j} \neq \emptyset} \\
    &\hspace{4cm} \chi_{B_{m, M_1, R}^{\varpi, \Im, j}}(z) \chi_{\widetilde{B}_{m_1, M_1, R}^{\varpi, \Im, j}}(z) \chi_{\widetilde{B}_{m_2, M_1, R}^{\varpi, \Im, j}}(z) \mathcal{B}_{R,j, M_1, M_2}^{>, >}(f^{R,j}_{m_1, M_1} ,g^{R,j}_{m_2, M_1})(z) .
\end{align*}
Now using the fact \eqref{Triangle inequality for p less than 1 case} and H\"older's inequality we get
\begin{align}
\label{Calculation of S12 with power 1/2 case}
    \|S_{12}\|_{L^{1/2}}^{1/2} &\leq \sum_{M_2=0}^{j} \sum_{M_1=0}^{M_2} \sum_{m=1}^{N_{M_1}} \sum_{m_1: B_{m, M_1, R}^{\varpi, \Im, j} \cap \widetilde{B}_{m_1, M_1, R}^{\varpi, \Im, j} \neq \emptyset} \sum_{m_2: B_{m, M_1, R}^{\varpi, \Im, j} \cap \widetilde{B}_{m_2, M_1, R}^{\varpi, \Im, j} \neq \emptyset} \\
    &\nonumber \hspace{3cm} \|\chi_{B_{m, M_1, R}^{\varpi, \Im, j}} \chi_{\widetilde{B}_{m_1, M_1, R}^{\varpi, \Im, j}} \chi_{\widetilde{B}_{m_2, M_1, R}^{\varpi, \Im, j}} \mathcal{B}_{R,j, M_1, M_2}^{>, >}(f^{R,j}_{m_1, M_1} ,g^{R,j}_{m_2, M_1})\|_{L^{1/2}}^{1/2} \\
    &\nonumber \leq C \sum_{M_2=0}^{j} \sum_{M_1=0}^{M_2} \sum_{m=1}^{N_{M_1}} \sum_{m_1: B_{m, M_1, R}^{\varpi, \Im, j} \cap \widetilde{B}_{m_1, M_1, R}^{\varpi, \Im, j} \neq \emptyset} \sum_{m_2: B_{m, M_1, R}^{\varpi, \Im, j} \cap \widetilde{B}_{m_2, M_1, R}^{\varpi, \Im, j} \neq \emptyset} \\
    &\nonumber \hspace{3cm} \left(\frac{2^{M_1(1+\gamma)}}{ R} \right)^{Q/2} \|\mathcal{B}_{R,j, M_1, M_2}^{>, >}(f^{R,j}_{m_1, M_1} ,g^{R,j}_{m_2, M_1})\|_{L^{1}}^{1/2} .
\end{align}
Applying H\"older's inequality, Proposition \ref{Proposition: Truncated restriction type estimate} and inequality \eqref{Discrete norm dominated by sup norm} yields
\begin{align}
\label{L1 norm estimate in S12 case}
    & \|\mathcal{B}_{R,j, M_1, M_2}^{>, >}(f^{R,j}_{m_1, M_1} ,g^{R,j}_{m_2, M_1})\|_{L^{1}} \\
    &\nonumber \leq C \sum_{l \in \mathbb{Z}} \|\phi_{j,l,R}^{>, M_1}(\sqrt{\mathcal{L}}, |T|)f^{R,j}_{m_1, M_1}\|_{L^2} \|\psi_{l,R}^{>, M_2}(\sqrt{\mathcal{L}}, |T|)g^{R,j}_{m_2, M_1}\|_{L^2} \\
    &\nonumber \leq C \sum_{l \in \mathbb{Z}} \left( R^{d/2+\epsilon} 2^{-M_1 \epsilon/2} + R^{Q/2} 2^{-M_1/2} \right) \|\phi_{j, l}\|_{L^\infty} \|f^{R,j}_{m_1, M_1}\|_{L^1} \\
    &\nonumber \hspace{5cm} \times \left( R^{d/2+\epsilon} 2^{-M_2 \epsilon/2} + R^{Q/2} 2^{-M_2/2} \right) \|\psi_{l}\|_{\infty} \|g^{R,j}_{m_2, M_1}\|_{L^1} \\
    &\nonumber \leq C 2^{j \epsilon_1} \left( R^{d+2\epsilon} 2^{-M_1 \epsilon} + R^{Q} 2^{-M_1} \right) \|f^{R,j}_{m_1, M_1}\|_{L^1} \|g^{R,j}_{m_2, M_1}\|_{L^1} ,
\end{align}
for some $\epsilon_1>0$ and where in the last inequality we have used $M_1<M_2$, \eqref{Convergence of sum of l using phi} and $2ab \leq a^2 + b^2$ for any $a, b \in \mathbb{R}$.

Hence plugging the above estimate into \eqref{Calculation of S12 with power 1/2 case} we get
\begin{align}
\label{After plugging the L1 estimate}
    \|S_{12}\|_{L^{1/2}}^{1/2} &\leq C \sum_{M_2=0}^{j} \sum_{M_1=0}^{M_2} \sum_{m=1}^{N_{M_1}} \sum_{m_1: B_{m, M_1, R}^{\varpi, \Im, j} \cap \widetilde{B}_{m_1, M_1, R}^{\varpi, \Im, j} \neq \emptyset} \sum_{m_2: B_{m, M_1, R}^{\varpi, \Im, j} \cap \widetilde{B}_{m_2, M_1, R}^{\varpi, \Im, j} \neq \emptyset} \\
    &\nonumber \hspace{1cm} 2^{j \epsilon_2} \left( 2^{M_1 Q/2} R^{-1/2+2\epsilon} + 2^{M_1 d/2} \right) \|f^{R,j}_{m_1, M_1}\|_{L^1}^{1/2} \|g^{R,j}_{m_2, M_1}\|_{L^1}^{1/2} ,
\end{align}
for some $\epsilon_2>0$ sufficiently small, which we get by choosing $\gamma>0$ very small.

Since for $i=1,2$ and $m_i \neq m_i'$ we have $\varpi(a_{m_i}, a_{m_i'}) \geq C \tfrac{2^{M_1(1+\gamma)}}{2 R}$, using Lemma \ref{Lemma: Triangle inequality for weight}, there exists a dimensional constant $C>0$ such that
\begin{align*}
    \sup_m \#\{m_i: B_{m, M_1, R}^{\varpi, \Im, j} \cap \widetilde{B}_{m_i, M_1, R}^{\varpi, \Im, j} \neq \emptyset \} &\leq \sup_m \#\{m_i: \varpi(a_m, a_{m_i}) \leq C \tfrac{2^{M_1(1+\gamma)} 2^{j\vartheta+1}}{ R} \} \\
    & \leq C 2^{C \vartheta j} .
\end{align*}
Therefore using the previous estimate and successive use of H\"older's inequality in \eqref{After plugging the L1 estimate} with respect to $m_1$, $m_2$ and $m$ we see 
\begin{align}
\label{Estimate of S12 before the condition}
    & \|S_{12}\|_{L^{1/2}}^{1/2} \leq C \sum_{M_2=0}^{j} \sum_{M_1=0}^{M_2} 2^{C \vartheta j} 2^{j \epsilon_2} \left( 2^{M_1 Q/2} R^{-1/2+2\epsilon} + 2^{M_1 d/2} \right) \\
    &\nonumber \left( \sum_{m=1}^{N_{M_1}} \sum_{m_1: B_{m, M_1, R}^{\varpi, \Im, j} \cap \widetilde{B}_{m_1, M_1, R}^{\varpi, \Im, j} \neq \emptyset} \|f^{R,j}_{m_1, M_1}\|_{L^1} \right)^{1/2} \left(\sum_{m=1}^{N_{M_1}} \sum_{m_2: B_{m, M_1, R}^{\varpi, \Im, j} \cap \widetilde{B}_{m_2, M_1, R}^{\varpi, \Im, j} \neq \emptyset} \|g^{R,j}_{m_2, M_1}\|_{L^1} \right)^{1/2} \\
    &\nonumber \leq C \sum_{M_2=0}^{j} \sum_{M_1=0}^{M_2} 2^{C \vartheta j} 2^{j \epsilon_2} \left( 2^{M_1 Q/2} R^{-1/2+2\epsilon} + 2^{M_1 d/2} \right) \|f_{R,j}\|_{L^1}^{1/2} \|g_{R,j}\|_{L^1}^{1/2} .
\end{align}
Since $0 \leq j \leq \lfloor\log_2 R\rfloor$, from the above estimate we obtain
\begin{align*}
    \|S_{12}\|_{L^{1/2}}^{1/2} &\leq C 2^{C \vartheta j} 2^{j \epsilon_3} 2^{j d/2} \|f_{R,j}\|_{L^1}^{1/2} \|g_{R,j}\|_{L^1}^{1/2} ,
\end{align*}
for some $\epsilon_3>0$.

Now taking square on both side of the above inequality we get
\begin{align*}
    \|S_{12}\|_{L^{1/2}} &\leq C 2^{2C \vartheta j} 2^{j 2\epsilon_3} 2^{j d} \|f_{R,j}\|_{L^1} \|g_{R,j}\|_{L^1} \leq C 2^{j d + j \epsilon_4} \|f_{R,j}\|_{L^1} \|g_{R,j}\|_{L^1} ,
\end{align*}
for some $\epsilon_4>0$, provided we choose $\vartheta>0$ sufficiently small.

\subsubsection{Estimate of $S_{13}$}
\label{Estimate of S13 case}
This estimate is similar to the estimate of $S_{11}$ (see \eqref{Estimate of S11 case}), but the main difference is that here we crucially use the assumption made on the Theorem \ref{Theorem: Bilinear Bochner-Riesz theorem with restricted f and g} about the input function $g$.

\medskip

Using H\"older's inequality we get
\begin{align*}
    & \|S_{13}\|_{L^{1/2}} \\
    &\leq C \left(\frac{2^{j(1+\gamma)}}{ R} \right)^{Q} \sum_{M_2=0}^{j} \sum_{M_1=0}^{M_2} \sum_{m_1=1}^{N_{M_1}} \sum_{m_2=1}^{N_{M_1}} \| \chi_{B_{R}^j} \chi_{\widetilde{B}_{m_1, M_1, R}^{\varpi, \Im, j}} (1-\chi_{\widetilde{B}_{m_2, M_1, R}^{\varpi, \Im, j}}) \mathcal{B}_{R,j, M_1, M_2}^{>, >}(f^{R,j}_{m_1, M_1} ,g^{R,j}_{m_2, M_1})\|_{L^1} .
\end{align*}
Again applying H\"older's inequality we obtain
\begin{align}
\label{Use of Holder from L1 to 1/2 1/2}
    & \| \chi_{B_{R}^j} \chi_{\widetilde{B}_{m_1, M_1, R}^{\varpi, \Im, j}} (1-\chi_{\widetilde{B}_{m_2, M_1, R}^{\varpi, \Im, j}}) \mathcal{B}_{R,j, M_1, M_2}^{>, >}(f^{R,j}_{m_1, M_1} ,g^{R,j}_{m_2, M_1})\|_{L^1} \\
    &\nonumber \leq C \sum_{l \in \mathbb{Z}} \| \phi_{j,l,R}^{ >, M_1}(\sqrt{\mathcal{L}}, |T|)f^{R,j}_{m_1, M_1}\|_{L^2} \|\chi_{B_{R}^j} (1-\chi_{\widetilde{B}_{m_2, M_1, R}^{\varpi, \Im, j}}) \psi_{l,R}^{>, M_2}(\sqrt{\mathcal{L}}, |T|)g^{R,j}_{m_2, M_1}) \|_{L^2} .
\end{align}
First we will focus on the estimate of $\|\chi_{B_{R}^j} (1-\chi_{\widetilde{B}_{m_2, M_1, R}^{\varpi, \Im, j}}) \psi_{l,R}^{>, M_2}(\sqrt{\mathcal{L}}, |T|)g^{R,j}_{m_2, M_1}) \|_{L^2}$. Applying Minkowski's integral inequality we get
\begin{align}
\label{Application of Minkowski in S13}
    & \|\chi_{B_{R}^j} (1-\chi_{\widetilde{B}_{m_2, M_1, R}^{\varpi, \Im, j}}) \psi_{l,R}^{>, M_2}(\sqrt{\mathcal{L}}, |T|)g^{R,j}_{m_2, M_1})\|_{L^2} \\
    &\nonumber \leq \int_{\mathbb{S}} |g^{R,j}_{m_2, M_1}(w)| \left(\int_{\mathbb{S}} |\chi_{B_{R}^j}(z) (1-\chi_{\widetilde{B}_{m_2, M_1, R}^{\varpi, \Im, j}})(z) \mathcal{K}_{\psi_{l,R}^{>, M_2}(\sqrt{\mathcal{L}}, |T|)}(z,w)|^2 \, d\sigma(z) \right)^{1/2} \, d\sigma(w) .
\end{align}
Notice that as earlier if $z \in \supp{\chi_{B_{R}^j} (1-\chi_{\widetilde{B}_{m_2, M_1, R}^{\varpi, \Im, j}})}$ and $w \in \supp{g^{R,j}_{m_2, M_1}}$, then we obtain
\begin{align*}
    \varpi(z,w) &\geq C \tfrac{2^{M_1(1+\gamma)} 2^{j\vartheta}}{ R} .
\end{align*}
Recall that $\psi_{l,R}^{>, M_2}(\sqrt{\eta_2}, \tau_2) = \psi_{l,R}^{>}(\sqrt{\lambda_{\ell_2, \ell_2'}}, |\ell_2 - \ell_2'|) \Omega_{\kappa_2',R}^{>}(|\ell_2 - \ell_2'|) \Theta(2^{M_2} |\ell_2 - \ell_2'|)$. Note that due to the support of $\Omega_{\kappa_2',R}^{>}$ we have $|\ell_2-\ell_2'| \geq \kappa_2' R^2$. Also note that from the support of $\Theta$ we have $|\ell_2-\ell_2'| \leq 2^{-M_2+1} R^2$. Hence combining both the facts we get
\begin{align*}
    2^{M_2} \leq 2 \kappa_2'^{-1} .
\end{align*}
Therefore using the above bound on $M_2$ and from Proposition \ref{Proposition: Weighted Plancherel for large N case} as well as the fact \eqref{Discrete norm dominated by sup norm} for any $N>0$ we get
\begin{align}
\label{L2 norm of the kernel estimate}
    &  \left(\int_{\mathbb{S}} |\chi_{B_{R}^j}(z) (1-\chi_{\widetilde{B}_{m_2, M_1, R}^{\varpi, \Im, j}})(z) \mathcal{K}_{\psi_{l,R}^{>, M_2}(\sqrt{\mathcal{L}}, |T|)}(z,w)|^2 \, d\sigma(w) \right)^{1/2} \\
    &\nonumber \leq C \left( \tfrac{2^{M_1(1+\gamma)} 2^{j\vartheta}}{ R} \right)^{-N} \left(\int_{\mathbb{S}} |\varpi(z,w)^N \mathcal{K}_{\psi_{l,R}^{>, M_2}(\sqrt{\mathcal{L}}, |T|)}(z,w)|^2 \, d\sigma(w) \right)^{1/2} \\
    &\nonumber \leq C \left( \tfrac{2^{M_1(1+\gamma)} 2^{j\vartheta}}{ R} \right)^{-N} R^{(Q/2-N)} 2^{M_2 N} \|\psi_{l}\|_{L^{\infty}} \\
    &\nonumber \leq C_{\kappa_2'} 2^{-j \vartheta N} R^{Q/2} .
\end{align}
On the other hand from Proposition \ref{Proposition: Truncated restriction type estimate} and inequality \eqref{Discrete norm dominated by sup norm} we have
\begin{align}
\label{L2 norm estimate for phi}
    \|\phi_{j,l,R}^{>, M_1}(\sqrt{\mathcal{L}}, |T|)f^{R,j}_{m_1, M_1}\|_{L^2} &\leq C \left( R^{d/2+\epsilon} 2^{-M_2 \epsilon/2} + R^{Q/2} 2^{-M_2/2} \right) \|\phi_{j,l}\|_{\infty} \|f^{R,j}_{m_1, M_1}\|_{L^1} .
\end{align}
Hence combining the estimates \eqref{L2 norm of the kernel estimate} and \eqref{Application of Minkowski in S13} and also using \eqref{L2 norm estimate for phi} from \eqref{Use of Holder from L1 to 1/2 1/2} we get
\begin{align}
\label{L1 norm estimate for S13 case}
    & \| \chi_{B_{R}^j} \chi_{\widetilde{B}_{m_1, M_1, R}^{\varpi, \Im, j}} (1-\chi_{\widetilde{B}_{m_2, M_1, R}^{\varpi, \Im, j}}) \mathcal{B}_{R,j, M_1, M_2}^{>, >}(f^{R,j}_{m_1, M_1} ,g^{R,j}_{m_2, M_1})\|_{L^1} \\
    &\nonumber \leq C 2^{-j \vartheta N} R^{Q/2} \sum_{l \in \mathbb{Z}} \|\phi_{j,l}\|_{L^{\infty}} \|f^{R,j}_{m_1, M_1}\|_{L^1} \left( R^{d/2+\epsilon} 2^{-M_2 \epsilon/2} + R^{Q/2} 2^{-M_2/2} \right) \|g^{R,j}_{m_2, M_1}\|_{L^1} \\
    &\nonumber \leq C 2^{-j \vartheta N} \left( R^{(Q+d)/2+\epsilon} 2^{-M_2 \epsilon/2} + R^{Q} 2^{-M_2/2} \right) \|f^{R,j}_{m_1, M_1}\|_{L^1} \|g^{R,j}_{m_2, M_1}\|_{L^1} ,
\end{align}
where in the last inequality we have used \eqref{Convergence of sum of l using phi}.

Now proceeding similarly as in the estimate of $S_{11}$ (see \eqref{Estimate of S11 case}) we obtain
\begin{align*}
    \|S_{13}\|_{L^{1/2}} &\leq C 2^{-j \epsilon_1} \|f_{R,j}\|_{L^1} \|g_{R,j}\|_{L^1} ,
\end{align*}
for some $\epsilon_1>0$.

Finally, combining all the estimates of $S_{11}$ \eqref{Estimate of S11 case}, $S_{12}$ \eqref{Estimate of S12 case} and $S_{13}$ \eqref{Estimate of S13 case} where $0 \leq j \leq \lfloor\log_2 R\rfloor$, we obtain
\begin{align*}
    \|S_{1}\|_{L^{1/2}} &\leq C 2^{jd+ j \epsilon} \|f_{R,j}\|_{L^1} \|g_{R,j}\|_{L^1} ,
\end{align*}
for some $\epsilon>0$.

\subsubsection{Estimate of $S_{2}$}
\label{Estimate of S2}
Note that we have
\begin{align*}
    S_2 &= \chi_{B_{R}^j}(z) \sum_{M_2=j+1}^{\infty} \sum_{M_1=0}^{j} \mathcal{B}_{R,j, M_1, M_2}^{>, >}(f_{R,j},g_{R,j})(z) .
\end{align*}
Estimate of $S_{2}$ is similar to the estimate of $S_{1}$ (see \eqref{Estimate of S1 case}). All the estimates of $S_{11}$, $S_{12}$ and $S_{13}$ will go through, one has to just notice that $M_2$ sum is summable. In fact, for analogous estimate as of \eqref{L1 norm estimate in S12 case} in $S_{12}$, instead of using $M_1<M_2$, in this case we sum over $M_2 >j$ and obtain the same estimate.

This completes the estimate of \eqref{To prove main for second theorem}. Now we move to the estimate of \eqref{To prove for Error for second theorem}.

\medskip
\noindent \textbf{Estimate of (\ref{To prove for Error for second theorem}):}
Similarly as in \eqref{Introducing cuoff of M1 and M2 term} and \eqref{Decomposition in terms of M1 and M2} we have the following decomposition
\begin{align}
\label{Writing Brj as sum of E1 E2 and E3}
    \mathcal{B}_{R,j}^{>, >}(f,g)(z) & = \sum_{M_2=0}^{j} \sum_{M_1=0}^{M_2} \mathcal{B}_{R,j, M_1, M_2}^{>, >}(f,g)(z) + \sum_{M_2=j+1}^{\infty} \sum_{M_1=0}^{j} \mathcal{B}_{R,j, M_1, M_2}^{>, >}(f,g)(z) \\
    &\nonumber\hspace{2cm} + \sum_{M_2=j+1}^{\infty} \sum_{M_1=j+1}^{M_2} \mathcal{B}_{R,j, M_1, M_2}^{>, >}(f,g)(z) \\
    &\nonumber =: E_1 + E_2 + E_3 ,
\end{align}
where from \eqref{Definition of Brj greater greater} one has
\begin{align*}
    \mathcal{B}_{R,j, M_1, M_2}^{>, >}(f,g)(z) &= \sum_{l \in \mathbb{Z}} \phi_{j,l,R}^{>, M_1}(\sqrt{\mathcal{L}}, |T|)f(z)\, \psi_{l,R}^{>, M_2}(\sqrt{\mathcal{L}}, |T|)g(z) .
\end{align*}

We now establish bounds for $E_1$, $E_2$ and $E_3$ analogous to those obtained for $S_1$, $S_2$ and $S_3$ respectively. Arguments follows the same line, with the following modifications. One has to just replace $\frac{2^{j(1+\gamma)}}{ R}$ by $1$ everywhere and also have to use the fact $j>\lfloor\log_2 R\rfloor$ in place of $0 \leq j \leq \lfloor\log_2 R\rfloor$. Hence we are very brief.

\subsubsection{Estimate of $E_3$}
\label{Estimate of E3 case}
Since $\mathbb{S}$ is compact, using H\"older's inequality twice we obtain
\begin{align*}
    \|E_3\|_{L^{1/2}(\mathbb{S})} &\leq C \sum_{M_2=j+1}^{\infty} \sum_{M_1=j+1}^{M_2} \|\mathcal{B}_{R,j, M_1, M_2}^{>, >}(f,g)\|_{L^1(\mathbb{S})} \\
    &\leq C \sum_{M_2=j+1}^{\infty} \sum_{M_1=j+1}^{M_2} \sum_{l \in \mathbb{Z}} \|\phi_{j,l,R}^{>, M_1}(\sqrt{\mathcal{L}}, |T|)f\|_{L^2} \|\psi_{l,R}^{>, M_2}(\sqrt{\mathcal{L}}, |T|)g\|_{L^2} .
\end{align*}
Proceeding similarly as in the estimate of \eqref{Use of proposition inside S3 estimate} and using Proposition \ref{Proposition: Truncated restriction type estimate}, inequality \eqref{Discrete norm dominated by sup norm} and the fact $j>\lfloor\log_2 R\rfloor$ we have
\begin{align*}
    \|E_3\|_{L^{1/2}} &\leq C \sum_{M_2=j+1}^{\infty} \sum_{M_1=j+1}^{M_2} \sum_{l \in \mathbb{Z}} \left( R^{d/2+\epsilon} 2^{-M_1 \epsilon/2} + R^{Q/2} 2^{-M_1/2} \right) \|\phi_{j, l}\|_{L^\infty} \|f\|_{L^1} \\
    &\nonumber \hspace{6cm} \left( R^{d/2+\epsilon} 2^{-M_2 \epsilon/2} + R^{Q/2} 2^{-M_2/2} \right) \|\psi_{l}\|_{\infty} \|g\|_{L^1} \\
    &\nonumber \leq C \sum_{M_2=j+1}^{\infty} \sum_{M_1=j+1}^{M_2} \Big(2^{j d} 2^{-M_1 \epsilon/2} 2^{-M_2 \epsilon/2} + 2^{j (d+1/2)} 2^{-M_1 \epsilon/2} 2^{-M_2/2} \\
    &\nonumber \hspace{1cm} + 2^{j (d+1/2)} 2^{-M_1/2} 2^{-M_2 \epsilon/2} + 2^{j Q} 2^{-M_1/2} 2^{-M_2/2} \Big) \Big(\sum_{l \in \mathbb{Z}} \|\phi_{j, l}\|_{L^\infty} \Big) \|f\|_{L^1} \|g\|_{L^1} \\
    &\nonumber \leq C 2^{j d + j \epsilon_1} \|f\|_{L^1} \|g\|_{L^1} ,
\end{align*}
for some $\epsilon_1>0$.

\subsubsection{Estimate of $E_1$}
\label{Estimate of E1 case}
Note that $\mathbb{S} = B(e, R_0)$ for $R_0>\sqrt{2}$. Hence from Lemma \ref{Lemma: Ball contained in product of two Euclidean like ball} we have
\begin{align*}
    \mathbb{S} = B(e, R_0) \subseteq B^{\varpi, \Im} \big(e, C (1,1) \big) \quad \quad \text{for some} \quad C>0 .
\end{align*}
For $0\leq M_1 \leq j$, decompositing the support of $f$ we get
\begin{align}
\label{Decomposition of f in E1}
    \supp{f} &= \bigcup_{m_1=1}^{N_{M_1}} B_{m_1, M_1}^{\varpi, \Im} ,
\end{align}
such that
\begin{align*}
    B_{m_1, M_1}^{\varpi, \Im} \subseteq \Big\{z \in \mathbb{S} : \varpi(z, a_{m_1}) \leq C \tfrac{2^{M_1(1+\gamma)}}{ R}, \ \  |\Im \langle z, e \rangle| \leq C \Big\} ,
\end{align*}
and for $m_1 \neq m_1'$ and it satisfies $\varpi(a_{m_1}, a_{m_1'}) \geq C \tfrac{2^{M_1(1+\gamma)}}{2 R}$.

Correspondingly, for $\vartheta>0$ we also define the dilated sets by
\begin{align*}
    \widetilde{B}_{m_1, M_1}^{\varpi, \Im} := \Big\{z \in \mathbb{S} : \varpi(z, a_{m_1}) \leq C \tfrac{2^{M_1(1+\gamma)} 2^{j\vartheta+1}}{ R}, \ \  |\Im \langle z, e \rangle| \leq C \Big\} ,
\end{align*}
and decompose the support of $g$ as $\supp g = \cup_{m_2=1}^{N_{M_1}} B_{m_2, M_1}^{\varpi, \Im} $.

Let us set $f_{m_1, M_1} := f \chi_{B_{m_1, M_1}^{\varpi, \Im}}$ and $g_{m_2, M_1} := g \chi_{B_{m_2, M_1}^{\varpi, \Im}}$, then decomposing similarly as in the estimate of \eqref{Decomposition of S1 wrt B tilde} we have
\begin{align*}
    E_1 &= \sum_{M_2=0}^{j} \sum_{M_1=0}^{M_1} \sum_{m_1=1}^{N_{M_1}} (1-\chi_{\widetilde{B}_{m_1, M_1}^{\varpi, \Im}})(z) \mathcal{B}_{R,j, M_1, M_2}^{>, >}(f_{m_1, M_1} ,g)(z) \\
    &\nonumber + \sum_{M_2=0}^{j} \sum_{M_1=0}^{M_2} \sum_{m_1=1}^{N_{M_1}} \sum_{m_2=1}^{N_{M_1}} \chi_{\widetilde{B}_{m_1, M_1}^{\varpi, \Im}}(z) \chi_{\widetilde{B}_{m_2, M_1}^{\varpi, \Im}}(z) \mathcal{B}_{R,j, M_1, M_2}^{>, >}(f_{m_1, M_1} ,g_{m_2, M_1})(z) \\
    &\nonumber + \sum_{M_2=0}^{j} \sum_{M_1=0}^{M_2} \sum_{m_1=1}^{N_{M_1}} \sum_{m_2=1}^{N_{M_1}} \chi_{\widetilde{B}_{m_1, M_1}^{\varpi, \Im}}(z) (1-\chi_{\widetilde{B}_{m_2, M_1}^{\varpi, \Im}})(z) \mathcal{B}_{R,j, M_1, M_2}^{>, >}(f_{m_1, M_1} ,g_{m_2, M_1})(z) \\
    &\nonumber =: E_{11} + E_{12} + E_{13} .
\end{align*}

\subsubsection{Estimate of $E_{11}$}
\label{Estimate of E11 case}
Since $\mathbb{S}$ is compact, applying H\"older's inequality yields
\begin{align}
\label{Holder in estimate of E11 case}
    \|E_{11}\|_{L^{1/2}} &\leq C \sum_{M_2=0}^{j} \sum_{M_1=0}^{M_2} \sum_{m_1=1}^{N_{M_1}} \|(1-\chi_{\widetilde{B}_{m_1, M_1}^{\varpi, \Im}}) \mathcal{B}_{R,j, M_1, M_2}^{>, >}(f_{m_1, M_1},g)\|_{L^1} .
\end{align}
Estimating similar to \eqref{L1 norm estimate in S11 part} in this case we get
\begin{align*}
    & \|(1-\chi_{\widetilde{B}_{m_1, M_1}^{\varpi, \Im}}) \mathcal{B}_{R,j, M_1, M_2}^{>, >}(f_{m_1, M_1},g)\|_{L^1} \\
    &\leq C 2^{-j \vartheta N} \left( R^{(Q+d)/2+\epsilon} 2^{-M_2 \epsilon/2} + R^{Q} 2^{-M_2/2} \right) \|f_{m_1, M_1}\|_{L^1} \|g\|_{L^1} .
\end{align*}
Now plugging the above estimate into \eqref{Holder in estimate of E11 case} we obtain
\begin{align*}
    \|E_{11}\|_{L^{1/2}} &\leq C \sum_{M_2=0}^{j} \sum_{M_1=0}^{M_2} \sum_{m_1=1}^{N_{M_1}} 2^{-j \vartheta N} \left( R^{(Q+d)/2+\epsilon} 2^{-M_2 \epsilon/2} + R^{Q} 2^{-M_2/2} \right) \|f_{m_1, M_1}\|_{L^1} \|g\|_{L^1} \\
    &\nonumber \leq C 2^{j \epsilon_1} 2^{-j \vartheta N} 2^{j Q} \|f\|_{L^1} \|g\|_{L^1} \\
    &\nonumber \leq C 2^{-j \epsilon_2} \|f\|_{L^1} \|g\|_{L^1} ,
\end{align*}
for some $\epsilon_2>0$, where in the second last inequality we have used that $j>\lfloor\log_2 R\rfloor$ and we get the last inequality by choosing $N$ sufficiently large enough.

\subsubsection{Estimate of $E_{12}$}
\label{Estimate of E12 case}
As seen previously in \eqref{Decomposition of f in E1} we now decompose 
\begin{align*}
    \mathbb{S} &= \bigcup_{m=1}^{N_{M_1}} B_{m, M_1}^{\varpi, \Im} ,
\end{align*}
such that
\begin{align*}
    B_{m, M_1}^{\varpi, \Im} \subseteq \Big\{z \in \mathbb{S} : \varpi(z, a_{m}) \leq C \tfrac{2^{M_1(1+\gamma)}}{ R}, \ \  |\Im \langle z, e \rangle| \leq C \Big\} ,
\end{align*}
and for $m \neq m'$ and it satisfies $\varpi(a_{m}, a_{m'}) \geq C \tfrac{2^{M_1(1+\gamma)}}{2 R}$.

Now proceeding exactly as in the estimate of $S_{12}$ (see subsection \ref{Estimate of S12 case}) from \eqref{Estimate of S12 before the condition} we get
\begin{align*}
    \|E_{12}\|_{L^{1/2}} &\leq C \left(\sum_{M_2=0}^{j} \sum_{M_1=0}^{M_2} 2^{C \vartheta j} 2^{j \epsilon_2} \left( 2^{M_1 Q/2} R^{-1/2+2\epsilon} + 2^{M_1 d/2} \right) \right)^2 \|f\|_{L^1} \|g\|_{L^1} \\
    &\leq C 2^{2 C \vartheta j} 2^{j \epsilon_3} \Big(2^{j Q} R^{-1+4\epsilon} + 2^{j d} \Big) \|f\|_{L^1} \|g\|_{L^1} \\
    &\nonumber \leq C 2^{j \epsilon_4} 2^{j Q} R^{-1+4\epsilon} \|f\|_{L^1} \|g\|_{L^1} ,
\end{align*}
for some $\epsilon_3, \epsilon_4>0$, provided we choose $\vartheta$ sufficiently small.

\subsubsection{Estimate of $E_{13}$}
\label{Estimate of E13 case}
In this case using H\"older's inequality and arguing as in \eqref{L1 norm estimate for S13 case} and the fact $j>\lfloor\log_2 R\rfloor$ we get
\begin{align*}
    & \|E_{13}\|_{L^{1/2}} \\
    &\leq C \sum_{M_2=0}^{j} \sum_{M_1=0}^{M_2} \sum_{m_1=1}^{N_{M_1}} \sum_{m_2=1}^{N_{M_1}} \| \chi_{\widetilde{B}_{m_1, M_1}^{\varpi, \Im}} (1-\chi_{\widetilde{B}_{m_2, M_1}^{\varpi, \Im}}) \mathcal{B}_{R,j, M_1, M_2}^{>, >}(f_{m_1, M_1} ,g_{m_2, M_1})\|_{L^1} \\
    &\leq C \sum_{M_2=0}^{j} \sum_{M_1=0}^{M_2} \sum_{m_1=1}^{N_{M_1}} \sum_{m_2=1}^{N_{M_1}} 2^{-j \vartheta N} \left( R^{(Q+d)/2+\epsilon} 2^{-M_2 \epsilon/2} + R^{Q} 2^{-M_2/2} \right) \|f_{m_1, M_1}\|_{L^1} \|g_{m_2, M_1}\|_{L^1} \\
    &\leq C 2^{j \epsilon_1} 2^{-j \vartheta N} 2^{j Q} \|f\|_{L^1} \|g\|_{L^1} \\
    &\leq C 2^{-j \epsilon_2} \|f\|_{L^1} \|g\|_{L^1} ,
\end{align*}
for some $\epsilon_1, \epsilon_2>0$ and provided we choose $N>0$ sufficiently large.

\medskip

Combining all the estimates of $E_{11}$, $E_{12}$ and $E_{13}$, for $j>\lfloor\log_2 R\rfloor$ we obtain
\begin{align*}
    \|E_1\|_{L^{1/2}} &\leq C 2^{j \epsilon} 2^{j Q} R^{-1+4\epsilon} \|f\|_{L^1} \|g\|_{L^1} ,
\end{align*}
for some $\epsilon>0$.

\subsubsection{Estimate of $E_{2}$}
Estimate of $E_2$ is similar to the estimate of $S_2$ (see \eqref{Estimate of S2}).

\medskip
Finally, combining the estimates of $E_1$, $E_2$ and $E_3$ from \eqref{Writing Brj as sum of E1 E2 and E3} we obtain
\begin{align}
\label{Final estimate of Brj bigger bigger}
    \|\mathcal{B}_{R,j}^{>, >}(f,g)\|_{L^{1/2}} &\leq C 2^{j \epsilon} (2^{j d} + 2^{j Q} R^{-1+4\epsilon}) \|f\|_{L^1} \|g\|_{L^1} ,
\end{align}
for some $\epsilon>0$.

Since $R>1$ and $\alpha>\alpha(1,1)=d$, using the definition \eqref{Definition of error partfor restricted theorem}, above estimate \eqref{Final estimate of Brj bigger bigger} and \eqref{Triangle inequality for p less than 1 case} we have
\begin{align*}
    \|\mathcal{E}_{R}^{\alpha, >, >}(f,g)\|_{L^{1/2}}^{1/2} &\leq \sum_{j=\lfloor\log_2 R \rfloor +1}^{\infty} 2^{-j\alpha/2} \|\mathcal{B}_{R,j}^{>, >}(f,g)\|_{L^{1/2}}^{1/2} \\
    &\leq C \sum_{j=\lfloor\log_2 R \rfloor +1}^{\infty} 2^{-j\alpha/2} 2^{j \epsilon/2} (2^{j d/2} + 2^{j Q/2} R^{-1/2+2\epsilon}) \|f\|_{L^1}^{1/2} \|g\|_{L^1}^{1/2} \\
    &\leq C \big(R^{-(\alpha-d-\epsilon)/2} + R^{-(\alpha-Q-5\epsilon+1)/2} \big) \|f\|_{L^1}^{1/2} \|g\|_{L^1}^{1/2} \\
    &\leq C \|f\|_{L^1}^{1/2} \|g\|_{L^1}^{1/2} ,
\end{align*}
provided we choose $\epsilon>0$ sufficiently small such that $\alpha-d-5\epsilon>0$.

Now squaring both side of the above inequality, the proof of the estimate \eqref{To prove for Error for second theorem} is completed.

\medskip
This completes the estimate of $\mathcal{B}_{R}^{\alpha}(f^{>},g^{>})$. Now to complete the proof of Theorem \ref{Theorem: Bilinear Bochner-Riesz theorem with restricted f and g}, it remains to estimate $\mathcal{B}_{R}^{\alpha}(f^{<},g^{<})$. This estimate is relatively simpler than the previous estimate of $\mathcal{B}_{R}^{\alpha}(f^{>},g^{>})$, this is because of the Corollary \ref{Corollary: L1 to L2 estimate for diagonal operator}.

\medskip
\noindent \textbf{Estimate of $\mathcal{B}_{R}(f^{<},g^{<})$:}
Once again, let $\Omega_{\kappa_1,R}^{<}$, $\Omega_{\kappa_2,R}^{<}$ be two smooth functions on $\mathbb{R}$ such that
\begin{align*}
    \Omega_{\kappa_1,R}^{<}(\tau_1) = \left\{\begin{array}{ll}
        1, & \quad |\tau_1| \leq \kappa_1 R \\
        0, & \quad |\tau_1| \geq 2\kappa_1 R 
    \end{array}\right. \quad \text{and} \quad 
    \Omega_{\kappa_2,R}^{<}(\tau_2) = \left\{\begin{array}{ll}
        1, & \quad |\tau_2| \leq \kappa_2 R \\
        0, & \quad |\tau_2| \geq 2\kappa_2 R 
    \end{array}\right.
\end{align*}
and defined smoothly elsewhere.

Previously as in \eqref{Rewriting Bochner-Riesz grater different way}, we can write
\begin{align}
\label{Rewriting the definition of Br less less}
    & \mathcal{B}_{R}(f^{<},g^{<})(z) = \sum_{\substack{\ell_1,\ell_1', \ell_2,\ell_2' =0}}^{\infty} \left(1- \tfrac{\sqrt{\lambda_{\ell_1, \ell_1'}} + \sqrt{\lambda_{\ell_2, \ell_2'}}}{R} \right)_{+}^{\alpha} \pi_{\ell_1, \ell_1'} f^{<}(z) \, \pi_{\ell_2, \ell_2'} g^{<}(z) \\
    &\nonumber = \sum_{\substack{\ell_1,\ell_1', \ell_2,\ell_2' =0}}^{\infty} \left(1- \tfrac{\sqrt{\lambda_{\ell_1, \ell_1'}} + \sqrt{\lambda_{\ell_2, \ell_2'}}}{R} \right)_{+}^{\alpha} \Omega_{\kappa_1,R}^{<}(|\ell_1-\ell_1'|) \Omega_{\kappa_2,R}^{<}(|\ell_2-\ell_2'|) \pi_{\ell_1, \ell_1'} f(z) \,  \pi_{\ell_2, \ell_2'} g(z) \\
    &\nonumber =: \mathcal{B}^{<, <}_{R}(f,g)(z) .
\end{align}
Analogous to the decomposition of \eqref{Decomposition of Bochner-Riesz greater} here we have
\begin{align*}
    \mathcal{B}_{R}^{\alpha, <, <}(f,g)(z) &= \sum_{j = 0} ^{\lfloor\log_2 R \rfloor} 2^{-j\alpha} \mathcal{B}_{R,j}^{<, <}(f,g)(z) + \mathcal{E}_{R}^{\alpha, <, <}(f,g)(z) ,
\end{align*}
where
\begin{align}
\label{Definition of error part for less less}
    \mathcal{E}_{R}^{\alpha, <, <}(f,g)(z) &= \sum_{j = \lfloor\log_2 R \rfloor + 1}^{\infty} 2^{-j\alpha} \mathcal{B}_{R,j}^{<, <}(f,g)(z) ,
\end{align}
and
\begin{align*}
    \mathcal{B}_{R,j}^{<, <}(f,g)(z) &= \sum_{\substack{\ell_1,\ell_1', \ell_2,\ell_2' =0}}^{\infty} \Psi_j\left( \tfrac{\sqrt{\lambda_{\ell_1, \ell_1'}}}{R}, \tfrac{\sqrt{\lambda_{\ell_2, \ell_2'}}}{R} \right) \Omega_{\kappa_1,R}^{<}(|\ell_1-\ell_1'|) \Omega_{\kappa_2,R}^{<}(|\ell_2-\ell_2'|) \pi_{\ell_1, \ell_1'} f(z) \,  \pi_{\ell_2, \ell_2'} g(z) .
\end{align*}
Therefore, it is enough to show that, whenever $\alpha>d$ we have
\begin{align}
\label{To prove for Error for second theorem in less less}
    \|\mathcal{E}_{R}^{\alpha, <, <}(f,g)\|_{L^{1/2}} &\leq C \|f\|_{L^{1}} \|g\|_{L^{1}} ,
\end{align}
and for $0 \leq j \leq \lfloor\log_2 R \rfloor$ we have
\begin{align}
\label{To prove main for second theorem in less less}
    \|\chi_{B_{R}^j} \mathcal{B}_{R,j}^{<, <}(f_{R,j},g_{R,j})\|_{L^{1/2}} &\leq C 2^{j d+j\epsilon} \|f_{R,j}\|_{L^{1}} \|g_{R,j}\|_{L^{1}} \quad \text{where} \quad \supp{f_{R,j}}, \supp{g_{R,j}} \subseteq B_{R}^j ,
\end{align}
and for some $\epsilon>0$ sufficiently small.

\medskip
\noindent \textbf{Estimate of (\ref{To prove main for second theorem in less less}):}
Similar to \eqref{Fourier series decomposition in restricted} here we write
\begin{align*}
    \chi_{B_{R}^j}(z) \mathcal{B}_{R,j}^{<, <}(f_{R,j},g_{R,j})(z) & = \chi_{B_{R}^j}(z) \sum_{l \in \mathbb{Z}} \phi_{j,l,R}^{<}(\sqrt{\mathcal{L}}, |T|)f_{R,j}(z)\, \psi_{l,R}^{<}(\sqrt{\mathcal{L}}, |T|)g_{R,j}(z) ,
\end{align*}
where $\phi_{j,l,R}^{<}(\sqrt{\mathcal{L}}, |T|)$ and $\psi_{l,R}^{<}(\sqrt{\mathcal{L}}, |T|)$ are defined analogously to those in \eqref{Fourier series decomposition in restricted}.

Using H\"older's inequality twice we obtain
\begin{align*}
    \|\chi_{B_{R}^j} \mathcal{B}_{R,j}^{<, <}(f_{R,j},g_{R,j})\|_{L^{1/2}} &\leq C \left(\frac{2^{j(1+\gamma)}}{ R} \right)^{Q} \sum_{l \in \mathbb{Z}} \|\phi_{j,l,R}^{<}(\sqrt{\mathcal{L}}, |T|)f_{R,j}\|_{L^2} \|\psi_{l,R}^{<}(\sqrt{\mathcal{L}}, |T|)g_{R,j}\|_{L^2} .
\end{align*}
Note that (see \eqref{Definition of the operator Ls})
\begin{align*}
    \phi_{j,l,R}^{<}(\sqrt{\mathcal{L}}, |T|)f_{R,j}(z) &= \sum_{\substack{\ell_1, \ell_1'=0 \\ |\ell_1- \ell_1'| \leq 2\kappa_1 R}}^{\infty} \phi_{j,l}(R^{-1}\sqrt{\lambda_{\ell_1, \ell_1'}})\, \pi_{\ell_1, \ell_1'}f_{R,j}(z) =: \chi_1(|T|) \phi_{j,l}(R^{-1}\sqrt{\mathcal{L}})f(z) .
\end{align*}
Hence from Corollary \ref{Corollary: L1 to L2 estimate for diagonal operator} with $s=1$ we get
\begin{align}
\label{L1 L2 estimate of near diagonal part}
    \|\phi_{j,l,R}^{<}(\sqrt{\mathcal{L}}, |T|)f_{R,j}\|_{L^2} &\leq C R^{d/2} \|\phi_{j,l}\|_{L^{\infty}} \|f_{R,j}\|_{L^1} .
\end{align}
Similarly we also have
\begin{align}
\label{L1 L2 estimate near diagonal for psi}
    \|\psi_{l,R}^{<}(\sqrt{\mathcal{L}}, |T|)g_{R,j}\|_{L^2} &\leq C R^{d/2} \|\psi_{l}\|_{L^{\infty}} \|g_{R,j}\|_{L^1} .
\end{align}
Therefore combining the above two estimate and \eqref{Convergence of sum of l using phi} yields
\begin{align*}
    \|\chi_{B_{R}^j} \mathcal{B}_{R,j}^{<, <}(f_{R,j},g_{R,j})\|_{L^{1/2}} &\leq C \left(\frac{2^{j(1+\gamma)}}{ R} \right)^{Q} \sum_{l \in \mathbb{Z}} R^{d/2} \|\phi_{j,l}\|_{L^{\infty}} \|f_{R,j}\|_{L^1} R^{d/2} \|\psi_{l}\|_{L^{\infty}} \|g_{R,j}\|_{L^1} \\
    &\leq C 2^{j \epsilon_1} 2^{j Q} R^{-1} \|f_{R,j}\|_{L^1} \|g_{R,j}\|_{L^1} ,
\end{align*}
for some $\epsilon_1>0$.

Since $0 \leq j \leq \lfloor\log_2 R\rfloor$ we obtain
\begin{align*}
    \|\chi_{B_{R}^j} \mathcal{B}_{R,j}^{<, <}(f_{R,j},g_{R,j})\|_{L^{1/2}} &\leq C 2^{j d + j \epsilon_1} \|f_{R,j}\|_{L^1} \|g_{R,j}\|_{L^1} .
\end{align*}
This completes the estimate of \eqref{To prove main for second theorem in less less}.

\medskip
\noindent \textbf{Estimate of (\ref{To prove for Error for second theorem in less less}):}
Since $\mathbb{S}$ is compact, using H\"older's inequality and analogous estimate as of \eqref{L1 L2 estimate of near diagonal part}, \eqref{L1 L2 estimate near diagonal for psi} we obtain
\begin{align*}
    \|\mathcal{B}_{R,j}^{<, <}(f,g)\|_{L^{1/2}} &\leq C \sum_{l \in \mathbb{Z}} \|\phi_{j,l,R}^{<}(\sqrt{\mathcal{L}}, |T|)f\|_{L^2} \|\psi_{l,R}^{<}(\sqrt{\mathcal{L}}, |T|)g\|_{L^2} \\
    &\leq C R^d \left(\sum_{l \in \mathbb{Z}} \|\phi_{j,l}\|_{L^{\infty}} \right) \|f\|_{L^1} \|g\|_{L^{1}} \\
    &\leq C 2^{j d + j \epsilon}  \|f\|_{L^1} \|g\|_{L^{1}} ,
\end{align*}
for some $\epsilon>0$, where in the last inequality we have used $j>\lfloor\log_2 R\rfloor$ and \eqref{Convergence of sum of l using phi}.

Consequently, in view of the above estimate and from \eqref{Definition of error part for less less} we get the required estimate \eqref{To prove for Error for second theorem in less less} whenever $\alpha>d$, and hence with this the estimate of $\mathcal{B}_{R}^{\alpha}(f^{<},g^{<})$ is also completed.


\section{Proof of Theorem \ref{Theorem: Bilinear near diagonal case}}
\label{Section: Proof of third main theorem}
In order to prove Theorem \ref{Theorem: Bilinear near diagonal case}, due to bilinear interpolation \cite[Section 4.3]{Bernicot_Grafakos_Song_Yan_Bilinear_Bochner_Riesz_2015}, it is enough to prove Theorem \ref{Theorem: Bilinear near diagonal case} for $(p_1, p_2, p)=(1,1,1/2)$, $(1,2,2/3)$, $(2,2,1)$, $(\infty, \infty, \infty)$, $(2, \infty, 2)$ and $(1, \infty, 1)$. Here we will provide the proof for the points $(p_1, p_2, p)=(1,1,1/2)$, $(\infty, \infty, \infty)$ and $(2, \infty, 2)$. Note that the estimate at the point $(2,2,1)$ can be seen in the proof of the Theorem \ref{Theorem: Bilinear Bochner-Riesz Main theorem}. The estimates at the remaining two points $(p_1, p_2, p)=(1,2,2/3)$ and $(1, \infty, 1)$ can be obtain by following the proof at the point $(1,1,1/2)$.

\subsection{Proof at \texorpdfstring{$(p_1, p_2, p)=(1,1,1/2)$}{}}
Observe that
\begin{align*}
    \mathcal{B}_{R, s}^{\alpha}(f,g)(z) &= \mathcal{B}^{<, <}_{R^s}(f,g)(z) ,
\end{align*}
where $\mathcal{B}^{<, <}_{R^s}$ is as defined in \eqref{Rewriting the definition of Br less less} with $R$ replaced by $R^s$ and $\kappa_1=\kappa_2=\kappa$.

Consequently, proceeding similarly as in the estimate of $\mathcal{B}^{<, <}_{R}(f,g)$ (see bellow of \eqref{Rewriting the definition of Br less less}), it is enough to show, whenever $\alpha>\alpha(1,1)=d-1+s$ we have
\begin{align}
\label{To prove for Error for second theorem in less less near zero}
    \|\mathcal{E}_{R^s}^{\alpha, <, <}(f,g)\|_{L^{1/2}} &\leq C \|f\|_{L^{1}} \|g\|_{L^{1}} ,
\end{align}
and for $0 \leq j \leq \lfloor\log_2 R \rfloor$ and $\epsilon>0$ sufficiently small we have
\begin{align}
\label{To prove main for second theorem in less less near zero}
    \|\chi_{B_{R}^j} \mathcal{B}_{R^s,j}^{<, <}(f_{R,j},g_{R,j})\|_{L^{1/2}} &\leq C 2^{j (d-1+s)+j\epsilon} \|f_{R,j}\|_{L^{1}} \|g_{R,j}\|_{L^{1}} \quad \text{where} \quad \supp{f_{R,j}}, \supp{g_{R,j}} \subseteq B_{R}^j ,
\end{align}
and
\begin{align}
\label{Error part definition for near zero}
    \mathcal{E}_{R^s}^{\alpha, <, <}(f,g)(z) &= \sum_{j = \lfloor\log_2 R\rfloor+1}^{\infty} 2^{-j\alpha} \mathcal{B}_{R^s,j}^{<, <}(f,g)(z) .
\end{align}
and 
\begin{align}
\label{Dyadic part definition near zero}
    \mathcal{B}_{R^s,j}^{<, <}(f,g)(z) &= \sum_{\substack{\ell_1,\ell_1', \ell_2,\ell_2' =0\\ |\ell_1-\ell_1'|\leq \kappa R^s, |\ell_2-\ell_2'|\leq \kappa R^s }}^{\infty} \Psi_j\left( \tfrac{\sqrt{\lambda_{\ell_1, \ell_1'}}}{R}, \tfrac{\sqrt{\lambda_{\ell_2, \ell_2'}}}{R} \right) \pi_{\ell_1, \ell_1'} f(z) \,  \pi_{\ell_2, \ell_2'} g(z) .
\end{align}

Following the same line of argument as in the proof of estimates \eqref{To prove for Error for second theorem in less less} and \eqref{To prove main for second theorem in less less} we get required above two estimates \eqref{To prove for Error for second theorem in less less near zero} and \eqref{To prove main for second theorem in less less near zero} and completes the proof at the point $(p_1, p_2, p)=(1,1,1/2)$.

\subsection{Proof at \texorpdfstring{$(p_1, p_2, p)=(\infty,\infty,\infty)$}{}}
In this case it is enough to show, whenever $\alpha>d+s-3/2$ we have
\begin{align}
\label{To prove error part in infinity with s}
    \|\mathcal{E}_{R^s}^{\alpha, <, <}(f,g)\|_{L^{\infty}(\mathbb{S})} &\leq C \|f\|_{L^{\infty}(\mathbb{S})} \|g\|_{L^{\infty}(\mathbb{S})} ,
\end{align}
and for $0\leq j \leq \lfloor\log_2 R \rfloor$ we have
\begin{align}
\label{To prove main at two points with s operator}
    \|\chi_{B_{R}^j} \mathcal{B}_{R^s,j}^{<, <}(f_{R,j},g_{R,j})\|_{L^{\infty}(\mathbb{S})} &\leq C 2^{j (d+s-3/2)+j\epsilon} \|f_{R,j}\|_{L^{\infty}(\mathbb{S})} \|g_{R,j}\|_{L^{\infty}(\mathbb{S})} ,
\end{align}
for some $\epsilon>0$ sufficiently small, where $\mathcal{E}_{R^s}^{\alpha, <, <}$ and $\mathcal{B}_{R^s,j}^{<, <}$ are defined as in \eqref{Error part definition for near zero} and \eqref{Dyadic part definition near zero}.

We will start with proof of the estimate \eqref{To prove main at two points with s operator}.

\subsubsection{Proof of \eqref{To prove main at two points with s operator}}
Applying H\"older's inequality we get
\begin{align}
\label{Use of twice Holder in infinity with s}
    |\chi_{B_{R}^j}(z) \mathcal{B}_{R^s,j}^{<, <}(f_{R,j},g_{R,j})(z)| &\leq \left(\int_{\mathbb{S}} \int_{\mathbb{S}}  \varpi(z,w)^{2\beta_1} \varpi(z,u)^{2\beta_2} | \mathcal{K}_{R^s,j}^{<, <}(z, w, u)|^2 \,d\sigma(w) \, \,d\sigma(u) \right)^{1/2} \\
    &\nonumber \hspace{1cm} \times \left(\int_{\mathbb{S}} \frac{|f_{R,j}(w)|^2}{\varpi(z,w)^{2\beta_1}} \,d\sigma(w) \right)^{1/2} \left(\int_{\mathbb{S}} \frac{|g_{R,j}(u)|^2}{\varpi(z,u)^{2\beta_2}} \,d\sigma(u) \right)^{1/2} ,
\end{align}
where $\mathcal{K}_{R^s,j}^{<, <}$ is the integral kernel corresponding to the operator $\mathcal{B}_{R^s,j}^{<, <}$ given by (see also \eqref{Definition of bilinear kernel with s operator})
\begin{align*}
    \mathcal{K}_{R^s,j}^{<, <}(z, w, u) &= \sum_{\substack{\ell_1,\ell_1', \ell_2,\ell_2' =0\\ |\ell_1-\ell_1'|\leq \kappa R^s, |\ell_2-\ell_2'|\leq \kappa R^s }}^{\infty} \Psi_j\left(\tfrac{\sqrt{\lambda_{\ell_1, \ell_1'}}}{R}, \tfrac{\sqrt{\lambda_{\ell_2, \ell_2'}}}{R} \right) \overline{Y_z^{\ell_1,\ell_1'}(w)}\,  \overline{Y_z^{\ell_2,\ell_2'}(u)} \\
    &= \mathcal{K}_{\Psi_j(\sqrt{\mathcal{L}_1}, \sqrt{\mathcal{L}_2})}^{bi, \chi_s}(z, w, u) .
\end{align*}
From Proposition \ref{Proposition: Bilinear weighted Plancherel with weight and s operator} and \eqref{Discrete norm computation} we have
\begin{align}
\label{Weighted Plancherel in psi j part}
    & \left(\int_{\mathbb{S}} \int_{\mathbb{S}}  \varpi(z,w)^{2\beta_1} \varpi(z,u)^{2\beta_2} | \mathcal{K}_{R^s,j}^{<, <}(z, w, u)|^2 \,d\sigma(w) \, \,d\sigma(u) \right)^{1/2} \\
    &\nonumber \leq C R^{Q-1+s-\beta_1-\beta_2} \|\Psi_j(R\cdot, R \cdot)\|_{\lceil R \rceil, \lceil R \rceil, 2} \\
    &\nonumber \leq C R^{Q-1+s-\beta_1-\beta_2} \max\{2^{-j/2}, R^{-1/2} \} .
\end{align}
Consequently, since $0 \leq j \leq \lfloor\log_2 R \rfloor$ and $R>1$, plugging the above estimate \eqref{Weighted Plancherel in psi j part} and \eqref{Integral of wight with f case for main}, \eqref{Integral of weight with g case for main} into \eqref{Use of twice Holder in infinity with s} for $s \in [0,1]$ and $0 \leq \beta_1, \beta_2<1/2$ we obtain
\begin{align*}
    |\chi_{B_{R}^j}(z) \mathcal{B}_{R^s,j}^{<, <}(f_{R,j},g_{R,j})(z)| &\leq C R^{Q-\beta_1-\beta_2-1+s} 2^{-j/2} \left(\tfrac{2^{j(1+\gamma)}}{ R} \right)^{Q-\beta_1-\beta_2} \|f_{R, j}\|_{L^{\infty}} \|g_{R, j}\|_{L^{\infty}} \\
    &\leq C 2^{-j(1-s)} 2^{-j/2} 2^{j \epsilon} 2^{j d} \|f_{R, j}\|_{L^{\infty}} \|g_{R, j}\|_{L^{\infty}} \\
    &\leq C 2^{j (d+s-3/2)+j\epsilon} \|f_{R,j}\|_{L^{\infty}} \|g_{R,j}\|_{L^{\infty}} ,
\end{align*}
for some $\epsilon>0$.

From the above the proof of \eqref{To prove main at two points with s operator} follows.

\subsubsection{Proof of \eqref{To prove error part in infinity with s}} Observe that, we have $\varpi(z, w) \geq 1$. Applying H\"older's inequality and arguing similarly as in \eqref{Use of twice Holder in infinity with s}, \eqref{Weighted Plancherel in psi j part} for $j> \lfloor\log_2 R \rfloor$ and $0 \leq \beta_1, \beta_2<1/2$ yields
\begin{align*}
    |\mathcal{B}_{R^s,j}^{<, <}(f,g)(z)| &\leq \left(\int_{\mathbb{S}} \int_{\mathbb{S}}  \varpi(z,w)^{2\beta_1} \varpi(z,u)^{2\beta_2} | \mathcal{K}_{R^s,j}^{<, <}(z, w, u)|^2 \,d\sigma(w) \, \,d\sigma(u) \right)^{1/2} \\
    &\nonumber \hspace{2cm} \left(\int_{\mathbb{S}} |f(w)|^2 \,d\sigma(w) \right)^{1/2} \left(\int_{\mathbb{S}} |g(u)|^2 \,d\sigma(u) \right)^{1/2} \\
    &\leq C R^{Q-1+s-\beta_1-\beta_2} R^{-\mu} 2^{j\mu-j/2} \|f\|_{L^{\infty}} \|g\|_{L^{\infty}} .
\end{align*}
Consequently, similar to \eqref{Final estimate second case infinity}, since $R>1$ and taking $\alpha>d+s-3/2$ we obtain
\begin{align*}
    \|\mathcal{E}_{R^s}^{\alpha, <, <}(f,g)\|_{L^{\infty}} &\leq C R^{Q-\beta_1-\beta_2+s-1} R^{-\mu} \|f\|_{L^{\infty}} \|g\|_{L^{\infty}} \sum_{j=\lfloor\log_2 R \rfloor+1}^{\infty} 2^{-j \alpha} 2^{j(\mu-1/2)} \\
    &\leq C R^{-\alpha+(Q-\beta_1-\beta_2+s-3/2)} \|f\|_{L^{\infty}} \|g\|_{L^{\infty}} \\
    &\leq C \|f\|_{L^{\infty}} \|g\|_{L^{\infty}} ,
\end{align*}
provided we choose $\beta_1 +\beta_2$ very close to $1$.

\subsection{Proof at \texorpdfstring{$(p_1, p_2, p)=(2,\infty,2)$}{}}
The proof of Theorem \ref{Theorem: Bilinear near diagonal case} at $(p_1, p_2, p)=(2,\infty,2)$ follows closely to that of Theorem \ref{Theorem: Bilinear Bochner-Riesz Main theorem} at $(p_1, p_2, p)=(2,\infty,2)$ (see Subsection \ref{Subsection: Proof of main theorem at 2 infinity}) with obvious modification. If fact, in order get the desired analogous estimate as of \eqref{Calculating weighted Plancherel} and \eqref{Weighted Plancherel calculation for j bigger}, instead of using Corollary \ref{Corollary: Linear weighted Plancherel}, here we have to use Proposition \ref{Proposition: Improved weighted Plancherel for near} and then the rest of the proof follows similarly.

\section*{Acknowledgments}
The first and third author would like to acknowledge the support provided by the FIST grant (SR/FST/MS-II/2019/51) at IISER Kolkata provided by the Government of India. S. Bagchi was partially supported by the Anusandhan National Research Foundation (ANRF), India, under the research project ANRF/ARG/2025/003732/MS. J. Singh was supported from Prime Minister's Research Fellowship (PMRF), and institute post-doctoral fellowship at Indian Institute of Science Education and Research Mohali. M. N.~Vempati's research was supported, in part, by the Louisiana Experimental Program to Stimulate Competitive Research (EPSCoR), funded by the National Science Foundation and the Board of Regents Support Fund award number OIA-2437963.


\providecommand{\bysame}{\leavevmode\hbox to3em{\hrulefill}\thinspace}
\providecommand{\MR}{\relax\ifhmode\unskip\space\fi MR }
\providecommand{\MRhref}[2]{%
  \href{http://www.ams.org/mathscinet-getitem?mr=#1}{#2}
}
\providecommand{\href}[2]{#2}

\end{document}